\documentclass[10pt]{amsart}
\usepackage{amssymb}

\usepackage{graphicx}
\usepackage{xcolor}
\usepackage{enumerate}
\usepackage{tikz-cd}
\usepackage{amsxtra}

\usepackage[T1]{fontenc}

\def\chk#1{#1^{\smash{\scalebox{.7}[1.4]{\rotatebox{90}{\textnormal{\guilsinglleft}}}}}}

\newcommand{\res}{\upharpoonright}

\newcommand{\un}{\underline} 
\newcommand{\ov}{\overline}

\newcommand{\ec}{\overset{c}{=}}
\newcommand{\eb}{\overset{b}{=}}
\newcommand{\ac}{\overset{c}{\to}}
\newcommand{\ab}{\overset{b}{\to}}
\newcommand{\eq}{\overset{q}{=}}
\newcommand{\aq}{\overset{q}{\to}}

\theoremstyle{plain}
\newtheorem{theorem}{Theorem}[section]
\newtheorem{corollary}[theorem]{Corollary}
\newtheorem{lemma}[theorem]{Lemma}
\newtheorem{proposition}[theorem]{Proposition}

\newtheorem*{main}{Main Lemma}

\theoremstyle{remark}

\newtheorem{question}{Question}[section]
\newtheorem*{claim*}{Claim}
\newtheorem{claim}{Claim}

\numberwithin{equation}{section} 
\numberwithin{figure}{section}

\usepackage{chngcntr}

\usepackage{tikz, tkz-euclide}
\usetikzlibrary{arrows, calc,cd}

\pgfdeclareplotmark{||||}
{%
  \pgfpathmoveto{\pgfqpoint{-4.5\pgflinewidth}{\pgfplotmarksize}}
  \pgfpathlineto{\pgfqpoint{-4.5\pgflinewidth}{-\pgfplotmarksize}}
  \pgfpathmoveto{\pgfqpoint{-1.5\pgflinewidth}{\pgfplotmarksize}}
  \pgfpathlineto{\pgfqpoint{-1.5\pgflinewidth}{-\pgfplotmarksize}}
  \pgfpathmoveto{\pgfqpoint{1.5\pgflinewidth}{\pgfplotmarksize}}
  \pgfpathlineto{\pgfqpoint{1.5\pgflinewidth}{-\pgfplotmarksize}}
  \pgfpathmoveto{\pgfqpoint{4.5\pgflinewidth}{\pgfplotmarksize}}
  \pgfpathlineto{\pgfqpoint{4.5\pgflinewidth}{-\pgfplotmarksize}}
  \pgfusepathqstroke
}

\def\side{3}
\def\lwa{1.2pt}

\def\crb{2.0pt}
\newcommand\Square[1]{+(-#1,-#1) rectangle +(#1,#1)}

\begin{document}

\title[Complexes, moves, and amalgamation]{Simplicial complexes, stellar moves,\\ and projective amalgamation}

\author{S\l{}awomir Solecki}
\address{Department of Mathematics, 
Cornell University, Ithaca, NY 14853}
\email{ssolecki@cornell.edu}

\begin{abstract}
We explore connections between stellar moves on simplicial complexes (these are fundamental operations of combinatorial 
topology) and projective Fra{\"i}ss{\'e} limits (this is a model theoretic construction with topological applications). 

We identify a class of simplicial maps that arise from the stellar moves of welding and subdividing. 
We call these maps weld-division maps. The core of the paper is the proof that the category of 
weld-division maps fulfills the projective amalgamation property. This gives an example of an amalgamation class that substantially differs from known classes. 

The weld-division amalgamation class naturally gives rise to a projective Fra{\"i}ss{\'e} class. 
We compute the canonical limit of this projective Fra{\"i}ss{\'e} class and its canonical quotient space. 
This computation gives a combinatorial description of the geometric realization of a simplicial complex and 
an example of a combinatorially defined projective Fra{\"i}ss{\'e} class whose canonical quotient space has topological dimension strictly bigger than $1$. 

The method of proof of the amalgamation theorem is new. It is not geometric or topological, but rather it consists of combinatorial calculations performed on finite sequences of finite sets and functions among such sequences. Set theoretic nature of the entries of the sequences is crucial to the arguments. 
\end{abstract}

\thanks{Research supported by NSF grant DMS-2246873.}




\maketitle

\setcounter{tocdepth}{1}

\tableofcontents{}

\newpage

\part{Introduction}

The paper interacts with three areas---amalgamation classes, projective Fra{\"i}ss{\'e} limits, and stellar moves of combinatorial topology. It gives a new amalgamation class that is quite distinct from the amalgamation classes studied so far. Through projective Fra{\"i}ss{\'e} limits, it gives a new, combinatorial presentation of the geometric realization of a simplicial complex and a combinatorially defined projective Fra{\"i}ss{\'e} class with the canonical quotient having topological dimension bigger than $1$. Finally, it develops a new method of dealing with stellar moves on simplicial complexes and simplicial maps among them.

The heart of the paper is a proof of the amalgamation property for a natural class of simplicial maps acting among simplicial complexes. 
In the definition of the amalgamation class, we take as our departure point the two operations, known as stellar moves, on simplicial complexes that are fundamental to combinatorial topology---stellar subdivision and, its inverse, welding. We start with decoupling the two operations: 
\begin{enumerate} 
\item[---] we use stellar subdivision as a generating procedure for our amalgamation class; it produces new simplicial complexes and, together with composition, new simplicial maps; 

\item[---] we use welding to define the base family of simplicial maps in our amalgamation class; simplicial maps in the class are produced by closing of the base family under stellar subdivision and composition.
\end{enumerate} 
We prove the amalgamation theorem for the class of simplicial maps sketched out above. Our results reveal that the class is small enough to have strong combinatorial  properties and large enough to remember the topology of simplicial complexes.

Now, we outline the content of the paper in some more detail. 

{\bf Stellar moves}, that is, stellar subdivision and welding, provide a combinatorial method of modifying simplicial complexes while retaining their geometric and topological structures. 
They go back to the papers of Alexander \cite{Al} and Newman \cite{Ne}, \cite{Ne2}. Stellar moves have been shown to be ``geometrically complete'' in various senses
in \cite{Ne}, \cite{Al}, \cite{Wl}, \cite{AP}; see \cite[Theorem~4.5]{Li} for an exposition of the theorem from \cite{Ne} and \cite{Al} and see also \cite{LR} for a related result.
Stellar moves have formed the basis of combinatorial topology since the publication of the three papers by Alexander and Newman mentioned above; see \cite{Gl} and \cite{Li}. 
We recall the fundamental notions of the theory of stellar moves in Appendix~\ref{S:Div}.

{\bf Amalgamation classes}, that is, classes of structures that fulfill the amalgamation property, are a well-researched area of combinatorics with firm connections to the study of homogeneous structure and Ramsey theory. A number of amalgamation classes have been described, both of the direct and projective kinds, and classifications of amalgamation classes in prescribed contexts have been achieved. The reader may consult \cite{Ch}, \cite{HN}, \cite{La}, \cite{Ne} and the papers cited in the paragraph below for more background information. 

{\bf Projective Fra{\"i}ss{\'e} theory} is a method of producing ``generic'' compact topological spaces from classes of finite 
combinatorial objects by taking canonical projective limits and quotients. 
The method makes it possible to approach topological questions using combinatorial arguments. 
Projective Fra{\"i}ss{\'e} theory was developed in \cite{IrSo}.
It builds on and extends model theoretic ideas coming from Fra{\"i}ss{\'e} \cite{Fra}, see also \cite[Section~7.1]{Hod}. 
This approach was applied in various topological situations; see \cite{BK}, 
\cite{BK2}, \cite{BC0}, \cite{BC}, \cite{Ca}, \cite{CKR}, \cite{CKRY}, \cite{CK}, \cite{Kw0},  
\cite{PS}. We recall the basic notions of projective Fra{\"i}ss{\'e} theory in Appendix~\ref{Su:prfr}.

In this paper, we start with an arbitrary (abstract) simplicial complex $\mathbf A$. We consider the class $\langle {\mathbf A}\rangle$ 
consisting of simplicial complexes produced from $\mathbf A$ by iterative application of stellar subdivisions. These are our objects. 
Next, we define a class of simplicial maps among complexes in $\langle {\mathbf A}\rangle$, which we name {\bf weld-division maps}. 
These are our morphisms. Now, the complexes in $\langle {\mathbf A}\rangle$, as objects, and weld-division maps among them, as morphisms, 
form a natural category ${\mathcal D}({\mathbf A})$ associated with the simplicial complex $\mathbf A$.

Since the class of weld-division maps is new, we comment briefly on the way it is defined. First, we introduce the operation of stellar subdivision of a simplicial map that is parallel to the notion of stellar subdivision of a simplicial complex; this is done in Section~\ref{Su:addsu}. Weld-division maps are then defined as follows; for details see Section~\ref{S:divmap} and Appendix~\ref{A:isome}. The  basic building blocks in this definition are what we call {\bf weld maps} between simplicial complexes, which are inverses of stellar subdivisions of simplicial complexes and are a refinement of the operation of welding. Weld-division maps are obtained from weld maps 
by closing them under composition and stellar subdivision of simplicial maps.

We investigate properties of the category ${\mathcal D}({\mathbf A})$, which amounts to 
a combinatorial study of weld-division maps. Our principal result, Theorem~\ref{T:frai}, 
asserts that the class of weld-division maps has an amalgamation property---the {\bf projective amalgamation property}. That is, given two weld-division maps $f'\colon B\to A$ and $g'\colon C\to A$, where $A, B, C$ are in $\langle {\mathbf A}\rangle$, there exist weld-division maps $f\colon D\to B$ and $g\colon D\to C$, for some $D$ in $\langle{\mathbf A}\rangle$, such that 
\[
f'\circ f= g'\circ g, 
\]
or, in the form of a commuting diagram,  
\begin{figure}[htb]
  \centering
  \begin{tikzcd}
    & B \ar[dl, "f'"'] & \\
    A & & D \ar[ul, dashed, "f"'] \ar[dl, dashed, "g"] \\
    & C \ar[ul, "g'"] &
  \end{tikzcd}
  \label{fig:amalgamation}
\end{figure}

We deduce from this result that the category 
${\mathcal D}({\mathbf A})$ forms a {\bf transitive projective Fra{\"i}ss{\'e} class}; this is done in Theorem~\ref{C:trprc}. 
As such, it is subject to the general approach described 
in Appendix~\ref{Su:prfr}. In particular, it has a canonical limit, called the {\bf projective Fra{\"i}ss{\'e} limit}. 
From the limit, we extract, again following the general theory, 
a compact zero-dimensional metric space $\mathbb A$ with 
a compact equivalence relation $R^{\mathbb A}$ on it. Now, we are in a position to form 
the quotient space  ${\mathbb A}/R^{\mathbb A}$---the {\bf canonical quotient space of ${\mathcal D}({\mathbf A})$}. A natural question arises of topologically identifying this space. 
In Theorem~\ref{T:prli}, we prove that the canonical quotient space is homeomorphic to the geometric realization 
of the simplicial complex $\mathbf A$; a definition of geometric realization is recalled in Appendix~\ref{S:Div}. 
Two aspects of the above results may be worth highlighting.
First, this is the first case of the canonical quotient space of a combinatorially defined projective Fra{\"i}ss{\'e} class having topological dimension strictly greater than $1$. Second, the results give a purely combinatorial definition of the geometric realization of an abstract simplicial complex.

A comment about the method of proof is in order. The interest in the class ${\mathcal D}({\mathbf A})$ 
comes, to a large degree, from the geometric, multidimensional nature of its objects and morphisms. However, 
in proving the amalgamation property for ${\mathcal D}({\mathbf A})$, 
we found it impossible to employ geometric or topological methods.
Consequently, the proof of this property is rather unexpected. The high dimensional geometric problems are handled by forming a calculus of finite sequences of finite sets. 
To be a bit more specific, we note that finite sets are fundamental to our considerations. They are the building blocks of simplicial complexes. 
They are also operators on simplicial complexes,  that is, a finite set applied to a complex subdivides it. 
Since we consider iterative stellar subdivisions, that is, subdivisions implemented by finite sequences of finite sets, and simplicial maps among so subdivided complexes, we are naturally led to a study of finite sequences of finite sets and appropriately defined functions among such sequences. 
Developing these ideas, we carry out the main arguments by performing computations and combinatorial manipulations on finite sequences of finite sets and functions among them. 
Curiously, crucial to these considerations is the set theoretic character of the entries of the sequences---in particular, the cumulative hierarchy of sets, the relation $\in$ of membership and its well foundedness, boolean operations, formation of singletons, etc. The proof of the amalgamation property is long and rather involved---it occupies in Parts~\ref{P:frpr} and \ref{P:repr} and Appendices~\ref{A:faces} and \ref{A:isome}---but it is self-contained.


An earlier attempt to construct a combinatorial projective Fra{\"i}ss{\'e} category whose canonical quotient space is equal to the geometric realization of a given simplicial complex was made in the circulated note \cite{PS2}. This attempt was not successful as the proof of the projective amalgamation theorem for the category from that note contains a gap. (In the proof of  \cite[Theorem~3.2]{PS2}, the transition from $h$ to $\beta h$  is unjustified.)
So, the projective Fra{\"i}ss{\'e} limit in \cite{PS2} is not defined. The approach in the current paper differs from that in \cite{PS2} in two major ways. 
First, the category studied here is broader than the one in \cite{PS2}. 
The second difference is a complete change in the method of proof with the leading role played 
now by the combinatorial calculus of sequences of sets mentioned above.

\bigskip

\noindent {\bf Acknowledgement.} I thank Aristotelis Panagiotopoulos for our stimulating discussions during the time of our work on the note \cite{PS2}. 
I am grateful to Jarik Ne{\v s}et{\v r}il, Honza Hubi{\v c}ka, and Mat{\v e}j Kone{\v c}n{\'y} for making it possible for me to present the subject matter of 
this paper at the Midsummer Combinatorial  Workshop in August 2024, which gave me a chance to consolidate my thinking on the subject. 
Finally, I thank Kostya Slutsky for preparing the very nice drawings for me.

\newpage

\part{Background and statements of the main results}\label{Pa:back}

\section{Subdivided simplicial complexes and maps among them}

\subsection{Stellar subdivision} 

We recall the definition of stellar subdivision of complexes. 
The formulation below of the definition is from \cite[Section~2.1.5]{Koz}. It is equivalent to the classical formulation, as given, for example, in \cite{Li}. 

Let $A$ be a complex and let $s$ be a non-empty finite set. We define the {\bf stellar subdivison $sA$ of $A$ by $s$}. 
Fix a new vertex $\chk{s}$. We declare $sA$ to consist of the following sets
\begin{equation}\label{E:divin}
\begin{cases}
y\cup \{ \chk{s}\}, &\hbox{ if }s\not\subseteq y \hbox{ and } s\cup y\in A;\\
y, &\hbox{ if } s\not\subseteq y\hbox{ and } y\in A. 
\end{cases}
\end{equation} 
It is easy to check that so defined family $sA$ of sets is a complex. 
We will sometimes refer to faces of the form $y\cup \{ \chk{s}\}$ as the {\bf new faces of $sA$} and to faces of the form $y$ as the {\bf old faces of $sA$}. 

One can view forming $sA$ as follows. We keep from $A$ all the faces $t$ with $s\not\subseteq t$. We remove all the faces $t$ of $A$ with $s\subseteq t$ and replace each such face with the family 
\[
\{ y\cup \{ \chk{s}\} \mid s\cup y=t \}, 
\]
that is, each face $t$ of $A$ with $s\subseteq t$ is shattered into a number of new faces. 

Observe further that if $s$ is not a face of $A$, then 
$sA=A$.
Defining $sA$ in the cases when $s$ is not a face of $A$ may seem gratuitous; in fact, the general definition makes computations easier 
as it allows us  to write formulas and manipulate them without worrying about excluding certain divisions because they do not change the 
complex being divided.

The operation of division can, of course, be iterated. For a sequence $s_0,\dots , s_n$ of finite non-empty sets, let 
\begin{equation}\label{E:itsi0}
s_0 s_1 \cdots s_l A = s_0( s_1  \cdots (s_l A)\cdots).
\end{equation} 
Our convention is that for the empty sequence $\emptyset$, we have $\emptyset A = A$.

Below, we will say {\bf division} instead of the longer stellar subdivision.

\bigskip

\centerline{\bf For the remainder of Part~\ref{Pa:back}, we fix a complex $\mathbf A$.}

\subsection{A class of complexes and maps among them}\label{Su:acla} $\empty$

We consider the family of all complexes obtained by iteratively dividing $\mathbf A$. 
These complexes carry additional structure; namely, each face $s$ of such a complex remembers in which faces $\sigma$ of 
$\mathbf A$ it is ``included.'' To make it formal, for a complex $B$, which is obtained by iterative application of division to $\mathbf A$, and for a face $\sigma$ of $\mathbf A$, we define a family ${\rm D}^B_\sigma$ of faces of $B$. If $B= {\mathbf A}$, then, for a face $s$ of $\mathbf A$, 
\[
s\in {\rm D}^{\mathbf A}_\sigma \Leftrightarrow s\subseteq \sigma.  
\]
Let now $B= t B'$, for a finite set $t$, and assume ${\rm D}^{B'}_\sigma$ is defined. Let $s$ be a face of $B$. If $s$ is an old face of $B$, that is, it is a face of $B'$, then 
\[
s\in {\rm D}^B_\sigma \Leftrightarrow s\in {\rm D}^{B'}_\sigma. 
\]
Otherwise, $s= \{ \chk{t}\}\cup y$ and $t\cup y$ is a face of $B'$. Let 
\[
s\in {\rm D}^B_\sigma \Leftrightarrow t\cup y\in {\rm D}^{B'}_\sigma. 
\]
It can be checked that with the definition from Appendix~\ref{S:Div} the geometric realizations of $B$ and $\mathbf A$ are homeomorphic with each other with a homeomorphism such that $s$ is an element of ${\rm D}^B_\sigma$ precisely when each vertex of $s$ is mapped to a point that is in the convex span of the vertices of $\sigma$. 
From this point on when talking about a simplicial complex $B$ obtained from $\mathbf A$ by iterated subdivision, we assume that the complex comes equipped with with the additional structure ${\rm D}^B_\sigma$, for $\sigma \in {\mathbf A}$. We call this family ${\rm D}^B_\sigma$, $\sigma\in {\mathbf A}$, the {\bf face structure of} $B$.

Let $A, A'$ be complexes obtained from ${\mathbf A}$ by iterating the division operation. 
A function $f\colon A\to A'$ is called a {\bf grounded simplicial map} if it is simplicial, 
and, for each $\sigma\in {\mathbf A}$,  we have 
\[
t\in {\rm D}^{A'}_\sigma\; \Leftrightarrow\;\big( t= f(s),\hbox{ for some }s\in {\rm D}^A_\sigma\big).
\]
A function $f\colon A\to A'$ is called a 
{\bf grounded isomorphism} if it is a bijection that is grounded simplicial. 
Note that for a grounded isomorphism $f$ we have 
\[
s\in {\rm D}^A_\sigma\Leftrightarrow f(s)\in {\rm D}^{A'}_\sigma, \;\hbox{ for each }s\in A.
\]
Obviously, if $f$ is a grounded isomorphism, so is $f^{-1}$. In the above situation, we say that $A$ and $A'$ are {\bf ground isomorphic}.

Let 
\[
\langle{\mathbf A}\rangle
\]
be the family of all complexes obtained from $\mathbf A$ be iterated application of the division operation, where 
we identify two such complexes if they are ground isomorphic. We equip each element $B$ of $\langle{\mathbf A}\rangle$ with the relations $D_\sigma^B$, for $\sigma\in {\mathbf A}$.

The relation of the class $\langle{\mathbf A}\rangle$ with the class of all complexes obtained from $\mathbf A$ by applying all stellar moves, that is, divisions and, their inverses, welds, is clarified by the recent paper of Adiprasito and Pak \cite{AP}. They prove that if a complex $A'$ is obtained from $\mathbf A$ using stellar moves, then there exist $B\in \langle {\mathbf A}\rangle$ and $B'$ obtained from $A'$ by applying divisions such that $B$ and $B'$ are isomorphic as simplicial complexes. So the family $\langle {\mathbf A}\rangle$ is coinitial in the family of all complexes obtained from $\mathbf A$ using all stellar moves.

\subsection{Additive families of faces}\label{Su:addsu} 

We need to introduce the notion of stellar subdivision of a simplicial map, which in turn requires the notion of additive family of faces. 

Let $S$ be a family of faces of $A$, that is, $S\subseteq A$. We say that a sequence 
${\bar S} = s_0, s_1, \dots, s_l$ is a {\bf non-decreasing enumeration of $S$} if it injectively lists the elements of $S$ and, 
for all $i,j$, $s_i\subseteq s_j$ implies $i\leq j$. That each family $S$ has such an enumeration is a consequence of \cite{Sz}. 
In general, if $\ov{S}$ and $\ov{S}'$ are two non-decreasing enumerations of the same family $S$, then the complexes 
$\ov{S} A$ and  $\ov{S}' A$ need not be equal. There is, however, an important situation in which such a discrepancy cannot occur. 
We say that $S\subseteq A$ is {\bf additive in $A$} if $s_1\cup s_2\in S$ for all $s_1, s_2\in S$ with $s_1\cup s_2\in A$.
We have the following lemma on the independence of division by an additive family on its non-decreasing enumeration. This lemma is derived from its more precise version Lemma~\ref{L:dedf}.

\begin{lemma}\label{L:adar}
Let $S$ be an additive family in a simplicial complex $A$. Let $\vec{S}$ and $\vec{S}'$ be two non-decreasing enumerations 
of $S$. Then $\vec{S} A=\vec{S}' A$.
\end{lemma}

Lemma~\ref{L:adar} allows us to introduce the following piece of notation. For an additive family $S$ of faces of $A$, we write 
\[
S A 
\]
for $\vec{S} A$, where $\vec{S}$ is an arbitrary non-decreasing enumeration of $S$.

\subsection{Stellar subdivisions of simplicial maps}\label{Su:stdm}

Let $f\colon B\to A$ be a simplicial map between simplicial complexes $A$ and $B$. Let $s$ be a face of $A$. We define
now the subdivision of $f$ based on $s$. Consider the following family of faces of $A$
\[
f^{-1}(s) = \{ t \in B \colon f(t)= s\}.
\]
It is easy to check that $f^{-1}(s)$ is an additive family of faces of $B$.  We define
\[
s f
\]
to be the function from the vertices of $\big( f^{-1}(s)\big) B$ to the vertices of $s A$ that maps
each vertex $\chk{t}$ in $\big( f^{-1}(s)\big) B$, for $t \in f^{-1}(s)$, to the vertex $\chk s$ in $sA$, and each vertex $v$ of $B$
to the vertex $f(v)$ of $A$. It is easy to check that the map 
\[
sf\colon \big( f^{-1}(s)\big) B\to sA
\]
is simplicial. 
The map $s f$ will be called the {\bf subdivision of $f$ based on $s$}. To shorten the phrase and since $s$ can be read of the notation $sf$, 
we will often say that $sf$ is the {\bf division of $f$}. Figure~\ref{fig:division-of-a-weld-map} provides an illustration of a division of a grounded simplcial map. 

The following lemma, with the conventions introduced in Part~\ref{P:stfra}, is proved as Lemma~\ref{L:divpr}.

\begin{lemma} 
Assume $A, B\in \langle {\mathbf A}\rangle$. Let $f\colon B\to A$ be a grounded simplicial map, and let $s$ be a face of $A$. 
Then the map 
\[
sf\colon \big( f^{-1}(s)\big) B\to sA
\]
is grounded simplicial. 
\end{lemma}

\section{Weld-division maps}\label{S:divmap}

We define the class of grounded simplicial maps that is fundamental to our considerations.

Let $t$ be a finite non-empty set and let $p\in t$. Such a pair $(p,t)$ determines a simplicial map from $tA$ to $A$ for each complex $A$ as follows.  We define 
\begin{equation}\label{E:wel}
\pi^A_{p,t} \colon  t A \to A
\end{equation} 
by letting it be identity on ${\rm Vr}(A)$ and, when $t\in A$, mapping the new vertex $\chk t$ of $tA$ to $p$. We call a map of this form a {\bf weld map}. Note that if $t$ is not a face of $A$, then the map $\pi^A_{p,t}$ is the identity map. The weld map \eqref{E:wel} reverses the stellar subdivision leading from $A$ to $tA$, so, in this respect, it is a refinement of the weld operation, which is the inverse of the stellar subdivision operation. 
Figure~\ref{fig:weld-map} provides an illustration of a weld map. 

The following lemma, with the conventions of Part~\ref{P:stfra}, is proved as Lemma~\ref{L:welgr2}. 

\begin{lemma}\label{L:welgr} 
Let $A\in \langle {\mathbf A}\rangle$, and let $t$ be a finite set with $p\in t$. Then the weld map $\pi^A_{p.t}$ is grounded simplicial. 
\end{lemma}

We are now ready to equip the class $\langle {\mathbf A}\rangle$ with morphisms, which we will call weld-division maps. 
{\bf Weld-division maps among complexes in $\langle {\mathbf A}\rangle$} is the smallest class of maps that 
\begin{enumerate}
\item[---] contains all weld maps between complexes in $\langle {\mathbf A}\rangle$,   

\item[---] is closed under division of simplicial maps, and  

\item[---] is closed under composition.
\end{enumerate} 
We use 
\[
{\mathcal D}({\mathbf A}) 
\]
to denote the category whose objects are complexes in $\langle {\mathbf A}\rangle$ and whose morphisms 
are weld-division maps among them. 


We provide two schematic drawings illustrating the concepts introduced above. 
The following figure shows a weld map. In it, the squared vertex is mapped to the squared vertex and all the vertices of the outside triangle are mapped to themselves. 

\begin{figure}[htb]
  \centering
  \begin{center}
    \begin{tikzpicture}[scale=1.4]
      \begin{scope}
        \coordinate (A) at (0,0);
        \coordinate (B) at (\side,0);
        \coordinate (C) at (\side/2,{(sqrt(3)/2)*\side});
        \coordinate (D) at (barycentric cs:A=1,B=1,C=1);

        \draw[line width=\lwa] (A) -- (B) -- (C) -- cycle;
        \draw (A) -- (D) -- (C);
        \draw (B) -- (D);
        \node[below left] at (A) {\(a\vphantom{b}\)};
        \node[below right]  at (B) {\(b\)};
        \node[above]  at (C) {\(c\)};

        \tkzMarkSegment[size=2,pos=.5,mark=](A,B)
        \tkzMarkSegment[size=2,pos=.5,mark=](A,C)
        \tkzMarkSegment[size=2,pos=.5,mark=](B,C)
        \tkzMarkSegment[size=2,pos=.5,mark=](A,D)
        \tkzMarkSegment[size=2,pos=.5,mark= ](B,D)
        \tkzMarkSegment[size=2,pos=.5,mark=](C,D)

        \draw (D) \Square{\crb};

      \end{scope}
      \begin{scope}[xshift=4cm]
        \coordinate (A) at (0,0);
        \coordinate (B) at (\side,0);
        \coordinate (C) at (\side/2,{(sqrt(3)/2)*\side});

        \draw[line width=\lwa] (A) -- (B) -- (C) -- cycle;
        \node[below left] at (A) {\(a\vphantom{b}\)};
        \node[below right]  at (B) {\(b\)};
        \node[above]  at (C) {\(c\)};

        \tkzMarkSegment[size=2,pos=.5,mark=](A,B)
        \tkzMarkSegment[size=2,pos=.5,mark=](A,C)
        \tkzMarkSegment[size=2,pos=.5,mark=](B,C)

        \draw (B) \Square{\crb};

      \end{scope}
      \draw[->] (\side,1.2) -- node[pos=0.5,anchor=south]{\(\pi^A_{p,t}\)} (\side+1,1.2);
    \end{tikzpicture}
  \end{center}
  \caption{A weld map with \(p=b\), \(t = \{a,b,c\}\), and $A=$ all non-empty subsets of \(\{a, b, c\}\)}
  \label{fig:weld-map}
\end{figure}
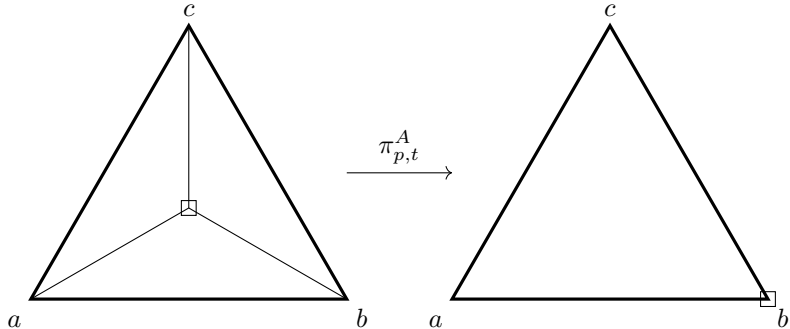

The figure below illustrates a division of a weld map. In it, the circled vertices are mapped to the circled vertex and the squared vertex to the squared vertex. The vertices lying on the outer triangle are mapped to themselves. Some edges internal to the triangle on the left are marked with strokes. The number of strokes indicates the order in which the edges are produced in the process of iterated division of the complex $A$ that is the domain of $\pi^A_{p,t}$.

\begin{figure}[htb]
  \centering
  \begin{center}
    \begin{tikzpicture}[scale=1.4]
      \begin{scope}
        \coordinate (A) at (0,0);
        \coordinate (B) at (\side,0);
        \coordinate (C) at (\side/2,{(sqrt(3)/2)*\side});
        \coordinate (D) at (barycentric cs:A=1,B=1,C=1);
        \coordinate (E) at ($(A)!0.5!(D)$);
        \coordinate (F) at (\side/2,0);
        \coordinate (G) at (barycentric cs:A=1,D=1,B=1);

        \draw[line width=\lwa] (A) -- (B) -- (C) -- cycle;
        \draw (A) -- (D) -- (B);
        \draw (C) -- (D) -- (G) -- (A);
        \draw (C) -- (E) -- (G) -- (B);
        \draw (F) -- (G);

        \node[below left] at (A) {\(a\vphantom{b}\)};
        \node[below right]  at (B) {\(b\)};
        \node[above]  at (C) {\(c\)};

        \tkzMarkSegment[size=2,pos=.5,mark=](A,C)
        \tkzMarkSegment[size=2,pos=.5,mark=](C,B)
        \tkzMarkSegment[size=2,pos=.5,mark=](A,F)
        \tkzMarkSegment[size=2,pos=.5,mark=](F,B)
        \tkzMarkSegment[size=2,pos=.5,mark=](C,D)
        \tkzMarkSegment[size=2,pos=.5,mark=](D,B)
        \tkzMarkSegment[size=2,pos=.5,mark=](D,E)
        \tkzMarkSegment[size=2,pos=.5,mark=](A,E)
        \tkzMarkSegment[size=2,pos=.5,mark=|](D,G)
        \tkzMarkSegment[size=2,pos=.5,mark=|](B,G)
        \tkzMarkSegment[size=2,pos=.5,mark=|](A,G)
        \tkzMarkSegment[size=2,pos=.5,mark=|||](E, G)
        \tkzMarkSegment[size=2,pos=.5,mark=||](G, F)
        \tkzMarkSegment[size=2,pos=.5,mark=|||](E, C)

        \draw (D) \Square{\crb};

        \draw (E) circle (0.8mm);
        \draw (G) circle (0.8mm);
      \end{scope}

      \begin{scope}[xshift=4cm]
        \coordinate (A) at (0,0);
        \coordinate (B) at (\side,0);
        \coordinate (C) at (\side/2,{(sqrt(3)/2)*\side});
        \coordinate (D) at (\side/2,0);

        \draw[line width=\lwa] (A) -- (B) -- (C) -- cycle;
        \draw (C) -- (D);
        \node[below left] at (A) {\(a\vphantom{b}\)};
        \node[below right]  at (B) {\(b\)};
        \node[above]  at (C) {\(c\)};

        \tkzMarkSegment[size=2,pos=.5,mark=](A,D)
        \tkzMarkSegment[size=2,pos=.5,mark=](D,B)
        \tkzMarkSegment[size=2,pos=.5,mark=](A,C)
        \tkzMarkSegment[size=2,pos=.5,mark=](B,C)
        \tkzMarkSegment[size=2,pos=.5,mark=](C,D)

                \draw (B) \Square{\crb};

        \draw (D) circle (0.8mm);
      \end{scope}
      \draw[->] (\side,1.2) -- node[pos=0.5,anchor=south]{\(s\pi^A_{p,t}\)}
      (\side+1,1.2);
    \end{tikzpicture}
  \end{center}
  \caption{A division of a weld map with \(p=b \), \(s = \{a,b\}\), \(t= \{ a, b, c\}\),  and $A=$ all non-empty subsets of \(\{a, b, c\}\) }
  \label{fig:division-of-a-weld-map}
\end{figure}
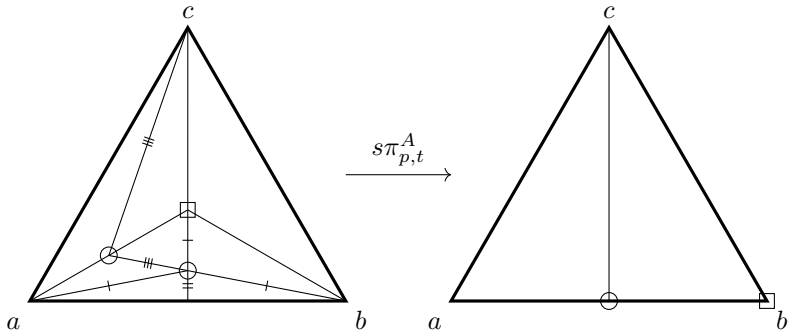

\section{Statements of the main theorems}\label{S:stma}

We state the three main results of the paper---Theorems~\ref{T:frai}, \ref{C:trprc}, and \ref{T:prli}.
Their sharper versions phrased using the setup of Part~\ref{P:stfra} are proved in Part~\ref{P:prfd}.

Theorem~\ref{T:frai} is a structural theorem about the category ${\mathcal D}({\mathbf A})$ of weld-division maps over a complex $\mathbf A$; it gives projective amalgamation for the class of weld-division maps and constitutes the foundations for the proofs of Theorems~\ref{C:trprc} and \ref{T:prli}. 
Theorem~\ref{C:trprc} asserts that ${\mathcal D}({\mathbf A})$ forms a transitive projective Fra{\"i}ss{\'e} class; see Apendix~\ref{Su:prfr}. It follows from Theorem~\ref{C:trprc} that ${\mathcal D}({\mathbf A})$ has the projective Fra{\"i}ss{\'e} limit that has the canonical quotient; see Apendix~\ref{Su:prfr}. Theorem~\ref{T:prli} identifies this quotient as the geometric realization of the complex $\mathbf A$.

\begin{theorem}\label{T:frai} 
For $f',\, g'\in {\mathcal D}({\mathbf A})$ with the same codomain, there exist $f, \, g\in {\mathcal D}({\mathbf A})$ 
such that 
\begin{equation}\notag
f'\circ f = g'\circ g.
\end{equation} 
\end{theorem}

The theorem above, actually a sharper version of it, is proved as Theorem~\ref{T:frai2}.

We will show that Theorem~\ref{T:frai} implies Theorem~\ref{C:trprc} asserting that ${\mathcal D}({\mathbf A})$ is a 
transitive projective Fra{\"i}ss{\'e} class. 
To make the statement of Theorem~\ref{C:trprc} precise, we need to move from ${\mathcal D}({\mathbf A})$ to a category, whose objects are sets equipped with a symmetric and reflexive binary relation; see Appendix~\ref{Su:prfr}. We fix 
a symbol $R$. Given a complex $A$ in ${\mathcal D}({\mathbf A})$, we interpret $R$ in ${\rm Vr}(A)$ in a natural way. 
Namely, for $v,w\in {\rm Vr}(A)$, we let 
\[
vR^Aw\;\Leftrightarrow\; \{ v,w\}\in A, 
\]
that is, $v$ and $w$ are related with respect to $R$ precisely when they are 
elements of a face of $A$. Then $R^A$ is a symmetric and reflexive binary relation on ${\rm Vr}(A)$. We also note that any 
surjective simplicial map $f\colon A\to B$ is a strong homomorphism from ${\rm Vr}(A)$ equipped with $R^A$ to ${\rm Vr}(B)$ 
equipped with $R^B$. In particular, 
all weld-division maps in ${\mathcal D}({\mathbf A})$ are strong homomorphisms. In the context of projective Fra{\"i}ss{\'e} classes, 
we can and do regard each complex $A$ in 
${\mathcal D}({\mathbf A})$ as the set ${\rm Vr}(A)$ equipped with the relation $R^A$. Morphisms among this class of objects  
are just weld-division maps in ${\mathcal D}({\mathbf A})$ regarded as functions among the sets ${\rm Vr}(A)$. 
We denote this new category by ${\mathcal D}_R({\mathbf A})$; see \cite{Se} for more on the close relationship between the two categories.

\begin{theorem}\label{C:trprc} 
${\mathcal D}_R({\mathbf A})$ is a transitive projective Fra{\"i}ss{\'e} class. 
\end{theorem} 
 
The theorem above is proved as Theorem~\ref{T:trprc2}.

In light of Theorem~\ref{C:trprc}, the class ${\mathcal D}({\mathbf A})$ has the canonical quotient. Theorem~\ref{T:prli} determines this quotient. 

\begin{theorem}\label{T:prli}
The canonical quotient of the transitive projective Fra{\"i}ss{\'e} class ${\mathcal D}_R({\mathbf A})$ is homeomorphic to 
the geometric realization of $\mathbf A$.
\end{theorem} 

A more precise version of the theorem above is proved as Theorem~\ref{T:prli2}. 

Reversing the perspective, we may view Theorem~\ref{T:prli} as providing a combinatorial description of the geometric realization of the complex $\mathbf A$. Indeed the abstract complex $\mathbf A$ is a combinatorial object, the class ${\mathcal D}({\mathbf A})$ of weld-division maps is combinatorially generated from $\mathbf A$, the projective Fra{\"i}ss{\'e} limit of this class is a totally disconnected compact space equipped with a compact equivalence relation, so also a combinatorial object. Theorem~\ref{T:prli} asserts that the geometric realization of $\mathbf A$ is obtained as a canonical quotient of this last object.

\newpage

\part{Set theoretic setup for complexes and their divisions}\label{P:stfra}

Our proofs will consist of performing calculations on finite sequences of finite sets and on appropriately defined functions among them. The set theoretic setup described in this section is necessary to carry out and apply these calculations. So the point of this part is to motivate definitions in Part~\ref{P:frpr} and to make possible applications in Part~\ref{P:prfd} of the results of Part~\ref{P:repr} to obtain Theorems~\ref{T:frai}--\ref{T:prli}. 

In the current part, we present an arbitrary complex so that its faces have a particular, simple set theoretic structure and observe consequences of this presentation for division and for grounded simplicial maps. Concretely, we need to do two things:

1. restrict the class of sets allowed to be faces of complexes; this class of sets is ${\rm Fin}^+$, 
which is defined in Section~\ref{Su:finp};

2. make a uniform specification of what the new vertex $\chk{s}$ in the division given by \eqref{E:divin} is; this is done in formula \eqref{E:divinf}. 

The restriction in 1 will be purely set theoretic and it will not limit the class of complexes or the class of divisions 
if those are taken up to isomorphisms. On the other hand, we will have to ensure that divisions of complexes that are admitted by 1 will not lead outside of the class of those complexes. 
The specification in 2 will help with this goal.

\section{The family of sets ${\rm Fin}^+$}\label{Su:finp}

We assume that there exists a set $\rm Ur$ consisting of all urelements of arbitrary cardinality; see Appendix~\ref{A:sett} for a discussion.
Define 
\[
{\rm Fin}^+_n = 
\begin{cases} 
{\rm Ur}, &\text{ if $n=0$};\\
\{ x\mid x \hbox{ is finite and } \emptyset\not= x \subseteq {\rm Fin}^+_{n-1}\}, &\text{ if $n>0$}.
\end{cases} 
\]
Finally, let 
\[
{\rm Fin}^+ = \bigcup_{n>0}^\infty {\rm Fin}^+_n. 
\]
Note that the union above starts with $n=1$. So, ${\rm Fin}^+$ is the family of all sets that are constructed from elements of ${\rm Ur}$ using iterations of the operation of taking finite non-empty sets. In particular, ${\rm Fin}^+$ is disjoint with $\rm Ur$. 
For $s\in {\rm Fin}^+$, we define the support of $s$ by letting  
\[
{\rm sp}(s)= {\rm tc}(s)\cap {\rm Ur},
\]
where the set theoretic transitive closure operation $\rm tc$ is defined in Appendix~\ref{A:sett}.
So, ${\rm sp}(s)$ collects all elements of $\rm Ur$ involved in the construction of $s$. The family ${\rm Fin}^+$ and the operation $\rm sp$ are closely related to the smallest admissible set and the support function in \cite[Sections I.6 and II.2]{Bar}. The following properties of $\rm sp$ are implied by \eqref{E:trpr}. For $s,t\in {\rm Fin}^+$, we have 
\begin{equation}\label{E:sppr}
\begin{split}
{\rm sp}(s) &\subseteq {\rm sp}(t),\hbox{ if $s\in t$ or $s\subseteq t$},\\
{\rm sp}(s\cup t) &= {\rm sp}(s)\cup {\rm sp}(t),\\
{\rm sp}\big(\{ s\}\big) &= {\rm sp}(s).
\end{split} 
\end{equation}

\section{Grounded and divided complexes and grounded simplicial maps}\label{S:divset}

\subsection{Complexes in ${\rm Fin}^+$ and their divisions}\label{Su:divinf}

It will be sufficient and convenient to limit our attention to complexes that are subsets of ${\rm Fin}^+$ and that are obtained by dividing complexes that are placed inside of ${\rm Fin}^+$ in a particularly simple way. A complex $A$ is called {\bf grounded} if each face of $A$ is a subset of $\rm Ur$, that is, ${\rm Vr}(A)\subseteq {\rm Ur}$. 

\begin{lemma}\label{L:latt} 
Each complex is isomorphic to a grounded complex. 
\end{lemma}

\begin{proof}
Given a complex $A$, find an injective function from ${\rm Vr}(A)$ to $\rm Ur$ and use it to transfer $A$ to its isomorphic copy $A'$ with ${\rm Vr}(A')\subseteq {\rm Ur}$. 
\end{proof}

We restate the definition of division in a way set theoretically appropriate for our proofs. We specify what the new vertex in the division formula \eqref{E:divin} is---when dividing by a face $s$, the new vertex $\chk{s}$ is equal the set $s$ itself.  

Let $A$ be a family of sets in ${\rm Fin}^+$ and let $s$ is a finite non-empty set. Define $sA$ to consist of the following sets 
\begin{equation}\label{E:divinf}
\begin{cases}
y\cup \{ s\}, &\hbox{ if }s\not\subseteq y \hbox{ and } s\cup y\in A;\\
y, &\hbox{ if } s\not\subseteq y\hbox{ and } y\in A. 
\end{cases}
\end{equation} 
The formula above differs from \eqref{E:divin} by specifying that the new vertex $\chk{s}$ is equal to the set $s$.

For a sequence $s_0\cdots s_n$ of finite non-empty sets, let 
\begin{equation}\label{E:itsi}
s_0 s_1 \cdots s_l A = s_0( s_1  \cdots (s_l A)\cdots).
\end{equation} 
As before, $\emptyset A =A$.

\begin{lemma}\label{L:ok}
Let $A$ be a grounded complex and let $s_0 s_1\cdots s_n$ be a sequence of finite non-empty sets. Then $s_0 s_1 \cdots s_l A$ is a complex with the property 
\begin{equation}\label{E:nottr} 
s\not\in {\rm tc}(t), \hbox{ for all }s,t\in s_0\cdots s_n A. 
\end{equation} 
\end{lemma} 

\begin{proof}
It is immediate to see that if $B$ is a family of sets in ${\rm Fin}^+$ that is closed under taking non-empty subsets and $s$ is a finite non-empty set, then $sB$ is closed under taking non-empty subsets. So, to see that $s_0 s_1 \cdots s_l A$ is a complex it suffices to check that 
\begin{equation}\notag
s\not\in t, \;\hbox{ for all }s,t\in s_0\cdots s_nA. 
\end{equation} 
Note that condition \eqref{E:nottr} implies the condition above. It follows that it suffices to check that $s_0 s_1 \cdots s_l A$ fulfills \eqref{E:nottr}. We check that grounded complexes fulfill \eqref{E:nottr} and that if $B$ fulfills \eqref{E:nottr}, then so does $sB$. So, the advantage of assuming \eqref{E:nottr} is that this condition is preserved under divisions as specified above.

If $A$ is grounded and $s,t\in A$, then $s\in {\rm Fin}^+$ and $t\subseteq {\rm Ur}$, so $s\not\in {\rm tc}(t)$ as all elements of ${\rm tc}(t) =t$ are elements of $\rm Ur$ and $s$ is not. 

Assume now that $B$ is a complex consisting of sets in ${\rm Fin}^+$ that fulfills condition \eqref{E:nottr}. Let $s$ be a finite non-empty set. We show that $sB$ fulfills \eqref{E:nottr}.  
If $s$ is not a face of $B$, then $sB=B$ and there is nothing to check. Assume $s$ is a face of $B$. Let $r_1, r_2$ be faces of $sB$. Assume towards a contradiction that 
\[
r_1\in {\rm tc}(r_2). 
\]
We have the following cases based on \eqref{E:divinf}. 

$r_1\in B$ and $r_2\in B$. This case is impossible by $B$ fulfilling \eqref{E:nottr}. 

$r_1\in B$ and $r_2= y \cup \{ s\}$, for some $y$ with $y\cup s\in B$. Then $r_1\in {\rm tc}(y)$ or $r_1\in {\rm tc}(s)$ or $r_1=s$. The first two possibilities are excluded by $B$ fulfilling \eqref{E:nottr}. For the third possibility note that it implies $s\in sB$ since $r_1\in sB$. But then 
$s= y'\cup\{ s\}$ for some $y'$ with $y'\cup s\in B$, which implies $s\in s$ contradicting $B$ fulfilling \eqref{E:nottr}. 

$r_1= y_1 \cup \{ s\}$ and $r_2= y_2\cup \{ s\}$, for some $y_1, y_2$ with $y_1\cup s,\, y_2\cup s\in B$. Then $s\in {\rm tc}(y_2)$ or $s\in {\rm tc}(s)$, both of which are excluded by the assumption that $B$ fulfills \eqref{E:nottr}. 

$r_1= y\cup \{ s\}$, for some $y$ with $y\cup s\in B$ and $r_2\in B$. Then $s\in {\rm tc}(r_2)$, which contradicts $B$ fulfilling \eqref{E:nottr}. 
\end{proof}

We define a complex to be {\bf divided} if it is of the form $s_0 s_1 \cdots s_l A$ for a grounded complex $A$.

Lemmas~\ref{L:latt} and \ref{L:ok} show that we can restrict our attention to divided complexes. Further, we can restrict the division operation to sets in ${\rm Fin}^+$ since the faces of a divided complex are sets in ${\rm Fin}^+$, so dividing a divided complex by a set not in ${\rm Fin}^+$ leaves the complex unchanged. As explained in the next section, the operation of support $\rm sp$ recovers the grounded complex $A$ from the divided complex $s_0s_1\cdots s_n A$, in particular, the grounded complex $A$ is determined by $s_0s_1\cdots s_n A$.

\subsection{Divided complexes and grounded simplicial maps among them}

We point out some important advantages/simplifications resulting from considering complexes as in Section~\ref{Su:divinf}.
We show that when considering complexes $A$ obtained by iterative division from grounded complexes, the face structure ${\rm D}^A_\sigma$ does not need to be specified explicitly as it is encoded by $A$ alone  and groundedness of simplicial maps among such complexes is defined from the complexes alone without the face structure.  
The next lemma shows that the operation $\rm sp$ of support recovers from $A$ the grounded complex from which $A$ was constructed as well as the face structure ${\rm D}^A_\sigma$. 

Fix a grounded complex ${\mathbf A}$. 
Recall from Section~\ref{Su:acla}, the family $\langle {\mathbf A}\rangle$ of complexes obtained by iteratively dividing $\mathbf A$ and 
the relations ${\rm D}^A_\sigma$ for a complex $A$ in $\langle {\mathbf A}\rangle$ and a face $\sigma$ of $\mathbf A$. 

\begin{lemma}\label{L:dsig} 
Let $A$ be a complex in $\langle {\mathbf A}\rangle$.
\begin{enumerate}
\item[(i)] ${\mathbf A}= \{ {\rm sp}(s)\mid s\in A\}$. 

\item[(ii)] For $\sigma\in {\mathbf A}$ and for $s\in A$, we have 
\[
s\in {\rm D}^A_\sigma\; \Leftrightarrow\; {\rm sp}(s)\subseteq \sigma. 
\]
\end{enumerate} 
\end{lemma} 

\begin{proof} (i) The equality is obvious for $A={\mathbf A}$. If $A=tA'$, for some $A'\in \langle {\mathbf A}\rangle$, with the equality holding for $A'$, then it follows directly from the definition of dividing that the equality holds for $A$ since ${\rm sp}\big( y\cup \{ t\}\big) = {\rm sp}(y\cup t)$, for all $t,y\in {\rm Fin}^+$. 

(ii) The equivalence is obviously true if $A= {\mathbf A}$. Assume the equivalence holds for $A'$ and let $A=tA'$. We show it for $A$. 
The only case that needs to be considered is $s= \{ t\}\cup y$ with $t\cup y\in A'$. In this situation, notice that ${\rm sp}(s) = {\rm sp}(t\cup y)$. So 
we have 
\[
s\in {\rm D}^A_\sigma\Leftrightarrow t\cup y \in {\rm D}^{A'}_\sigma \Leftrightarrow {\rm sp}(t\cup y) \subseteq \sigma \Leftrightarrow {\rm sp}(s) \subseteq \sigma,
\]
where the first equivalence follows from the definition of ${\rm D}^A_\sigma$, the second one follows from the inductive assumption, and the third one from the previous sentence. 
\end{proof}

The next lemma asserts that groundedness of a simplicial surjection is remembered through the support operation $\rm sp$. 

\begin{lemma}\label{L:simsp} 
Let $A$ and $B$ be in $\langle {\mathbf A}\rangle$, and let $f\colon B\to A$ be simplicial. 
Then $f$ is grounded if and only if 
\begin{enumerate}
\item[---] ${\rm sp}\big(f(t)\big)\subseteq {\rm sp}(t)$, for each $t\in B$ and 

\item[---] for each $s\in A$, there exists $t\in B$ with $f(t)=s$ and ${\rm sp}(t)={\rm sp}(s)$.
\end{enumerate}
\end{lemma} 

\begin{proof} 
The direction $\Leftarrow$ follows immediately from Lemma~\ref{L:dsig}. For the other direction, let $s\in B$. Let $\sigma={\rm sp}(s)\in {\mathbf A}$. Then by Lemma~\ref{L:dsig} we have $s\in {\rm D}^B_\sigma$. Since $f$ is assumed to be grounded, we get $f(s)\in {\rm D}^A_\sigma$. So again by Lemma~\ref{L:dsig}, 
$f(s)\subseteq\sigma = {\rm sp}(s)$, and the conclusion follows. 
\end{proof}

{\bf Precise restatements of the definitions of weld maps, grounded isomorphisms, and division of grounded simplicial maps in the framework described above can be found in Appendix~\ref{A:isome}.}

\subsection{Internal description of faces of divided complexes}

We have a lemma that gives a description of faces of a complex $s_0\cdots s_m {\mathbf A}$ obtained by iterative division of the grounded complex $\mathbf A$ in terms of the support operation $\rm sp$ and the sequence $s_0\cdots s_m$. Note that condition (ii) in the lemma can be used to recursively check condition (i). Condition (ii) will be crucial in implementing our calculus of sequences of sets.

\begin{lemma}\label{L:conr}
Let $t$ and $s_0, \dots, s_m$ be sets in ${\rm Fin}^+$. Then (i) and (ii) are equivalent.
\begin{enumerate} 
\item[(i)] $t$ is a face of $s_0\cdots s_m{\mathbf A}$; 

\item[(ii)] ${\rm sp}(t)\in {\mathbf A}$ and one of the following two conditions holds: 
\begin{enumerate}
\item[---]  $t\subseteq {\rm Ur}$ and  $s_j\not\subseteq t$, for all $j\leq m$ with $s_j$ being a face of $s_{j+1}\cdots s_m{\mathbf A}$;

\item[---] there is $i\leq m$ such that 

\begin{enumerate}
\item[--] $s_i\in t$, 

\item[--] $\big(t\setminus \{ s_i\}\big)\cup s_i$ is a face of $s_{i+1}\cdots s_m{\mathbf A}$, 

\item[--] $s_j\not\subseteq t$, for all $j\leq i$ with $s_j$ being a face of $s_{j+1}\cdots s_m{\mathbf A}$.
\end{enumerate} 
\end{enumerate} 
\end{enumerate} 
\end{lemma} 

\begin{proof} Set $\vec{s} = s_0\cdots s_m$.

We prove direction (ii)$\Rightarrow$(i) first. Assume $t$ satisfies (ii). If $t\subseteq {\rm Ur}$, then, since $t= {\rm sp}(t)\in {\mathbf A}$, we see that $t$ is a face of ${\mathbf A}$. 
It remains to show, inductively on $j\leq m$, that $t$ is a face of $s_j\cdots s_m {\mathbf A}$. If $s_j$ is not a face of $s_{j+1}\cdots s_m {\mathbf A}$, then $s_js_{j+1}\cdots s_m {\mathbf A} = s_{j+1}\cdots s_m {\mathbf A}$. If $s_j$ is a face of $s_{j+1}\cdots s_m {\mathbf A}$, then, by assumption, $s_j\not\subseteq t$. In either case, if $t$ is a face of $s_{j+1}\cdots s_m {\mathbf A}$, then it remains a face of $s_j s_{j+1}\cdots s_m {\mathbf A}$. It follows that $t$ is a face of $\vec{s}{\mathbf A}$. 

Now assume that there exists $i\leq m$ with the property: 
$s_i\in t$,  $(t\setminus \{ s_i\})\cup s_i$ is a face of $s_{i+1}\cdots s_m{\mathbf A}$, and $s_j\not\subseteq t$, for each $j\leq i$ such that $s_j$ is a face of $s_{j+1}\cdots s_m {\mathbf A}$. Note that, in particular, $s_i$ is a face of $s_{i+1}\cdots s_m{\mathbf A}$, and, therefore, $s_i\not\subseteq t$.
We have 
\[
t = \{ s_i\} \cup \big( t\setminus \{ s_i\}\big). 
\]
So, by the definition of dividing of $s_{i+1}\cdots s_m{\mathbf A}$ by $s_i$,  it follows that $t$ is a face of $s_is_{i+1}\cdots s_m {\mathbf A}$ as $\{ s_i\} \cup \big( t\setminus \{ s_i\}\big)$ is a face of $s_{i+1}\cdots s_m{\mathbf A}$ by our assumption on $t$. Now we argue that, for each $j\leq i$, $t$ is a face of $s_j\cdots s_m {\mathbf A}$. This is proved inductively. It is true for $j=i$. For each $j<i$, either $s_j$ is not a face of $s_{j+1}\cdots s_mA$, so 
\[
s_js_{j+1}\cdots s_m {\mathbf A}= s_{j+1}\cdots s_m{\mathbf A}, 
\]
or $s_j\not\subseteq t$. In either case, if $t$ is a face of $s_{j+1}\cdots s_m{\mathbf A}$, it continues to be a face of $s_js_{j+1}\cdots s_m {\mathbf A}$. It follows that $t$ is a face of $\vec{s} {\mathbf A}$.

We now prove (i)$\Rightarrow$(ii). The proof is by induction on the length of the sequence $\vec{s}$. If $\vec{s}$ is the empty sequence $\emptyset$, the implication is obvious since $\emptyset {\mathbf A}={\mathbf A}$, faces of $\mathbf A$ are subsets of $\rm Ur$, and, for $t\subseteq {\rm Ur}$, ${\rm sp}(t)=t$. 

Assume the implication holds for $\vec{s}$, and let $t$ be a face of $s\vec{s}{\mathbf A}$ for some $s\in {\rm Fin}^+$. By Lemma~\ref{L:dsig}(i), we see that ${\rm sp}(t)\in {\mathbf A}$. We show that $t$ satisfies the remainder of (ii) for $s\vec{s}$. By the definition of division of $\vec{s}{\mathbf A}$ by $s$, we have two cases.

Case 1. $t$ is a face of $\vec{s}{\mathbf A}$ and $s\not\subseteq t$ 

By our inductive assumption, $t$ satisfies (ii) for $\vec{s}$, which together with $s\not\subseteq t$, immediately gives that $t$ satisfies (ii) for $s\vec{s}$.

Case 2. $t$ is not a face of $\vec{s}{\mathbf A}$. 

We have that $t= \{ s\}\cup y$ for some $y$ with $s\cup y$ being a face $\vec{s}{\mathbf A}$ and $s\not\subseteq y$. Observe that, in this case, $s$ is a face $\vec{s}{\mathbf A}$ and 
\[
\big(t\setminus \{ s\}\big)\cup s = s\cup y\in \vec{s}{\mathbf A}.
\] 
Thus, we will get that $t$ satisfies (ii) for $s\vec{s}$ as long as we show that $s\not\subseteq t$.  
Since $s\not\in s$, as $\vec{s}{\mathbf A}$ fulfills condition \eqref{E:nottr} and $s$ is a face of $\vec{s}{\mathbf A}$, we see that $s\subseteq t$ implies  
$s\subseteq y$, which leads to a contradiction. 
\end{proof}

\subsection{Summary} 

In this section, we realized the following goals: 
\begin{enumerate}
\item[---] we represented an arbitrary complex as a grounded complex, in particular, as a divided complex, in particular, as a complex that is a subset of ${\rm Fin}^+$; 

\item[---] we showed that divided complexes fulfill condition \eqref{E:nottr};

\item[---] we showed that, for a divided complex, the face structure 
is expressed by the support operation as is the groundedness of simplicial surjections among divided complexes;

\item[---] we expressed the property of being a face of a divided complex, that is, a complex obtained by a sequence of divisions from a grounded complex, in terms of the sequence of sets performing the divisions.
\end{enumerate} 

As a consequence of the statements above, in what follows, {\bf we only need to consider divided complexes, divisions by sets in ${\rm Fin}^+$, and simplicial surjections with the property from Lemma~\ref{L:simsp}}.

\newpage

\part{Framework for the proofs}\label{P:frpr} 

To motivate the setup, we observe that finite sets are fundamental to our considerations as complexes are families of finite sets and 
finite sets applied to complex subdivide them as in \eqref{E:divin}.
Considering iterative divisions of complexes and simplicial maps among so divided complexes, we see that 
finite sequences of finite sets can be regarded as potential divisions performed on a complex, that is, if $s_0\cdots s_l$ is a finite sequence of finite sets and $A$ is a complex, then $s_0\cdots s_l$ can be applied to $A$ to form the complex $s_0(s_1(s_2\cdots (s_lA) \cdots ))$. Similarly, the appropriately defined functions among sequences of finite sets can be regarded as potential simplicial maps between the divided complexes, that is, in the presence of a complex $A$, a function from $t_0\cdots t_m$ to $s_0\cdots s_l$ will become a simplicial function from the complex obtained by applying $t_0\cdots t_m$ to $A$ to the complex obtained by applying $s_0\cdots s_l$ to $A$.

The paragraph above makes it plausible that arguments concerning divisions of complexes and simplicial maps among them can be carried out by performing computations on finite sequences of finite sets and functions among them. This will indeed be the case. 
To formalize these computations, or, in other words, to make them mathematically correct, we introduce the framework of this section.

\section{Sequences of sets and grounded simplicial functions}

\subsection{Sequences, their faces, and combinatorial equivalence of sequences}\label{Su:combs} 

We consider finite sequences of sets in ${\rm Fin}^+$. We refer to them simply as sequences. 
We define faces and vertices of sequences and combinatorial equivalence among sequences. These notions are motivated by Propositions~\ref{P:conr}, \ref{P:vecv}, and \ref{P:verd}. Lemmas~\ref{L:safa} and \ref{L:savr} determine the connections between combinatorial equivalence among sequences, on one side, and faces and vertices, on the other.

Our convention will be to write $s_0\cdots s_m$ for a sequence with $s_i\in {\rm Fin}^+$, $i\leq m$. The empty sequence, written $\emptyset$, is allowed. 
To abbreviate our notation, we will often write $\vec{s}$ for $s_0\cdots s_m$.

We introduce now the notion of face of a sequence $\vec{s}$. The definition is internal to $\vec{s}$ as this is the way it is used in the proofs. The motivation for it comes from 
Proposition~\ref{P:conr} below through Lemma~\ref{L:conr}. 
The definition of $t\in {\rm Fin}^+$ being a face of a sequence $s_0\cdots s_m$ of sets in ${\rm Fin}^+$ 
is recursive on $m$, that is, we assume the notion is defined for sequences strictly shorter than $s_0\cdots s_m$ when defining it for $s_0\cdots s_m$. 
A set $t\in {\rm Fin}^+$ is a {\bf face in} $s_0\cdots s_m$ if 
\begin{enumerate}
\item[---]  $t\subseteq {\rm Ur}$ and  $s_j\not\subseteq t$, for all $j\leq m$ with $s_j$ being a face of $s_{j+1}\cdots s_m$

\item[] or 

\item[---] there is $i\leq m$ such that $s_i\in t$, $\big(t\setminus \{ s_i\}\big)\cup s_i$ is a face of $s_{i+1}\cdots s_m$, and 
$s_j\not\subseteq t$, for all $j\leq i$ with $s_j$ being a face of $s_{j+1}\cdots s_m$.
\end{enumerate} 
Note that all finite non-empty subsets of $\rm Ur$ are faces of the empty sequence $\emptyset$.

The following proposition is an immediate consequence of Lemma~\ref{L:conr}. 

\begin{proposition}\label{P:conr}
Let $t$ be a set in ${\rm Fin}^+$, let $\vec{s}$ be a sequence, and let $A$ be a grounded complex. Then 
\[
t\hbox{ is a face of }\vec{s}A\;\Leftrightarrow \; t \hbox{ is a face of }\vec{s}\hbox{ and }{\rm sp}(t)\in A. 
\]
\end{proposition}

\begin{proposition}\label{P:trse} 
Let $\vec{s}$ be a sequence 
\begin{enumerate} 
\item[(i)] The family of all faces of $\vec{s}$ forms a complex. 

\item[(ii)] If $t_1$ and $t_2$ are faces of $\vec{s}$, then $t_1\not\in {\rm tc}(t_2)$. 
\end{enumerate} 
\end{proposition} 

\begin{proof} By Proposition~\ref{P:conr} the family of all faces of $\vec{s}$ consists of all faces of $\vec{s}A$, where $A$ is the complex consisting of all finite non-empty subsets of $\rm Ur$. So, by Lemma~\ref{L:ok}, this is a complex and it satisfies the property in (ii) . 
\end{proof}

It will be convenient to identify some sequences with each other. To this end, we define an equivalence relation $\equiv$ on all sequences, which will be called {\bf combinatorial equivalence}. Proposition~\ref{P:vecv} motivates this definition. We let $\equiv$ be the smallest equivalence relation with the following three properties:
\begin{enumerate} 
\item[(a)] for $t\in {\rm Fin}^+$, if $t$ is not a face of $\vec{t}$, then 
\[
t\,\vec{t} \equiv \vec{t};
\]

\item[(b)] for $s, t\in {\rm Fin}^+$ such that $s\not\in t$ and $t\not\in s$, 
if $s\cup t$ is not a face of $\vec{t}$ or $s\cap t=\emptyset$, then 
\[
s\, t\,\vec{t}\equiv t\,s\, \vec{t};
\]

\item[(c)] if $\vec{t}\equiv \vec{t'}$, then, for each set $s\in {\rm Fin}^+$, 
\[
s\, \vec{t} \equiv s\, \vec{t'}.
\]
\end{enumerate}

As mentioned, Proposition~\ref{P:vecv} below gives the intuition behind the definition of the equivalence relation. 

\begin{proposition}\label{P:vecv}
Let $\vec{s}$ and $\vec{t}$ be sequences and let $A$ be a grounded complex. Then 
\[
\vec{s}\equiv \vec{t}\;\Rightarrow\; \vec{s} A =\vec{t} A.
\]
\end{proposition}

\begin{proof}
One only needs to show that the implication holds if $\vec{s}$ and $\vec{t}$ are related by one of the three conditions (a)--(c) in the definition of $\equiv$. The cases of (a) and (c) are obvious. The case of (b) is handled immediately by Corollary~\ref{C:comg}.
\end{proof}

The following lemma gives the connection between combinatorial equivalence and faces.

\begin{lemma}\label{L:safa} 
If $\vec{s}\equiv \vec{t}$, then $\vec{s}$ and $\vec{t}$ have the same faces. 
\end{lemma} 

\begin{proof} 
By Proposition~\ref{P:conr}, $r$ is a face of $\vec{s}$ if and only if $r$ is a face of $\vec{s} A$ for some grounded complex $A$. The same equivalence holds for $\vec{t}$. Now, the conclusion follows from Proposition~\ref{P:vecv}. 
\end{proof}

Now we define the set of vertices of a sequence in ${\rm Fin}^+$. 
For a sequence $\vec{t}$, define 
\begin{equation}\label{E:vrt}
{\rm vr}(\vec{s}\,) = \bigcup\{ t\mid t\hbox{ is a face of } \vec{s}\,\}. 
\end{equation} 
The following proposition follows immediately from Proposition~\ref{P:conr}. 

\begin{proposition}\label{P:verd} 
Let $A$ be a grounded complex. Then, for $x\in {\rm Fin}^+\cup {\rm Ur}$, we have 
\[
x\in {\rm Vr}(\vec{s}A)\; \Leftrightarrow\; x\in {\rm vr}(\vec{s}\,)\hbox{ and } {\rm sp}(\{ x\})\in A. 
\]
In particular, ${\rm Vr}( \vec{s}A)\subseteq {\rm vr}(\vec{s})$. 
\end{proposition}

\begin{lemma}\label{L:savr} 
If $\vec{s}\equiv \vec{t}$, then ${\rm vr}(\vec{s}\,)= {\rm vr}(\vec{t}\,)$. 
\end{lemma} 

\begin{proof} The conclusion is clear from Lemma~\ref{L:safa}. 
\end{proof}

We finish the section with three technical lemmas that will be needed later on. 

\begin{lemma}\label{L:subf} 
Let $\vec{t}=t_0\cdots t_n$ be a sequeence.  
\begin{enumerate}
\item[(i)] If $s$ be a non-empty subset of a face of $\vec{t}$, then $s$ is a face of $\vec{t}$. 

\item[(ii)] If $s\in {\rm vr}(\vec{t}\,)$, then $s\in {\rm Ur}$ or, for some $i\leq n$, $t_i=s$ and $t_i$ is a face of $t_{i+1}\cdots t_n$. 
\end{enumerate} 
\end{lemma} 

\begin{proof} (i) is proved by an easy induction on the length of $\vec{t}$ or use Proposition~\ref{P:conr}. We leave the details to the reader.

(ii) The proof is by induction on the length of the sequence $t_0\cdots t_n$. If this sequence is empty, the conclusion follows since all faces of the empty sequence 
$\emptyset$ are subsets of $\rm Ur$, so $s\in {\rm Ur}$, if $s\in t$ and $t$ is a face of $\emptyset$. 

Assume now that $s\in t$ and $t$ is a face of $t_0\cdots t_n$ and the conclusion holds for all sequences of length $<n+1$. Then either $t\subseteq {\rm Ur}$ or there is $i\leq n$ such that $t_i\in t$ and $(t\setminus \{ t_i\}) \cup t_i$ is a face of $t_{i+1}\cdots t_n$. In the first case, $s\in {\rm Ur}$. 
In the second case, we have two subcases: $s= t_i$, so $s\subseteq (t\setminus \{ t_i\})\cup t_i$, or $s\in (t\setminus \{ t_i\})\cup t_i$. 
In the first subcase, by (i), $s$ is a face of $t_{i+1}\cdots t_n$ as it is a subset of a face and it is not empty (being equal to $t_i$). In the second subcase, we apply our inductive assumption. 
\end{proof}

Point (ii) of the second lemma strengthens one of the conditions in the definition of face. Point (i) will be useful in assessing if a set is a face of a sequence.

\begin{lemma}\label{L:nonfa}
Let $\vec{t} = t_0\cdots t_n$ be a sequence.
\begin{enumerate} 
\item[(i)] If $t$ is a face of $\vec{s}\,\vec{t}$, for some sequence $\vec{s}$, and $t\subseteq \bigcup_{i\leq n} {\rm vr}(t_i\cdots t_n\,)$, then $t$ is a face of $\vec{t}$. 

\item[(ii)] If $t$ is a face of $\vec{t}$, then $t_j\not\subseteq t$ for each $j\leq n$ with $t_j\subseteq \bigcup_{j<i\leq n} {\rm vr}(t_{i}\cdots t_n)$. 
\end{enumerate} 
\end{lemma} 

\begin{proof} We start with a claim that will be useful in proving both (i) and (ii). 

\begin{claim*} 
Let $t$ be a face of $\vec{t} = t_0\cdots t_n$. Then, for each $j\leq n$ such that $t_j$ is a face of $t_{j+1}\cdots t_n$, we have $t_j\not\subseteq t$. 
\end{claim*}

\noindent {\em Proof of Claim.} This is proved by induction on the length of the sequence $t_0\cdots t_n$. The conclusion is clear for the empty sequence. 

If $t\subseteq {\rm Ur}$, then the conclusion holds by definition of face. Otherwise, there exists $i\leq n$ such that $t_i\in t$ and $(t\setminus \{ t_i\})\cup t_i$ is a face of $t_{i+1}\cdots t_n$, and if $j\leq i$ is such that $t_j$ is a face of $t_{j+1}\cdots t_n$, then $t_j\not\subseteq t$. Thus, if there exists $j\leq n$ such that $t_j$ is a face of $t_{j+1}\cdots t_n$ and  $t_j\subseteq t$, then $i<j$. If we show that $t_i\not\in t_j$, then from $t_j\subseteq t$, we will see that $t_j\subseteq t\setminus \{ t_i\}$, which, by induction, will give that 
$(t\setminus \{ t_i\})\cup t_i$ is not a face of $t_{i+1}\cdots t_n$, a contradiction. So, it suffices to see that $t_i\not\in t_j$. But if $t_i\in t_j$, then, by Lemma~\ref{L:subf}(ii), 
for some $k>j$, $t_i=t_k$ and $t_k$ is a face of $t_{k+1}\cdots t_n$. But then, by induction, $t_i$ is not a face of $t_{i+1}\cdots t_n$, a contradiction, which proves the claim. 

\smallskip

(i) is proved by induction on the length of $\vec{s}$. If $\vec{s}$ is the empty sequence, then the conclusion is evident. Assume 
that $t$ is a face of $\vec{s}\,\vec{t}$ and $\vec{s} = s_0\cdots s_m$. If $t$ is a face of $s_1\cdots s_m\vec{t}$, then we are done by induction. So assume 
that $t$ is not a face of $s_1\cdots s_m\vec{t}$. Then, by the definition of face, we see that $s_0\in t$ and $(t\setminus \{ s_0\})\cup s_0$ is a face of $s_1\cdots s_m \vec{t}$. It follows that $s_0\in {\rm vr}(t_i\cdots t_n)$ for some $i\leq n$ as we assume that $t\subseteq \bigcup_{i\leq n} {\rm vr}(t_i\cdots t_n)$. Since $s_0\not\in {\rm Ur}$, it follows from Lemma~\ref{L:subf}(ii) that, for some $i\leq j\leq n$, $s_0= t_j$ and $s_0$ is a face of $t_{j+1}\cdots t_n$. So, by Claim, the set $(t\setminus \{ s_0\})\cup s_0$ is not a face of $s_1\cdots s_m \vec{t}$ as $s_0$ is included in this set, a contradiction. 

(ii) Assume towards a contradicion that $t_j\subseteq t$ and $t_j\subseteq \bigcup_{j<i\leq n} {\rm vr}(t_{i}\cdots t_n)$. By Lemma~\ref{L:subf}(i), $t_j$ is a face of $\vec{t}$. By (i) of the current lemma, $t_j$ is a face of $t_{j+1}\cdots t_n$ contradicting Claim. 
\end{proof}

\subsection{Grounded simplicial maps between sequences} 

We now define maps among sequences of sets in ${\rm Fin}^+$. We call these maps grounded simplicial maps. The intuition for the introduction of such maps is contained in Proposition~\ref{P:simmad}. 

Given sequences $\vec{t}$ and $\vec{s}$, function
\begin{equation}\label{E:funde} 
f\colon {\rm vr}(\vec{t}\,) \to {\rm vr}(\vec{s}\,)
\end{equation} 
is called {\bf grounded simplicial from $\vec{t}$ to $\vec{s}$} if the following conditions hold: 
\begin{enumerate}
\item[---] for each face $t$ of $\vec{t}$, $f(t)$ is a face of $\vec{s}$ and ${\rm sp}\big(f(t)\big)\subseteq {\rm sp}(t)$;

\item[---] for each face $s$ of $\vec{s}$, there is a face $t$ of $\vec{t}$ with $f(t)=s$ and 
${\rm sp}(s)= {\rm sp}(t)$. 
\end{enumerate}
Strictly speaking the definition above concerns not just the function $f$ but the triple $(f, \vec{t}, \vec{s})$. To emphasize this, we write 
\[
f\colon \vec{t} \to \vec{s}.
\]

We have a proposition that explains the intuition behind the definition above. In the statement, we use the assertion ${\rm Vr}( \vec{t}A)\subseteq {\rm vr}(\vec{t})$ from Proposition~\ref{P:verd}. 

\begin{proposition}\label{P:simmad} 
Let $\vec{t}$ and $\vec{s}$ be sequences and let $f\colon {\rm vr}(\vec{t}\,)\to {\rm vr}(\vec{s}\,)$. Then 
$f$ is a grounded simplicial map from $\vec{t}$ to $\vec{s}$ if and only if 
$f\res {\rm Vr}( \vec{t}A)$ is a grounded simplicial map from $\vec{t}A$ to $\vec{s}A$, for each grounded complex $A$. 
\end{proposition}

\begin{proof} The conclusion is immediate from Proposition~\ref{P:conr} and Lemma~\ref{L:simsp}.
\end{proof}

The next lemma asserts that being a grounded simplicial map between sequences depends only on the combinatorial equivalence classes of the sequences. 

\begin{lemma}\label{L:grsi} 
If $f\colon \vec{t} \to \vec{s}$ is a grounded simplicial map and $\vec{t'}\equiv \vec{t}$ and $\vec{s'}\equiv \vec{s}$, then
$f$ is a grounded simplicial map from $\vec{t'}$ to $\vec{s'}$.
\end{lemma} 

\begin{proof} 
The conclusion is an immediate consequence of  Lemmas~\ref{L:safa} and \ref{L:savr}.   
\end{proof} 

Because of Lemma~\ref{L:grsi}, 
we regard the domain and codomain of grounded simplicial map $f\colon \vec{t}\to \vec{s}$ to be the combinatorial equivalence class of $\vec{t}$ and of $\vec{s}$, respectively.

\subsection{Additive families}

Let $\vec{t}$ be a sequence of sets. A family $S$ of faces of $\vec{t}$ is called {\bf additive with respect to $\vec{t}$}, or simply {\bf additive} if $\vec{t}$ is clear from the context, if, for all $s_1, s_2\in S$, if $s_1\cup s_2$ is a face of $\vec{t}$, then $s_1\cup s_2\in S$. Note that, by Lemma~\ref{L:safa}, if $\vec{t}$ and $\vec{t'}$ are two sequences such that $\vec{t}\equiv \vec{t'}$ and $S$ is additive with respect to $\vec{t}$, then $S$ is additive with respect to $\vec{t'}$. Recall the definition of a non-decreasing enumeration of $S$ from Section~\ref{Su:addsu}.

\begin{lemma}\label{L:dedf} 
Let $\vec{t}$ and $\vec{t'}$ be two sequences of sets with $\vec{t}\equiv \vec{t'}$. Let $S$ be a family additive with respect to $\vec{t}$ and so also with respect to $\vec{t'}$. If $\vec{S}_0$ and $\vec{S}_1$ are non-decreasing enumerations of $S$, then 
\[
\vec{S}_0\, \vec{t} \equiv \vec{S}_1\vec{t'}. 
\]
\end{lemma}

\begin{proof} Note first that by condition (c) in the definition of $\equiv$, it suffices to show the lemma for $\vec{t'}=\vec{t}$. 

By $\#(s)$ we denote the cardinality of $s$. 

We first show that if $\vec{S}=s_0\cdots s_l$ is a nondecreasing enumeration of $S$ and $s_{i_0-1}, s_{i_0}$ are two consecutive entries in it such that 
\begin{equation}\label{E:card}
\#(s_{i_0-1}) \geq \#(s_{i_0}), 
\end{equation}
then the sequence $\vec{S}'$ resulting from switching the positions of $s_{i_0-1}$ and $s_{i_0}$ 
is also a non-decreasing enumeration of $S$ and 
\begin{equation}\label{E:swi}
\vec{S}\, \vec{t} \equiv \vec{S}'\vec{t}.
\end{equation}
To see this observe that 
\begin{equation}\label{E:nont} 
s_{i_0-1}\not\subseteq s_{i_0}\;\hbox{ and }\;s_{i_0}\not\subseteq s_{i_0-1}. 
\end{equation} 
Indeed, the inclusion $s_{i_0-1}\subseteq s_{i_0}$ would lead by \eqref{E:card} to $s_{i_0-1}=s_{i_0}$, contradicting injectivity of the enumeration $\vec{S}$, while the inclusion $s_{i_0}\subseteq s_{i_0-1}$ would contradict the fact that this enumeration $\vec{S}$ is non-decreasing. The first formula in \eqref{E:nont} gives that $\vec{S'}$ is a non-decreasing enumeration of $S$. Now, set $t= s_{i_0-1}\cup s_{i_0}$ and note that, by \eqref{E:nont}, $t$ properly contains both $s_{i_0-1}$ and $s_{i_0}$. So, by additivity of $S$, we have that $t$ is not a face of $\vec{t}$ or $t\in S$, in which case $t$ is listed in $\vec{S}$ to the right of $s_{i_0-1}$ and $s_{i_0}$. Applying the definition of face and Lemma~\ref{L:nonfa}(ii), 
we see that 
in either case, $t$ is not a face of $s_{i_0+1}\cdots s_m \vec{t}$. Thus, we get \eqref{E:swi} from condition (b) in the definition of $\equiv$. 
Furthermore, observe that if $\#(s_{i_0-1}) > \#(s_{i_0})$, then the following quantity strictly increases
\begin{equation}\label{E:quant}
\sum \{ \#(s_{i+1})- \#(s_{i}) \mid i\leq l\; \hbox{ and }\; \#(s_{i}) \leq \#(s_{i+1}) \}.
\end{equation} 
after the switch of $s_{i_0-1}$ and $s_{i_0}$ is performed. 

It follows that by preforming switches as above to maximize \eqref{E:quant}, we can transform a given sequence $\vec{S}$ to a
sequence $\vec{S}' = s_0' \cdots s_m'$ that is a non-decreasing enumeration of $S$, fulfills \eqref{E:swi},
and is such that the function 
\begin{equation}\label{E:funo} 
i\to \#(s_i')\hbox{ is non-decreasing.}
\end{equation}

Now enumerations $\vec{S}_0$ and $\vec{S}_1$ as in the assumption of the lemma can be transformed to enumerations $\vec{S}'_0$ and $\vec{S}'_1$, for which conditions \eqref{E:swi} and \eqref{E:funo} hold. By \eqref{E:funo}, $\vec{S}'_0$ can be transformed to $\vec{S}'_1$ by switches as above since, by the argument above, any pair of consecutive faces of the same cardinality can be switched.
\end{proof}

Lemma~\ref{L:dedf} makes it possible to introduce the following piece of notation. If $S$ is a family additive with respect to $\vec{t}$, we write 
\begin{equation}\label{E:addin}  
S\,\vec{t}
\end{equation} 
for the combinatorial equivalence class of the sequence $\vec{S} \,\vec{t}$ for some, or equivalently, each, non-decreasing enumeration $\vec{S}$ of $S$.

\subsection{Division of grounded simplicial maps}

We define now the division of grounded simplicial maps by sets in ${\rm Fin}^+$. 
We need the following lemma. 

\begin{lemma}\label{L:adco} 
Let $\vec{s}$ and $\vec{t}$ be sequences of sets, let $f\colon \vec{t}\to \vec{s}$ be a simplicial map, and let $s$ be a set in ${\rm Fin}^+$. Consider the family 
\[
f^{-1}(s)= \{ t\mid t\hbox{ a face of }\, \vec{t} \hbox{ and } f(t)=s\}.
\]
\begin{enumerate} 
\item[(i)] $f^{-1}(s)$ is additive with respect to $\vec{t}$.

\item[(ii)] $f^{-1}(s)$ is convex, that is, if $t_1, t_2\in f^{-1}(s)$ and $t_1\subseteq r\subseteq t_2$, then $r\in f^{-1}(s)$. 
\end{enumerate} 
\end{lemma} 

\begin{proof} 
Both points are clear. 
\end{proof} 

Let now $\vec{t}$ and $\vec{s}$ be two sequences of sets, let $f\colon \vec{t}\to\vec{s}$ be a grounded simplicial map, and let $s$ be a 
set in ${\rm Fin}^+$. Define the map 
\[
s f\colon  {\rm vr}\big( f^{-1}(s)\,\vec{t}\,\big) \to {\rm vr}(s\, \vec{s}\,), 
\]
where the domain ${\rm vr}\big( f^{-1}(s)\,\vec{t}\,\big)$ of $sf$ is specified by Lemma~\ref{L:vrst} and Proposition~\ref{P:verd} as 
\[
f^{-1}(s)\cup \big({\rm vr}(\vec{t}\,)\setminus \{ x\mid \{ x\}\in f^{-1}(s)\}\big). 
\]
The map $sf$ is the function that maps each element $t$ of $f^{-1}(s)$ to $s$, and each other element $v$ of the domain of $sf$ to $f(v)$. So, in symbols, we have  
\[
(sf)(x) = f(x), \hbox{ for } x\in f^{-1}(s) \cup \Big( {\rm vr}(\vec{t}\,)\setminus \{ x\mid \{ x\}\in f^{-1}(s)\}\Big).
\]
To clarify the expression above, observe that if $x\in f^{-1}(s)$, then $x$ is a face of $\vec{t}$, so a subset of ${\rm vr}(\vec{t}\,)$, and $f(x)$ above stands for the pointwise image of the set $x$ under $f$; while if $x\in {\rm vr}(\vec{t}\,)\setminus \{ x\mid \{ x\}\in f^{-1}(s)\}$, then $f(x)$ stands for the value of $f$ at $x$. 

Note finally that if $s$ is not a face of $\vec{s}$, then $s\vec{s}\equiv \vec{s}$ and $f^{-1}(s)=\emptyset$. In this case, $sf$ is a function from ${\rm vr}(\vec{t}\,)$ to ${\rm vr}(\vec{s}\,)$ and it is equal to $f$.

The lemma below uses the notation introduced in \eqref{E:addin}.

\begin{lemma}\label{L:divde} 
If $f\colon \vec{t}\to\vec{s}$ is grounded simplicial and $s\in {\rm Fin}^+$, then $sf\colon f^{-1}(s) \, \vec{t}\to s \vec{s}$ is a grounded simplicial map. 
\end{lemma}

\begin{proof} The lemma follows from Proposition~\ref{P:simmad} and Lemma~\ref{L:welgr2}.
\end{proof}

\subsection{Notational conventions}\label{Su:notc}

We introduce two notational conventions. First, we abandon the symbol $\equiv$ and we simply write 
\[
\vec{s} = \vec{t}, \hbox{ for } \vec{s}\equiv \vec{t}. 
\]

Second, often, it will be very convenient to specify a grounded simplicial map $f\colon \vec{t}\to \vec{s}$ by what we call an {\bf assignment}. An assignment is a function $\phi\colon S\to T$, where $S$ and $T$ are sets such that 
\[
\phi\big( {\rm vr}(\vec{t}\,) \cap S\big) \subseteq 
 {\rm vr}(\vec{s}\,)\cap T \;\hbox{ and }\; {\rm vr}(\vec{t}\,)\setminus S = {\rm vr}(\vec{s}\,)\setminus T.
\]
Given such an assignment $\phi$, $f$ is determined by the formulas 
\begin{equation}\label{E:assign}
\begin{split}
f\res \big( {\rm vr}(\vec{t}\,)\cap S\big) &= \phi\res \big( {\rm vr}(\vec{t}\,)\cap S\big)\\
f\res \big( {\rm vr}(\vec{t}\,)\setminus S \big) &= {\rm id}.
\end{split} 
\end{equation} 
Of course, in each case we apply the above approach, we need to check that $f$ defined by \eqref{E:assign} is a grounded simplicial map from $\vec{t}$ to $\vec{s}$.

\section{Relevant categories}

\subsection{The ambient category}\label{Su:agbo}

We define a natural broad category of which all categories relevant to our arguments are subcategories. 
The {\bf objects} of the category are  sequences of sets in ${\rm Fin}^+$ taken up to combinatorial 
equivalence. The {\bf morphisms} are grounded simplicial maps between (combinatorial equivalence classes of) sequences; see Lemma~\ref{L:grsi}. 
Composition of morphisms is simply the composition of functions.
We point out that $f\colon \vec{t}\to \vec{s}$ is an {\bf isomorphism} in this ambient category if 
\begin{enumerate}
\item[---] $f$ is a bijection from ${\rm vr}(\vec{t}\,)$ to ${\rm vr}(\vec{s}\,)$; 

\item[---] $t$ is a face of $\vec{t}$ if and only if $f(t)$ is a face of $\vec{s}$;

\item[---] if $t$ is a face $\vec{t}$, then ${\rm sp}(f(t)) = {\rm sp}(t)$.
\end{enumerate}

Lemma~\ref{L:divde} defines the operation of division on morphisms. We consider the ambient category to be equipped with this division operation.

\subsection{Weld maps and neat weld maps}\label{Su:wene}

Weld maps form the most important to our considerations class of maps. Let $t$ be a set in ${\rm Fin}^+$, and let $x\in t$. This $x$ may be an element of ${\rm Fin}^+$ or of ${\rm Ur}$. 
Assume $t$ is not a vertex of $\vec{t}$. 
Then the map given by the assignment 
\begin{equation}\label{E:ttx}
t \to  x
\end{equation} 
is a morphism 
\[
\hbox{from }\; t\,\vec{t}\;  \hbox{ to }\;  \vec{t}.
\]
We observe that the assumption that $t$ is not a vertex of $\vec{t}$ yields the implication 
\[
t\in {\rm vr}(t\, \vec{t}\,) \;\Rightarrow\; x\in {\rm vr}(\vec{t}\,). 
\]
Indeed, if $t\in {\rm vr}(t\, \vec{t}\,)$, then $t$ is a face of $\vec{t}$ or $t\in {\rm vr}(\vec{t}\,)$, and since the latter possibility is excluded, the former holds implying $x\in {\rm vr}(\vec{t}\,)$. Thus, the assignment \eqref{E:ttx} is well defined. 
We denote the morphism above by 
\[
\pi^{\vec{t}}_{x,t}.
\]
As usual, the domain of this morphism is the combinatorial equivalence class of $t\,\vec{t}$ and the codomain in the combinatorial equivalence class of 
$\vec{t}$.

\begin{lemma}\label{L:grasi} 
Weld maps are grounded simplicial maps. 
\end{lemma}

\begin{proof} The lemma follows Proposition~\ref{P:simmad} and Lemma~\ref{L:welgr2}.
\end{proof}

We will state the main amalgamation theorem in a way that involves a subclass of morphisms of the ambient category. 
The subclass is defined by allowing only compositions of weld maps that are organized in a certain way. 
Let $S$ be an additive family of faces of $\vec{t}$. Let $\iota$ be a function defined on $S$ such that $\iota(s)\in s$ for each $s\in S$. 
Let 
\[
\pi^{\vec{t}}_\iota\colon S\,\vec{t}\to \vec{t}
\] 
be the map defined as follows. Observe that, by Lemma~\ref{L:vrst} and Proposition~\ref{P:verd}, 
\[
{\rm vr}(S\vec{t}\,) = S\cup \big( {\rm vr}(\vec{t}\,)\setminus \{ x\mid \{ x\}\in S\}\big),
\]
and so $\iota$ is a surjection from ${\rm vr}(S\vec{t}\,)$ to ${\rm vr}(\vec{t}\,)$, and we consider $\iota$ as an assignment determining a map $S\,\vec{t}\to \vec{t}$. This map is $\pi^{\vec{t}}_\iota$. We will say that $\iota$ is {\bf based on} the family $S$. We allow the family $S$, on which $\pi_\iota$ is based, to be empty, in which case $\pi^{\vec{t}}_\iota$ is the identity map.

\begin{lemma}\label{L:ioco} 
Let $S$ be an additive family of faces of $\vec{t}$, on which $\iota$ is based. 
The map $\pi^{\vec{t}}_\iota\colon S\,\vec{t}\to \vec{t}$ is a composition of weld maps, so it is a grounded simplicial map. 
\end{lemma} 

\begin{proof} By Proposition~\ref{P:simmad} and Lemma~\ref{L:ioco2}, $\pi^{\vec{t}}_\iota$ is a composition of weld maps. The rest of the conclusion follows from Lemma~\ref{L:grasi}.
\end{proof}

We call a map of the form $\pi^{\vec{t}}_\iota$ a {\bf neat weld} if $\iota$ is based on an upward closed family. 
Clearly, each upward closed family of faces is additive.

\subsection{Two subcategories}

In each category, the objects are all sequences taken up to the combinatorial equivalence. 
Each morphism will be a grounded simplicial map. Of course, in each category morphisms will be closed under composition. 

Below, we define two categories. We only need to specify morphisms in each of them. 
First we define the {\bf weld-division category}, which we will denote by 
\[
{\mathcal D}. 
\]
Morphisms in the weld-division category are obtained by closing under composition and division all weld maps and all isomorphisms. Morphisms in $\mathcal D$ will be called {\bf weld-division maps}. 
It will be convenient to state the main amalgamation theorem in a way that involves a subclass of weld-division maps. 
The subclass is defined by allowing only compositions of weld maps that are organized in a certain way. The second category is the {\bf neat weld category} denoted by 
\[
\mathcal{N}.
\]
The morphisms in $\mathcal N$ are obtained by closing under composition all neat weld maps. 
We observe that ${\mathcal N} \subseteq {\mathcal D}$.

\newpage

\part{Refined projective amalgamation and its proof}\label{P:repr}

\section{Amalgamation theorem for the category $\mathcal D$}

Theorem~\ref{T:fraiab} below is the main theorem on projective amalgamation. 
We state it in its optimal form, that is, with $f$ coming from ${\mathcal N}$. This is done partly to get the strongest result possible, but also it is this optimal 
form that will be used in the applications.

\begin{theorem}\label{T:fraiab} 
For $f',\, g'\in {\mathcal D}$ with the same codomain, there exist 
$f \in {\mathcal D}$ and $g\in {\mathcal N}$ 
such that 
\begin{equation}\label{E:amfrab} 
f'\circ f = g'\circ g. 
\end{equation} 
\end{theorem}

We say that $(f', f)$, $(g', g)$ form a {\bf projective amalgamation} if condition \eqref{E:amfrab} 
in the conclusion of Theorem~\ref{T:fraiab} hold.

\section{More weld-division maps}

\subsection{Combinatorial isomorphisms} 

We single out four types of isomorphisms in the ambient category; see Lemma~\ref{L:types} for a justification that these are, in fact, isomorphisms. They are given by combinatorial descriptions according to the conventions in Section~\ref{Su:notc}. Such descriptions will be used in the computations later in the paper. As before, the domains and codomains of the maps defined below are combinatorial equivalence classes. 

We illustrate the definitions with schematic drawings. In the drawings, we assume that the set of urelements is ${\rm Ur}=\{ a, b, c\}$. The isomorphisms in the drawings map circled vertices to circled vertices and squared ones to squared ones. The edges internal to the triangles are marked with strokes. The number of strokes indicates the order in which the edges are produced in the process of iterated division.  

\medskip

{\bf Type 1.} 

Let $r, s, t$ be sets in ${\rm Fin}^+$ such that $r, s\subseteq t$, $r\cup s\not=\emptyset$, and $r\cap s=\emptyset$. 
Assume that $t$ is not a vertex of $\vec{t}$. Then the map given by the assignment 
\begin{equation}\label{E:ast1}
t\to s\cup \{ t\}, \; r\cup \{ t\}\to t
\end{equation} 
is a morphism 
\[
\hbox{from }\; \big(r\cup \{ t\}\big) (r\cup s) \,t\, \vec{t}\;\hbox{ to }\; \big(s\cup \{ t\}\big) (r\cup s) \,t\, \vec{t}. 
\] 
We denote the above morphism by $\alpha^{\vec{t}}_{r,s,t}$.

As in the case of weld maps, the assumption that $t$ is not a vertex of $\vec{t}$ gives, by Theorem~\ref{T:ord3}(i) and Proposition~\ref{P:verd}, the implications 
\[
\begin{split}
t\in {\rm vr}\big(\big(r\cup \{ t\}\big) (r\cup s) \,t\, \vec{t}\,\big) \;&\Rightarrow\; s\cup \{ t\}\in {\rm vr}\big(  \big(s\cup \{ t\}\big) (r\cup s) \,t\, \vec{t}\,\big)\\
r\cup \{ t\}   \in {\rm vr}\big(\big(r\cup \{ t\}\big) (r\cup s) \,t\, \vec{t}\,\big) \;&\Rightarrow\; t \in {\rm vr}\big(  \big(s\cup \{ t\}\big) (r\cup s) \,t\, \vec{t}\,\big),
\end{split} 
\]
so the assignment \eqref{E:ast1} is well defined. 

The following drawing illustrates a type 1 isomorphism according to the rules from the beginning of this section.

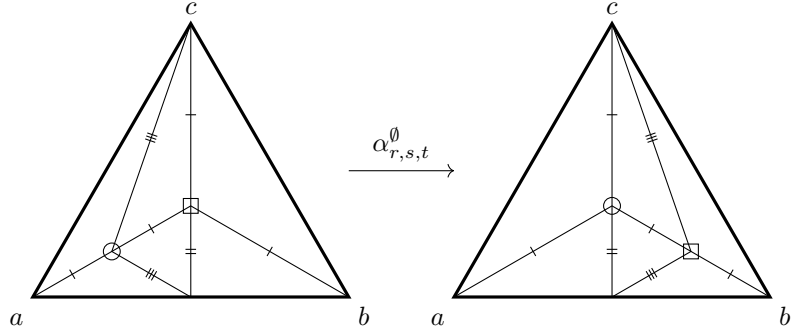
\begin{figure}[htb]
  \centering
  \begin{center}
    \begin{tikzpicture}[scale=1.4]
      \begin{scope}
        \coordinate (A) at (0,0);
        \coordinate (B) at (\side,0);
        \coordinate (C) at (\side/2,{(sqrt(3)/2)*\side});
        \coordinate (D) at (barycentric cs:A=1,B=1,C=1);;
        \coordinate (E) at ($(A)!0.5!(D)$);
        \coordinate (F) at (\side/2,0);

        \draw[line width=\lwa] (A) -- (B) -- (C) -- cycle;
        \draw (A) -- (D) -- (B);
        \draw (C) -- (D) -- (F);
        \draw (C) -- (E) -- (F);

        \node[below left] at (A) {\(a\vphantom{b}\)};
        \node[below right]  at (B) {\(b\)};
        \node[above]  at (C) {\(c\)};

        \tkzMarkSegment[size=2,pos=.5,mark=](A,C)
        \tkzMarkSegment[size=2,pos=.5,mark=](C,B)
        \tkzMarkSegment[size=2,pos=.5,mark=](A,F)
        \tkzMarkSegment[size=2,pos=.5,mark=](F,B)
        \tkzMarkSegment[size=2,pos=.5,mark=|](C,D)
        \tkzMarkSegment[size=2,pos=.5,mark=|](D,B)
        \tkzMarkSegment[size=2,pos=.5,mark=|](D,E)
        \tkzMarkSegment[size=2,pos=.5,mark=|](A,E)
        \tkzMarkSegment[size=2,pos=.5,mark=||](D,F)
        \tkzMarkSegment[size=2,pos=.5,mark=|||](E, C)
        \tkzMarkSegment[size=2,pos=.5,mark=|||](E, F)

        \draw (D) \Square{\crb};
        \draw (E) circle (0.8mm);
      \end{scope}
      \begin{scope}[xshift=4cm]
        \coordinate (A) at (0,0);
        \coordinate (B) at (\side,0);
        \coordinate (C) at (\side/2,{(sqrt(3)/2)*\side});
        \coordinate (D) at (\side/2,{(sqrt(3)/6)*\side});
        \coordinate (E) at ($(B)!0.5!(D)$);
        \coordinate (F) at (\side/2,0);

        \draw[line width=\lwa] (A) -- (B) -- (C) -- cycle;
        \draw (A) -- (D) -- (B);
        \draw (C) -- (D) -- (F);
        \draw (C) -- (E) -- (F);

        \node[below left] at (A) {\(a\vphantom{b}\)};
        \node[below right]  at (B) {\(b\)};
        \node[above]  at (C) {\(c\)};

        \tkzMarkSegment[size=2,pos=.5,mark=](A,C)
        \tkzMarkSegment[size=2,pos=.5,mark=](C,B)
        \tkzMarkSegment[size=2,pos=.5,mark=](A,F)
        \tkzMarkSegment[size=2,pos=.5,mark=](F,B)
        \tkzMarkSegment[size=2,pos=.5,mark=|](C,D)
        \tkzMarkSegment[size=2,pos=.5,mark=|](D,A)
        \tkzMarkSegment[size=2,pos=.5,mark=|](D,E)
        \tkzMarkSegment[size=2,pos=.5,mark=|](B,E)
        \tkzMarkSegment[size=2,pos=.5,mark=||](D,F)
        \tkzMarkSegment[size=2,pos=.5,mark=|||](E, C)
        \tkzMarkSegment[size=2,pos=.5,mark=|||](E, F)

        \draw (E) \Square{\crb};
        \draw (D) circle (0.8mm);
      \end{scope}
      \draw[->] (\side,1.2) -- node[pos=0.5,anchor=south]{\(\alpha^\emptyset_{r,s,t}\)} (\side+1,1.2);
    \end{tikzpicture}
  \end{center}
  \caption{A type 1 isomorphism with \(r = \{a\}\), \(s = \{b\}\), \(t = \{a,b,c\}\) }
  \label{fig:typeI-isomorphism}
\end{figure}

{\bf Type 2.} 
Let $s, t$ be sets in ${\rm Fin}^+$. 
Assume that $s, t$ are not vertices of $\vec{t}$. Then the map given by the assignment 
\begin{equation}\label{E:ast2}
(s\setminus t)\cup \{ t \} \to  (t\setminus s)\cup \{ s \}  
\end{equation} 
is a morphism 
\[
\hbox{ from } \big( (s\setminus t)\cup \{ t \} \big) \, s\, t \,\vec{t} \;\hbox{ to }\; \big( (t\setminus s)\cup \{ s\}\big) \,t\, s \,\vec{t}. 
\]
We denote the map above by $\beta^{\vec{t}}_{s,t}$. 

Again as above we see that the assumption that $s$ and $t$ are not vertices of $\vec{t}$ implies, by Theorem~\ref{T:ord}(i) and Proposition~\ref{P:verd}, that 
\[
(s\setminus t)\cup \{ t \}\in {\rm vr}\big(  \big( (s\setminus t)\cup \{ t \} \big) \, s\, t \,\vec{t}\,\big)\; \Rightarrow\; 
(t\setminus s)\cup \{ s \}  \in {\rm vr}\big(  \big( (t\setminus s)\cup \{ s\}\big) \,t\, s \,\vec{t}\big),
\]
and so assignment \eqref{E:ast2} is well defined. 

The following drawing illustrates a type 2 isomorphism. 

\begin{figure}[htb]
  \centering
  \begin{center}
    \begin{tikzpicture}[scale=1.4]
      \begin{scope}
        \coordinate (A) at (0,0);
        \coordinate (B) at (\side,0);
        \coordinate (C) at (\side/2,{(sqrt(3)/2)*\side});
        \coordinate (D) at ($(A)!0.5!(C)$);
        \coordinate (E) at ($(B)!0.5!(C)$);
        \coordinate (F) at ($(A)!0.5!(E)$);

        \draw[line width=\lwa] (A) -- (B) -- (C) -- cycle;
        \draw (D) -- (E) -- (F) -- cycle;
        \draw (A) -- (F) -- (B);

        \node[below left] at (A) {\(a\vphantom{b}\)};
        \node[below right]  at (B) {\(b\)};
        \node[above]  at (C) {\(c\)};

        \tkzMarkSegment[size=2,pos=.5,mark=](A,D)
        \tkzMarkSegment[size=2,pos=.5,mark=](D,C)
        \tkzMarkSegment[size=2,pos=.5,mark=](C,E)
        \tkzMarkSegment[size=2,pos=.5,mark=](E,B)
        \tkzMarkSegment[size=2,pos=.5,mark=](A,B)
        \tkzMarkSegment[size=2,pos=.5,mark=|](A,F)
        \tkzMarkSegment[size=2,pos=.5,mark=|](F,E)
        \tkzMarkSegment[size=2,pos=.5,mark=||](D,E)
        \tkzMarkSegment[size=2,pos=.5,mark=|||](D,F)
        \tkzMarkSegment[size=2,pos=.5,mark=|||](F,B)

        \draw (F) \Square{\crb};

      \end{scope}
      \begin{scope}[xshift=4cm]
        \coordinate (A) at (0,0);
        \coordinate (B) at (\side,0);
        \coordinate (C) at (\side/2,{(sqrt(3)/2)*\side});
        \coordinate (D) at ($(A)!0.5!(C)$);
        \coordinate (E) at ($(B)!0.5!(C)$);
        \coordinate (F) at ($(D)!0.5!(B)$);

        \draw[line width=\lwa] (A) -- (B) -- (C) -- cycle;
        \draw (D) -- (E) -- (F) -- cycle;
        \draw (A) -- (F) -- (B);

        \node[below left] at (A) {\(a\vphantom{b}\)};
        \node[below right]  at (B) {\(c\vphantom{b}\)};
        \node[above]  at (C) {\(b\)};

        \tkzMarkSegment[size=2,pos=.5,mark=](A,D)
        \tkzMarkSegment[size=2,pos=.5,mark=](D,C)
        \tkzMarkSegment[size=2,pos=.5,mark=](C,E)
        \tkzMarkSegment[size=2,pos=.5,mark=](E,B)
        \tkzMarkSegment[size=2,pos=.5,mark=](A,B)
        \tkzMarkSegment[size=2,pos=.5,mark=|](B,F)
        \tkzMarkSegment[size=2,pos=.5,mark=|](F,D)
        \tkzMarkSegment[size=2,pos=.5,mark=||](D,E)
        \tkzMarkSegment[size=2,pos=.5,mark=|||](E,F)
        \tkzMarkSegment[size=2,pos=.5,mark=|||](F,A)

        \draw (F) \Square{\crb};

      \end{scope}
      \draw[->] (\side,1.2) -- node[pos=0.5,anchor=south]{\(\beta^\emptyset_{s,t}\)} (\side+1,1.2);
    \end{tikzpicture}
  \end{center}
  \caption{A type 2 isomorphism with \(s = \{a,c\}\), \(t = \{b,c\}\) }
  \label{fig:typeII-isomorphism}
\end{figure}
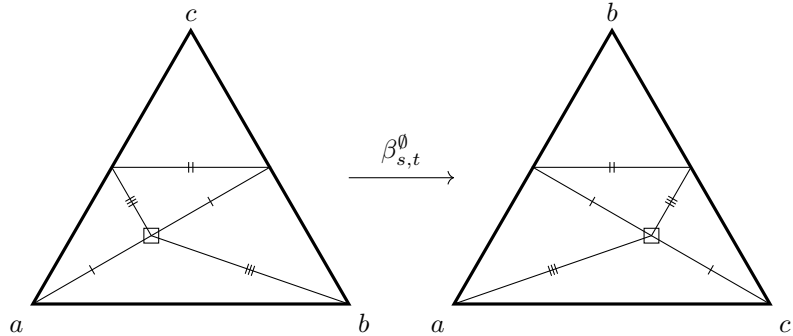

We define now two types of morphisms that essentially rename a vertex. They are given through assignments \eqref{E:ast3} and \eqref{E:ast4}. To see that these assignments are well defined, note that, for $x\in {\rm Fin}^+\cup {\rm Ur}$, under the assumption that $\{ x\}$ is not a vertex of $\vec{t}$, we have 
\[
x \in {\rm vr}\big( \vec{t}\,\big) \;\Leftrightarrow \; \{ x\} \in {\rm vr}(\{ x\} \, \vec{t}\,). 
\]

\medskip

{\bf Type 3a.} 

Let $x$ be in ${\rm Fin}^+\cup {\rm Ur}$. Assume that $\{ x\}$ is not a vertex of $\vec{t}$. 
Then the map given by the assignment 
\begin{equation}\label{E:ast3}
x \to  \{ x \}  
\end{equation} 
is a morphism 
\[
\hbox{ from } \vec{t} \;\hbox{ to }\; \{ x\}\, \vec{t}. 
\]

\medskip

{\bf Type 3b.} 

Let $x$ be in ${\rm Fin}^+\cup {\rm Ur}$. Assume that $\{ x\}$ is not a vertex of $\vec{t}$. 
Then the map given by the assignment 
\begin{equation}\label{E:ast4} 
\{ x\} \to  x
\end{equation} 
is a morphism 
\[
\hbox{ from } \{ x\} \;\vec{t} \;\hbox{ to }\;  \vec{t}. 
\]

\medskip
We denote maps of type 3a and 3b, respectively, by $\delta^{x, \vec{t}}\;\hbox{ and }\; \delta^{\vec{t}}_{\{ x\}}$.
Note that the map of type 3b is a weld map, but it will be convenient to have an additional name for this very particular kind of weld maps. 

The following drawing illustrates these two types of isomorphisms. 

\begin{figure}[htb]
  \centering
  \begin{center}
    \begin{tikzpicture}[scale=1.4]
      \begin{scope}
        \coordinate (A) at (0,0);
        \coordinate (B) at (\side,0);
        \coordinate (C) at (\side/2,{(sqrt(3)/2)*\side});

        \draw[line width=\lwa] (A) -- (B) -- (C) -- cycle;
        \node[below left] at (A) {\(a\vphantom{\{}\)};
        \node[below right]  at (B) {\(b\vphantom{\{}\)};
        \node[above]  at (C) {\(c\)};

        \tkzMarkSegment[size=2,pos=.5,mark=](A,B)
        \tkzMarkSegment[size=2,pos=.5,mark=](A,C)
        \tkzMarkSegment[size=2,pos=.5,mark=](B,C)

        \draw (A) \Square{\crb};

      \end{scope}
      \begin{scope}[xshift=4cm]
        \coordinate (A) at (0,0);
        \coordinate (B) at (\side,0);
        \coordinate (C) at (\side/2,{(sqrt(3)/2)*\side});

        \draw[line width=\lwa] (A) -- (B) -- (C) -- cycle;
        \node[below left, xshift=1mm] at (A) {\(\{a\}\vphantom{\{}\)};
        \node[below right]  at (B) {\(b\vphantom{\{}\)};
        \node[above]  at (C) {\(c\)};

        \tkzMarkSegment[size=2,pos=.5,mark=](A,B)
        \tkzMarkSegment[size=2,pos=.5,mark=](A,C)
        \tkzMarkSegment[size=2,pos=.5,mark=](B,C)

        \draw (A) \Square{\crb};

      \end{scope}
      \draw[->] (\side,1.5) -- node[pos=0.5,anchor=south]{\(\delta^{a, \emptyset}\)} (\side+1,1.5);
      \draw[<-] (\side,0.9) -- node[pos=0.5,anchor=south]{\(\delta^\emptyset_{\{a\}}\)} (\side+1,0.9);
    \end{tikzpicture}
  \end{center}
  \caption{Type 3a and 3b isomorphisms}
  \label{fig:type3a3b-isomorphisms}
\end{figure}
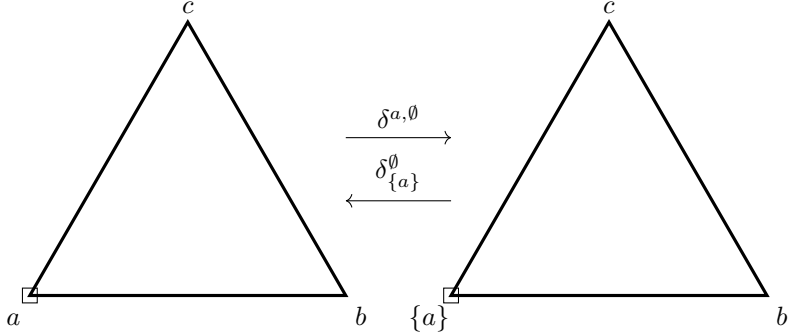

\begin{lemma}\label{L:types}
\begin{enumerate}
\item[(i)] Morphisms of types 1--3 are grounded simplicial maps. 

\item[(ii)] Morphisms of types 1--3 are isomorphisms in the ambient category; in fact, 
\[
\alpha^{\vec{t}}_{r,s,t}\circ \alpha^{\vec{t}}_{s,r,t} = {\rm id},\;\; \beta^{\vec{t}}_{s,t}\circ \beta^{\vec{t}}_{t,s}={\rm id},\;\;
\delta^{\vec{t}}_{\{ x\}}\circ \delta^{x,\vec{t}} = \delta^{x,\vec{t}}\circ \delta_{\{ x\}}^{\vec{t}}={\rm id}. 
\]
\end{enumerate} 
\end{lemma} 

\begin{proof} (i) for types 1 and 2 follows from Proposition~\ref{P:simmad} and Theorems~\ref{T:ord3}(ii) and \ref{T:ord}(ii), respectively. 
The assertion for type 3(b) is a special case of Lemma~\ref{L:grasi}. Handling type 3(a) is easy and is left to the reader.   

(ii) is checked using the defining formulas of the maps. 
\end{proof}

It will be convenient to introduce the following notion. The class of {\bf combinatorial isomorphisms} is the smallest class obtained by closing combinatorial isomorphisms of type 1, 2, 3a, and 3b under composition and division.

\subsection{Defining more weld-division maps}\label{S:produ}

The goal of this section is to prove Lemmas~\ref{L:commut} and \ref{L:org6}---the first one of the two produces new combinatorial isomorphisms, the second one produces weld-division maps. But we start with a lemma that will ease certain computations.

Given a grounded simplicial map $f\colon \vec{t}\to \vec{s}$, define 
\begin{equation}\label{E:defuu} 
{\rm u}(f) = \{ v\in {\rm vr}(\vec{s}\,)\mid v=f(w)\hbox{ for a unique } w\in {\rm vr}(\vec{t}\,)\}. 
\end{equation} 
This definition is used in Lemma~\ref{L:mapdif} and in the proof of Lemma~\ref{L:org6}. It will also be important later in the paper.

\begin{lemma}\label{L:mapdif} 
Let $f\colon \vec{t}\to \vec{s}$ be a grounded simplicial map. 
Let $r_0\cdots r_l$ be a non-decreasing sequence of faces of $\vec{t}$ such that $f(r_i)\subseteq {\rm u}(f)$, for each $i\leq l$. Then the domain and codomain of the grounded simplicial map 
$\big( f(r_0)\cdots f(r_l) \big) f$ are 
\[
r_0\cdots r_l\,\vec{t}\;\hbox{ and }\; f(r_0)\cdots f(r_l)\, \vec{s},
\]
respectively, and the map is given by the assignment 
\[
r_i\to f(r_i), \hbox{ for }i\leq l, \;\hbox{ and }\; {\rm vr}(\vec{t}\,) \ni x \to f(x)\in {\rm vr}(\vec{s}\,).
\]
\end{lemma} 

\begin{proof} First we make the following general observation. If $f\colon \vec{t}\to\vec{s}$ is a grounded simplicial map and $s$ is a set in ${\rm Fin}^+$, then 
\begin{equation}\label{E:exin} 
{\rm u}(f)\cap {\rm vr}(s\vec{s}) \subseteq {\rm u}(sf). 
\end{equation} 
Indeed, only vertices of ${\rm vr}\big(f^{-1}(s)\,\vec{t}\,\big)\cap {\rm vr}(\vec{t}\,)$ are mapped by $sf$ to vertices of ${\rm vr}(s\vec{s}\,)\cap {\rm vr}(\vec{s}\,)$. Further, 
\[
u(f)\cap {\rm vr}(s\vec{s}) \subseteq {\rm vr}(s\vec{s}\,)\cap {\rm vr}(\vec{s}\,)
\]
 $sf$ is equal to $f$ on ${\rm vr}\big(f^{-1}(s)\,\vec{t}\,\big)\cap {\rm vr}(\vec{t}\,)$. 

The conclusion of the lemma is obtained by an easy induction from the following statement:

\noindent {\em Let $f\colon \vec{t}\to \vec{s}$ be a grounded simplicial map, and let $r_1$ a face of $\vec{t}$ with $f(r_1)\subseteq {\rm u}(f)$. 
Set $g= (f(r_1)) f$. Then the domain and codomain of $g$ are $r_1\vec{t}$ and $f(r_1)\vec{s}$, respectively, and $g$ is given by the assignment 
\begin{equation}\label{E:raff} 
r_1\to f(r_1) \;\hbox{ and }\; {\rm vr}(\vec{t}\,) \ni x \to f(x).
\end{equation} 
Additionally, if $r_0$ is a face of $\vec{t}$ with $r_1\not\subseteq r_0$ and $f(r_0)\subseteq u(f)$, then  $r_0$ is a face of $r_1\vec{t}$ and $g(r_0)\subseteq {\rm u}(g)$.}

So, it suffices to prove the statement above. Note that $f(r_{1})\subseteq {\rm u}(f)$ implies that $r_{1}$ is the only face $t$ of $\vec{t}$ with $f(t) =f(r_{1})$, that is, 
$f^{-1}\big( f(r_{1}) \big) = \{ r_{1}\}$. Now the first sentence of the conclusion follows directly from the definition of the division of grounded simplicial maps. Additionally, $r_0$ is a face of $r_1\vec{t}$ since it is a face of $\vec{t}$ and $r_1\not\subseteq r_0$. Obviously, we have 
\begin{equation}\label{E:bdaa}
g(r_0)\subseteq {\rm vr}\big((f(r_1)) \vec{s}\,\big).
\end{equation}
Using \eqref{E:raff} and taking into account that $r_1\not\in r_0$, we get 
\begin{equation}\label{E:bdaaa}
g(r_0) =f(r_0)\subseteq {\rm u}(f). 
\end{equation} 
We obtain $g(r_0)\subseteq {\rm u}(g)$ from \eqref{E:exin}, \eqref{E:bdaa} and \eqref{E:bdaaa}. 
\end{proof}

The next lemma constructs new combinatorial isomorphisms.

\begin{lemma}\label{L:commut} 
Let $\vec{q}$ be a sequence, let $t$ be a face of $\vec{q}$, and let $\emptyset\not= s\subseteq t$. 
\begin{enumerate}
\item[(i)] The assignment 
\[
t\to \big( (t\setminus s)\cup \{ s\}\big)
\]
induces a combinatorial isomorphism 
\[
\hbox{from }\;s\, t\, \vec{q}\;\hbox{ to }\; \big( (t\setminus s)\cup \{ s\}\big) \, s\, \vec{q}.
\]

\item[(ii)]
If $r\subseteq s$, then the assignment 
\[
t \to \big( (t\setminus r)\cup \{ s \}\big) , \; \big( r\cup \{ t \}\big) \to \big( (t\setminus s)\cup \{ s\}\big)
\]
induces a combinatorial isomorphism 
\begin{equation}\notag
\hbox{from }\; \Big( r\cup \{ t \} \Big) \, s\, t \,\vec{q} \;\hbox{ to }\; \Big( (t\setminus s)\cup \{ s \} \Big)\, \Big((t\setminus r)\cup \{ s\}\Big) \, s\, \vec{q}.
\end{equation} 
\end{enumerate}
\end{lemma}

\begin{proof} 
(i) Note that $\{ t\}$ is not a vertex of $s\,t\, \vec{q}$ since $\{ t\}\not= s, t$ (as $t\not\in t$) and $\{ t\}$ is not a vertex of $\vec{q}$ since $t$ is a face of $\vec{q}$ (so $t\not\in {\rm tc}\big( {\rm vr}(\vec{q}\,)\big)$). So, the sequences 
\[
s\,t\, \vec{q}\;\hbox{ and }\; \{ t \}\, s\, t\, \vec{q}
\]
are combinatorially isomorphic by the type 3 assignment $t\to \{ t\}$. 
Since $s$ and $t$ are faces of $\vec{q}$, they are not vertices of that sequence. So, under the assumption $s\subseteq t$, we have a type 2 isomorphism between 
\[
\{ t \}\, s\, t \,\vec{q} \;\hbox{ and }\; \big( (t\setminus s)\cup \{ s \} \big) t\, s\, \vec{q} 
\]
via the assignment $\{ t \}\to (t\setminus s)\cup \{ s \}$.  
Further, since $s\subseteq t$, we see that $t$ is not a face of $s\,\vec{q}$ and, therefore, 
\[
\big( (t\setminus s)\cup \{ s \} \big) t\, s\, \vec{q} =  \big( (t\setminus s)\cup \{ s \} \big) s\, \vec{q}. 
\]
Putting together this equality and the two combinatorial isomorphisms above, we get the conclusion.

(ii) If $r=\emptyset$, then the two sequences in the conclusion become
\[
\{ t \}  \, s\, t \,\vec{q} \;\hbox{ and }\; \big( (t\setminus s)\cup \{ s \} \big)\, \big(t \cup \{ s\}\big) \, s\, \vec{q}.
\]
As in the proof of (i), $\{ t \}  \, s\, t \,\vec{q}$ and $s\, t \,\vec{q}$ are combinatorially isomorphic and $t \cup \{ s\}$ is not a face of $s\, \vec{q}$ since $s\subseteq t$, so 
\[
\big( (t\setminus s)\cup \{ s \} \big)\, \big(t \cup \{ s\}\big) \, s\, \vec{q} = \big( (t\setminus s)\cup \{ s \} \big)\, s\, \vec{q}.
\]
Thus, (ii) for $r=\emptyset$ follows directly from (i). 

Assume $r\not=\emptyset$, and set 
\[
v=(t\setminus s)\cup \{ s\}.
\]

Since $t$ is not a vertex of $\vec{q}$, 
\[
f\colon \big( r\cup \{ t \} \big) \, s\, t\, \vec{q} \to \big( (s\setminus r)\cup \{ t \} \big) \, s\, t\, \vec{q} 
\]
given by the assignment 
\begin{equation}\label{E:raz}
t\to \big(s\setminus r)\cup \{ t \} \;\hbox{ and }\;  r\cup \{ t \} \to t
\end{equation} 
is a  type 1 isomorphism. 

Next, we use point (i) to see that 
\[
g\colon s\,t\,\vec{q}\to v\, s\, \vec{q}
\]
given by the assignment 
\[
t\to v 
\] 
is a combinatorial isomorphism. So, 
\[
u(g)= {\rm vr}\big( v\, s\, \vec{q}\,\big).
\]
Also $(s\setminus r)\cup \{ t \}$ is a face of $s\, t\, \vec{q}$ since $(s\setminus r)\cup t = t$ is a face of $\vec{q}$ and $s\not\subseteq s\setminus r$ (since $r\not=\emptyset$). 
Thus, by Lemma~\ref{L:mapdif}, since $v= g(t)$, we have that 
\[
g'\colon \big( (s\setminus r)\cup \{ t \} \big) \, s\, t\, \vec{q} \to \big( (s\setminus r)\cup \{ v \} \big)\, 
v\, s\, \vec{q},  
\]
is a combinatorial isomorphism given by the assignment 
\begin{equation}\label{E:dwa} 
t\to v\;\hbox{ and }\; (s\setminus r)\cup \{ t \} \to  (s\setminus r)\cup \{ v \} .
\end{equation} 

Finally, observe that $(t\setminus r)\cup \{ s\}$ is a face of $s\,\vec{q}$ and $v \subseteq (t\setminus r)\cup \{ s\}$. 
Thus, we are allowed to use (i) again (with $v$ playing the role of $s$, $(t\setminus r)\cup \{ s\}$ playing the role of $t$, and $s\,\vec{q}$ playing the role of $\vec{q}\,$) to see that 
\[
h\colon v\, \big((t\setminus r)\cup \{ s\}\big) \, \big( s\,\vec{q}\,\big)\to 
 \big( (s\setminus r)\cup \{ v \} \big)\,v \, \big( s\,\vec{q}\,\big). 
\]
given by the assignment 
\begin{equation}\label{E:trz}
(t\setminus r)\cup \{ s\} \to (s\setminus r)\cup \{ v \} 
\end{equation} 
is a combinatorial isomorphism. 

The composition $h^{-1}\circ g'\circ f$ is a combinatorial isomorphism and, by composing the assignments \eqref{E:raz}, \eqref{E:dwa}, and the inverse of \eqref{E:trz}, we see that it is equal to the map in the conclusion of (ii).
\end{proof}

We have the following lemma identifying a weld-division map.

\begin{lemma}\label{L:org6} 
Let $\vec{q}$ be a sequence.  Let $\vec{p}$ be a sequence of faces of $\vec{q}$ and,  for $i=1, \dots, m$, let 
$r_i, s_{i1}, \dots, s_{in_i}$ be faces of $\vec{q}$ such that $r_i \subseteq s_{ij}$, for all $j\leq n_i$. Assume that 
the sequence 
\begin{equation}\label{E:weln} 
\vec{p}\, (r_1s_{11} \cdots s_{1n_1}) \cdots  (r_ms_{m1} \cdots s_{mn_m})
\end{equation} 
is nondecreasing.
Then $s_{ij}\not= s_{i'j'}$, if $i\not= i'$ or $j\not= j'$, and the map 
\begin{equation}\notag 
\vec{p}\,(r_1s_{11} \cdots s_{1n_1}) \cdots  (r_ms_{m1} \cdots s_{mn_m})\, \vec{q} \to \vec{p}\,r_1\cdots r_m\vec{q}
\end{equation} 
given by the assignment 
\begin{equation}\label{E:wela} 
s_{ij}\to r_i,\;\hbox{ for }i\leq m\hbox{ and }j\leq n_i, 
\end{equation} 
is a weld-division map. 
\end{lemma}

\begin{proof} First we observe that the condition $s_{ij}\not= s_{i'j'}$ if $i\not= i'$ or $j\not= j'$, follows from the sequence \eqref{E:weln} being nondecreasing. 
This condition implies that the assignment \eqref{E:wela} is well defined.

Since weld-division maps are closed under composition, it will suffice to show that, for a fixed $i\leq m$, the map 
\begin{equation}\label{E:later}
\vec{p'}\,r_is_{i1} \cdots s_{in_i}\vec{q'}\to \vec{p'} r_i\vec{q'}
\end{equation}
given by \eqref{E:wela} for the fixed $i$, is a weld-division map, where 
\[
\vec{p'} = \vec{p}\, r_1\cdots r_{i-1}\;\hbox{ and }\; \vec{q'} = (r_{i+1}s_{i+11} \cdots s_{i+1n_{i+1}}) \cdots  (r_ms_{m1} \cdots s_{mn_m})\, \vec{q}.
\]

Now, fix also $j\leq n_i$.  Observe that all entries of $\vec{p'}$ are faces of $r_is_{i1} \cdots s_{in_i}\vec{q'}$ since they are all faces of $\vec{q}$ and no entry of 
\[
(r_is_{i1}\cdots s_{in_i}) (r_{i+1}s_{i+11} \cdots s_{i+1n_{i+1}}) \cdots  (r_ms_{m1} \cdots s_{mn_m})
\]
is included in an entry of $\vec{p'}$ by the sequence \eqref{E:weln} being nondecreasing. We show that the map 
\begin{equation}\label{E:fixj}
\vec{p'}s_{i1}'\cdots s_{ij-1}' r_is_{ij} \cdots s_{in_i}\vec{q'}\to \vec{p'}s_{i1}'\cdots s_{ij}' r_is_{ij+1} \cdots s_{in_i}\vec{q'}
\end{equation} 
given by $s_{ij}\to s_{ij}'$ is a weld-division map, where 
\[
s_{ij}'= (s_{ij}\setminus r_i)\cup \{ r_i\},\; j=1, \dots, n_i, 
\]
and that all entries of $\vec{p'}s_{i1}'\cdots s_{ij-1}'$ are faces of $r_is_{ij} \cdots s_{in_i}\vec{q'}$. 
Indeed, the assumption that \eqref{E:weln} is nondecreasing implies that $r_i$ and $s_{ij}$ are faces of the sequence $s_{ij+1}\cdots s_{in_i} \vec{q'}$, so by Lemma~\ref{L:commut}(i), the map 
\[
f\colon r_is_{ij} \cdots s_{in_i}\vec{q'}\to s_{ij}' r_is_{ij+1} \cdots s_{in_i}\vec{q'}
\]
given by the assignment $s_{ij}\to s_{ij}'$ is a combinatorial isomorphism. Since $f$ is an isomorphism, ${\rm u}(f)$ is equal to the whole set 
\[
{\rm vr}\big( s_{ij}' r_is_{ij+1} \cdots s_{in_i}\vec{q'}\big), 
\]
and so we have 
\begin{equation}\label{E:later1} 
t\subseteq {\rm u}(f),\;\hbox{ for each entry } t \hbox{ of the sequence }\vec{p'}s_{i1}'\cdots s_{ij-1}'. 
\end{equation} 
Furthermore, $f$ is equal to the identity on each such set $t$ since $s_{ij}\not\in t$ as both $s_{ij}$ and $t$ are faces of $\vec{q}$. Thus, iteratively dividing $f$ by the entries of the sequence $\vec{p'}s_{i1}'\cdots s_{ij-1}'$ and using Lemma~\ref{L:mapdif}, we see that the map \eqref{E:fixj} is a weld-division map. Note further that $s_{ij}'$ is a face of $r_is_{ij+1}\cdots s_{in_i}\vec{q'}$ since $r_i$ and $\big(s_{ij}'\setminus \{ r_i\}\big)\cup r_i= s_{ij}$ are faces of $s_{ij+1}\cdots s_{in_i}\vec{q'}$. 

Using closure under composition of weld-division maps and the fact that, for each $j\leq n_i$, map \eqref{E:fixj} is a weld-division map, we see that the map 
\begin{equation}\label{E:rsf}
\vec{p'}r_is_{i1} \cdots s_{in_i}\vec{q'}\to \vec{p'} s_{i1}'\cdots s_{in_i}'r_i\vec{q'}
\end{equation}
given by 
\begin{equation}\label{E:ssp}
s_{ij}\to s_{ij}', \;\hbox{ for all }j\leq n_i, 
\end{equation} 
is a weld-division map. 

As observed above, $s_{ij}'$ is a face of $s_{ij+1}'\cdots s_{in_i}' r_i\vec{q'}$ and, clearly, $r_i\in s_{ij}'$. Thus, for each $j\leq n_i$, the map 
\begin{equation}\label{E:ssj}
g\colon s_{ij}' s_{i j+1}'\cdots s_{i n_i}' r_i\vec{q'}\to s_{i j+1}'\cdots s_{i n_i}' r_i \vec{q'}
\end{equation} 
given by the assignment 
\begin{equation}\label{E:sijr} 
s_{ij}'\to r_i
\end{equation} 
is a weld map. Let $t$ be an entry of $\vec{p'}$. Observe that $t$ is a face of $s_{ij}' s_{i j+1}'\cdots s_{i n_i}' r_i\vec{q'}$ as it is a face of $\vec{q'}$ and no entry of the sequence $s_{ij}' s_{i j+1}'\cdots s_{i n_i}' r_i$ is included in $t$. 
Note that, on the one hand, $r_i\not \in t$ as $t$ and $r_i$ are faces of $\vec{q}$, and on the other hand, by \eqref{E:sijr}, $g$ is equal to the identity on $t$ since $s_{ij}'\not\in t$ as $r_i\not\in {\rm tc}(t)$ by $t$ and $r_i$ both being faces of $\vec{q}$. It follows that 
\begin{equation}\label{E:later2}
t\subseteq {\rm u}(g),\;\hbox{ for each entry of the sequence }\vec{p'}.
\end{equation} 
So, by Lemma~\ref{L:mapdif}, the map 
\begin{equation}\label{E:colc2}
\vec{p'} s_{ij}' s_{i j+1}' \cdots s_{i n_i}' r_i\vec{q'}\to \vec{p'} s_{i j+1}'\cdots s_{i n_i}' r_i \vec{q'}
\end{equation} 
given by \eqref{E:sijr} is an iterative division by the entries of $\vec{p'}$ of the map \eqref{E:ssj}; thus, it is a weld-division map. 

We obtain the conclusion that the map \eqref{E:later} is a weld-division map as follows. We compose the $n_i$ weld-division maps \eqref{E:colc2}, which yields a weld-division map 
\[
\vec{p'}s_{i1}' \cdots s_{in_i}' r_i\vec{q'} \to \vec{p'} r_i\vec{q'}.
\]
Then we precompose this resulting map with \eqref{E:rsf} and note that the assignment \eqref{E:wela} is the composition of \eqref{E:ssp} and \eqref{E:sijr}. 
\end{proof}

\subsection{Dividing grounded simplicial maps by additive families of faces} 

Let $f\colon \vec{t}\to \vec{s}$ be grounded simplicial. If $r_1\cdots r_l$ is a sequence of sets in ${\rm Fin}^+$, then let 
\[
r_1\cdots r_l f = r_1(r_2(\cdots r_{l-1}(r_l\,f)\cdots )).
\]
The map above is grounded simplicial by Lemma~\ref{L:divde}.
For a family $S$ of faces of $\vec{s}$, let 
\[
f^{-1}(S) = \{ t\mid t\hbox{ is a face of }\vec{t}\hbox{ and } f(t)\in S\}.
\]

\begin{lemma}\label{L:difsa} 
Let $f\colon \vec{t}\to\vec{s}$. Assume $S$ is an additive set of faces of $\vec{s}$. 
\begin{enumerate} 
\item[(i)] $f^{-1}(S)$ is an additive set of faces of $\vec{t}$. 

\item[(ii)] If $\vec{S}$ is a non-decreasing enumeration of $S$, then the domain and codomain of $\vec{S} f$ are 
\begin{equation}\label{E:dcsf}
\begin{split}
{\rm vr}(S\,\vec{s}\,) &= S \cup \big({\rm vr}(\vec{s}\,)\setminus \{ x\mid \{ x\}\in S\}\big),\\
{\rm vr}(f^{-1}(S)\, \vec{t}\,) &= f^{-1}(S) \cup \big( {\rm vr}(\vec{t}\,)\setminus \{ y\mid \{ y\}\in f^{-1}(S)\}\big),
\end{split}
\end{equation}
respectively, and $\vec{S} f$ is the function 
\begin{equation}\label{E:ssff}
{\rm vr}(f^{-1}(S)\, \vec{t}\,)\ni x\to f(x)\in {\rm vr}(S\,\vec{s}\,).
\end{equation} 

\item[(iii)] For two non-decreasing enumerations $\vec{S}$ and $\vec{S'}$ of $S$, we have $\vec{S}f=\vec{S'}f$. 
\end{enumerate} 
\end{lemma}

To clarify \eqref{E:ssff}, observe that if $x\in f^{-1}(S)$, then $x$ is a face of $\vec{t}$, so a subset of 
${\rm vr}(\vec{t}\,)$, and $f(x)$ stands for the pointwise image of this set under $f$; if $x\in {\rm vr}(\vec{t}\,)\setminus \{ y\mid \{ y\}\in f^{-1}(S)\}$, 
then $f(x)$ stands for the value of $f$ at $x$.

\begin{proof} (i) follows directly from $f(t_1\cup t_2)= f(t_1)\cup f(t_2)$ for all $t_1, t_2\subseteq {\rm vr}(\vec{t})$. 

(ii) Fix a non-decreasing enumeration $\vec{S}= s_0s_1\cdots s_n$ of $S$. We note that $S\setminus \{ s_0\}$ is an additive family of faces of $A$ and $s_1\cdots s_n$ is its non-decreasing enumeration. The statements in this proof are proved by induction assuming that they hold for $S\setminus \{ s_0\}$ and establishing them for $S$. By easy induction, we prove that the codomain and domain of the map $\vec{S} f$ are equal to $S\vec{s}$ and $f^{-1}(S)\,\vec{t}$, respectively, and by Lemma~\ref{L:vrst} and Proposition~\ref{P:verd}, we get formulas \eqref{E:dcsf}. 
Again by induction, we see that $\vec{S} f$ is given by the formula \eqref{E:ssff}. 

(ii) Since formulas \eqref{E:dcsf} and \eqref{E:ssff} depend only on $S$ and not on the enumeration $\vec{S}$, we get the conclusion of (iii). 
\end{proof}

If $f\colon \vec{t}\to\vec{s}$ is grounded simplicial and $S$ is an additive family of faces of $\vec{s}$, then Lemma~\ref{L:difsa}(iii) allows us to define 
\[
Sf
\]
to be the map $\vec{S} f$, where $\vec{S}$ is an arbitrary nondecreasing enumeration of $S$. 

\begin{lemma}\label{L:X} Let $f\colon \vec{t}\to\vec{s}$ and let $g\colon \vec{u}\to\vec{t}$ be grounded simlipcial maps. Let $S$ be an additive family of faces of $\vec{s}$. Then 
\[
S(f\circ g)= Sf\circ \big( f^{-1}(S) g\big).
\]
\end{lemma} 

\begin{proof} We have 
\begin{equation}\notag
(f\circ g)^{-1}(S) = g^{-1}\big( f^{-1}(S)\big). 
\end{equation} 
It follows that $S(f\circ g)$ and $Sf\circ \big( f^{-1}(S) g\big)$ have the same codomain, $S\vec{s}$, and domain given by 
\begin{equation}\label{E:doco}
(f\circ g)^{-1}(S)\,\vec{u}  = g^{-1}\big( f^{-1}(S)\big)\,\vec{u}.
\end{equation} 
For each subset or element $x$ of 
\[
{\rm vr}\Big( (f\circ g)^{-1}(S)\,\vec{u}\,\Big) = {\rm vr}\Big( g^{-1}\big( f^{-1}(S)\big)\,\vec{u}\Big),
\]
we obviously have $(f\circ g)(x) = f(g(x))$, and the conclusion of the current lemma follows from Lemma~\ref{L:difsa}(ii). 
\end{proof}

\section{Proof of amalgamation for $\mathcal D$ from a coinitiality theorem for $\mathcal D$}\label{S:ampr}

In this section, we state a coinitiality result, Theorem~\ref{T:clsu}, from which the amalgamation result, Theorem~\ref{T:fraiab}, will be derived. Note, however, that Theorem~\ref{T:clsu} is a special case of Theorem~\ref{T:fraiab}. We also provide the proof of this implication. Consequently, Theorem~\ref{T:clsu} will be what will remain to be shown.

Define the {\bf weld category} denoted by 
\[
{\mathcal W}. 
\]
Morphisms of the weld category are obtained by closing the set of all weld maps under composition.

\begin{theorem}\label{T:clsu} 
Weld category ${\mathcal W}$ is coinitial in the weld-division category ${\mathcal D}$, that is, 
for each $g'\in {\mathcal D}$, there exists $g\in {\mathcal D}$ such that 
\[
g'\circ g\in {\mathcal W}.
\]
\end{theorem}

We start with some amalgamation lemmas. 
Recall that pairs $(f', f),\, (g', g)$, where $f', f, g', g$ are weld-division maps, form an amalgamation if 
\begin{equation}\label{E:gfpr}
 f'\circ f  = g'\circ g. 
\end{equation}

Lemma~\ref{L:ama} below gives the base case of the amalgamation theorem. 

\begin{lemma}\label{L:ama} 
Let $\pi$ be a weld map and $g'\in {\mathcal D}$ with $\pi$ and $g'$ having the same codomain. 
Then there exist a neat weld map $\pi_\iota$ and $f\in {\mathcal D}$ such that 
$(\pi, f),\, (g', \pi_\iota)$ is an amalgamation. 
\end{lemma}

\begin{proof} Let $\pi=\pi^{\vec{p}}_{p,s_0} \colon s_0\, \vec{p}\to \vec{p}$, with $p\in s_0$ and $s_0\not\in {\rm vr}(\vec{p}\,)$. If $s_0$ is not a face of $\vec{p}$, then $s_0\, \vec{p}=\vec{p}$ and $\pi={\rm id}$. Then we take $f= g'$ and $\pi_\iota= {\rm id}$, that is, $\pi_\iota$ based on the empty family of faces. 

Assume $s_0$ is a face of $\vec{p}$. Let $\vec{q}$ be the domain of $g'$, that is, $g'\colon \vec{q}\to\vec{p}$. Set 
\[
S= (g')^{-1}(s_0). 
\]
By Lemma~\ref{L:adco}, $S$ is additive in $\vec{q}$ and has the convexity property: for $s_1, s_2\in S$
\begin{equation}\label{E:sfd}
\hbox{if }s_1\subseteq s\subseteq s_2, \hbox{ then } s\in S. 
\end{equation}
For each $s\in S$, let $\iota_S(s)\in s$ be such that $g'(\iota_S(s))= p$. Since $g'(s)=s_0$ and $p_0\in s_0$ such $\iota_S(s)\in s$ exists. Consider 
$\pi_{\iota_S}\colon S \vec{q}\to \vec{q}$. Let $f_1\colon S\vec{q}\to s_0\, \vec{p}$ be defined as $f_1= s_0 g'$. Then $f_1$ is a weld-division map.  
We have 
\begin{equation}\label{E:half}
g'\circ \pi_{\iota_S}= \pi\circ f_1.
\end{equation}
This identity is easily checked by evaluating the left- and right-hand sides on ${\rm vr}(S\vec{q})\cap {\rm vr}(\vec{q})$ and ${\rm vr}(S\vec{q})\setminus {\rm vr}(\vec{q})$. 

Now define $T$ by
\[
t\in  T\;\Leftrightarrow \; \big( t\hbox{ is a face of } \vec{q}\hbox{ and } s\subseteq t,\hbox{ for some }s\in S\big). 
\]
Clearly $T$ is upwards closed in $\vec{q}$. For $t\in T$, let $s^t$ be the largest with respect to inclusion element of $S$ contained in $T$. Such $s^t$ exists by additivity of $S$ in $\vec{q}$. 
For $s\in S$, let 
\[
T_s= \{ t\in T\mid s^t=s\}. 
\] 
Note that, for each $s\in S$, we have that $s\in T_s$, $s$ is included in each set in $T_s$, and, by \eqref{E:sfd}, the family $T_s$ is additive in $\vec{q}$. 
Furthermore, the families $T_s$, $s\in S$, are pairwise disjoint and their union is equal to $T$. 
Let $s_1, \dots, s_n$ be a non-decreasing enumeration of $S$. Observe that if $i<j$, then no element of $T_{s_j}$ is included in an element of $T_{s_i}$. 
Indeed, if $t\in T_{s_i}$, $t'\in T_{s_j}$, and $t'\subseteq t$, then $s_j\subseteq t$. Thus, since $s_i$ is the largest under inclusion element of $S$ included in $t$, we have $s_j\subseteq s_i$, contradicting $i<j$.  
So, we can write the sequence $T\vec{q}$ as 
\[
T\vec{q} = \big(s_1 (T_{s_1}\setminus \{ s_1\}) \big) \cdots \big( s_n (T_{s_n}\setminus \{ s_n\})\big) \vec{q}, 
\]
that is, we can enumerate $T$ in a non-decreasing manner so that in the enumeration $s_1$ is the first entry, followed by a non-decreasing enumeration of $T_{s_1}\setminus \{ s_1\}$, followed by $s_2$ and a non-decreasing enumeration of $T_{s_2}\setminus \{ s_2\}$, etc. 
The map 
\[
f_2\colon T \vec{q} \to  s_1\cdots s_n \vec{q} = S\vec{q}
\]
given by the assignment $T_{s_i}\setminus \{ s_i\} \ni t\to s_i$, for $i\leq n$, is a weld-division map by Lemma~\ref{L:org6}. 

Define now $\iota(t)= \iota_S(s^t)$, for $t\in T$. Note that $\iota(t)\in s^t\subseteq t$. Clearly, we also have 
\begin{equation}\label{E:pig} 
\pi_{\iota_S} \circ f_2 = \pi_\iota. 
\end{equation} 
Set $f=f_1\circ f_2$ and observe that $(\pi, f),\, (g', \pi_\iota)$ is an amalgamation as 
equality \eqref{E:gfpr} follows from \eqref{E:half} and \eqref{E:pig}. 
\end{proof}

Lemma~\ref{L:ama2} gives the relevant formal properties of the amalgamation relation.

\begin{lemma}\label{L:ama2} 
\begin{enumerate} 
\item[(i)] For $i=1,2$, let $(f_i', f_i),\, (g_i', g_i)$ be an amalgamation. Assume that $f_1=g_2'$. Then 
\[
(f_1'\circ f_2', f_2),\, (g_1', g_1\circ g_2)
\] 
is an amalgamation. 

\item[(ii)] Assume $(f'\circ f'', f),\, (g', g)$ is an amalgamation, then 
$(f', f''\circ f),\, (g', g)$ is an amalgamation. 
\end{enumerate}
\end{lemma}

\begin{proof} 
Property \eqref{E:gfpr} of being an amalgamation is checked in points (i) and (ii) by an easy diagram chase. 
%
%
\end{proof}

The following lemma is a consequence of Lemma~\ref{L:ama} and \ref{L:ama2}~(i). 

\begin{lemma}\label{C:cam} 
Let $f' \in {\mathcal W}$ and $g'\in {\mathcal D}$ have the same codomain. Then there exist $f\in {\mathcal D}$ and $g\in {\mathcal N}$ such that 
$(f', f),\, (g', g)$ is an amalgamation. 
\end{lemma}

\begin{proof} The map $f'$ is a composition of a finite number $n$ of weld maps. The proof proceeds by induction on $n$. For $n=1$, the conclusion follows from Lemma~\ref{L:ama}. In the passage from $n$ to $n+1$, we use Lemma~\ref{L:ama} and Lemma~\ref{L:ama2}~(i), and then the closure of $\mathcal N$ under composition. 
\end{proof}

We are ready to present an argument proving Theorem~\ref{T:frai} from Theorem~\ref{T:clsu}.

\begin{proof}[Proof of Theorem~\ref{T:frai} from Theorem~\ref{T:clsu}]
Define ${\mathcal A}$ to consist of all $f'\in {\mathcal D}$ with the following property: for each 
$g'\in {\mathcal D}$ with the same codomain as $f'$, 
there exist $f\in {\mathcal D}$ and $g\in {\mathcal N}$ such that $(f', f),\; (g', g)$ is an amalgamation. 
We need to show that ${\mathcal A}={\mathcal D}$. 
By Lemma~\ref{L:ama2}(ii), if $f', f''\in {\mathcal D}$ and  
$f'\circ f''\in {\mathcal A}$, then $f'\in {\mathcal A}$. 
Thus, it suffices to show that, given $f'\in {\mathcal D}$, there exists $f''\in {\mathcal D}$ such that $f'\circ f'' \in {\mathcal A}$. By 
Theorem~\ref{T:clsu}, for a given $f'$, there exists $f''\in {\mathcal D}$ such that $f'\circ f''$ is in $\mathcal W$. By Lemma~\ref{C:cam}, 
$\mathcal A$ contains all maps in $\mathcal W$ and the conclusion follows. 
\end{proof}

\section{Main Lemma and the proof of the coinitiality theorem from it}\label{S:math}

Recall the definition of maps $\pi_\iota$ from Section~\ref{Su:wene}. If $\pi_\iota$ is based on an additive family $T$ of faces of $\vec{q}$, all faces of $T$ have $p$ in common, and $\iota(t)=p$ for each $t\in T$, we write 
\begin{equation}\notag
\pi^{\vec{q}}_{p,T}
\end{equation} 
for $\pi_\iota$.

We state below the main technical result of the proof of Theorem~\ref{T:clsu}. Its point is to elucidate the structure of 
subdivisions $S\pi_{p,T}$ of $\pi_{p,T}$ for certain additive families $S$ and $T$. Such maps are always weld-division maps. Main Lemma shows that, assuming appropriate interactions between $S$ and $T$, the map $S\pi_{p,T}$ is better than just a weld-division map.

We introduce a class of maps that lies between weld maps and weld-division maps. 
The class is auxiliary but crucial for our arguments in this section and in Section~\ref{S:mal}. Recall the definition \eqref{E:defuu} of the set ${\rm u}(f)$ for a grounded simplicial map $f$. We say that $sf$ is obtained from $f$ by a {\bf pure division based on $s$} if $s\subseteq {\rm u}(f)$. 
The class of {\bf pure weld-division maps} is the smallest class closed under compositions and pure division and 
containing all weld maps and all combinatorial isomorphisms.

\begin{main}
Let $S$ and $T$ be additive families of faces of a sequence $\vec{q}$ such that 
for $s\in S$ and $t\in T$, if $s\cup t \hbox{ is a face of }\vec{q}$, then $s\cup t \in T$.
Let $p$ be a vertex of $\vec{q}$ contained in every face from $T$. 
Then the map $S\,\pi^{\vec{q}}_{p,T}$ is a pure weld-division map. 
\end{main}

The proof of Main Lemma is rather involved, and we postpone it till Section~\ref{S:mal}. Now, we proceed with the proof of Theorem~\ref{T:clsu} from Main Lemma. We start by giving a characterization of pure weld-division maps.

\begin{lemma}\label{L:chrpu} 
A map is a pure division-weld map if and only if it is a composition of combinatorial isomorphisms and maps of the form 
\begin{equation}\label{E:welpur} 
s_0\cdots s_m\, \pi^{\vec{q}}_{p,t}, 
\end{equation} 
where $s_0,\dots, s_m, t\in {\rm Fin}^+$, $p\in t\not\in {\rm vr}(\vec{q}\,)$, and $p\notin s_i$ for all $i\leq m$. 
\end{lemma} 

\begin{proof} We make some preliminary observations about maps in \eqref{E:welpur}. Note that if $t$ is not a face of $\vec{q}$ or $t$ is a one element set, then $\pi^{\vec{q}}_{p,t}$ is a combinatorial isomorphism---it is the identity, in the first case, and a combinatorial isomorphism of type 3b, in the second case. 
Assume $t$ is a face of $\vec{q}$ and has at least two elements. Set $\pi= \pi^{\vec{q}}_{p,t}$. Let now $s_0, \dots , s_m\in {\rm Fin}^+$ be such that $p\not\in s_i$, for $i\leq m$. 
By an easy induction one proves that the domain and codomain of $s_0\cdots s_m \pi$ are $s_0\cdots s_m t\,\vec{q}$ and $s_0\cdots s_m \vec{q}$, respectively, and that we have 
\begin{enumerate} 
\item[(a)] ${\rm vr}(s_0\cdots s_m t\,\vec{q}\,) \setminus \{ t\} = {\rm vr}(s_0\cdots s_m \vec{q}\,)$\; and \;$s_0\cdots s_m\pi\res {\rm vr}(s_0\cdots s_m \vec{q}\,) = {\rm id}$;

\item[(b)] ${\rm u}(s_0\cdots s_m\pi) = {\rm vr}(s_0\cdots s_m \vec{q}\,)\setminus \{ p\}$.
\end{enumerate} 
We use the assumptions that $t$ is a face of $\vec{q}$ and $p\not\in s_i$ to see (b) and the assumption that $t$ has at least two elements to see (a). 

We conclude that 
\begin{equation}\label{E:dile} 
\hbox{$\pi$ 
is a combinatorial isomorphism or $s_0\cdots s_m\pi$ fulfills (b)}.
\end{equation}
We proceed to proving the equivalence from the conclusion of the lemma. 

($\Leftarrow$) Since each combinatorial isomorphism $f\colon \vec{t}\to \vec{s}$ is a bijection as a function ${\rm vr}(\vec{t}\,)\to {\rm vr}(\vec{s}\,)$, we see that ${\rm u}(f) = {\rm vr}(\vec{s}\,)$. Thus, each division of a combinatorial isomorphism is a pure division, so, inductively, each combinatorial isomorphism is a pure weld-division map. It remains to show that maps in \eqref{E:welpur} arepure weld-division maps assuming that $\pi^{\vec{q}}_{p,t}$ is not a combinatorial isomorphism. In that case, the conclusion follows from (b) again by induction.

($\Rightarrow$) It suffices to prove that the family $\mathcal B$ consisting of compositions of combinatorial isomorphisms and maps in \eqref{E:welpur} is closed under pure division. We note that if $g\colon \vec{t}\to \vec{s}$, $f\colon \vec{s}\to \vec{r}$ are grounded simplicial maps then 
\begin{equation}\label{E:insu}
{\rm u}(f\circ g)\subseteq {\rm u}(f) \cap f({\rm u}(g)), 
\end{equation} 
which is a consequence of the surjectivity of $g$ as a function ${\rm vr}(\vec{t}) \to {\rm vr}(\vec{s})$. Inclusion \eqref{E:insu} gives, for $r\in {\rm Fin}^+$, that if $r\subseteq {\rm u}(f\circ g)$, then $r\subseteq {\rm u}(f)$ and, for the unique set $s\subseteq {\rm vr}(\vec{s}\,)$ with $f(s)=r$, we have $s\subseteq {\rm u}(g)$. Thus, with the assumption $r\subseteq {\rm u}(f\circ g)$, we have, for some $s\in {\rm Fin}^+$, 
\begin{equation}\label{E:puco}
\begin{split}
r(f\circ g) &= 
\begin{cases}
f\circ g, &\text{ if $r$ is not a face of $\vec{r}$;}\\
(rf) \circ (s g), &\text{ if  $r$ is a face of $\vec{r}$},
\end{cases}\\
r &\subseteq {\rm u}(f),\\ 
s &\subseteq {\rm u}(g).
\end{split}
\end{equation} 
Further, we observe that if $s_0\cdots s_m\, \pi^{\vec{q}}_{p,t}$ is as in \eqref{E:welpur}, $\emptyset\not=s\subseteq u(s_0\cdots s_m\, \pi^{\vec{q}}_{p,t})$, and 
$\pi^{\vec{q}}_{p,t}$ is not a combinatorial isomorphism, then 
\[
ss_0\cdots s_m\, \pi^{\vec{q}}_{p,t}, 
\]
is also as in \eqref{E:welpur}, that is, $p\not\in s$. This follows directly from condition (b). By this observation, assertion \eqref{E:dile}, the fact that combinatorial isomorphisms are closed under division, and by \eqref{E:puco}, we conclude that $\mathcal B$ is closed under pure division. 
\end{proof}

\begin{proof}[Proof of Theorem~\ref{T:clsu} from Main Lemma] Define 
\[
{\mathcal Z} = \{ h\in {\mathcal D} \mid h\circ f \in {\mathcal W}, \hbox{ for some }f\in {\mathcal D} \}.
\]
Proving Theorem~\ref{T:clsu} amounts to showing that ${\mathcal Z}= {\mathcal D}$. 
Obviously, this will be accomplished if we show points (b), (c), (e), and (f) listed below. We state them along with auxiliary properties (a) and (d) and then prove them in turn.

\begin{enumerate}
\item[(a)] If $h\circ f\in {\mathcal Z}$, for some $f\in {\mathcal D}$, then $h\in {\mathcal Z}$. 

\item[(b)] $\mathcal Z$ contains all combinatorial isomorphisms.

\item[(c)] ${\mathcal Z}$ is closed under composition. 

\item[(d)] $\mathcal Z$ contains allpure weld-division maps. 

\item[(e)] ${\mathcal Z}$ contains all maps of the form $S\pi$, where $\pi$ is a weld map and $S$ is an additive family of faces of the codomain of $\pi$.

\item[(f)] $\mathcal Z$ is closed under division. 
\end{enumerate}

Point (a) is clear since $\mathcal D$ is closed under composition.

For (b), note that if $g$ is a combinatorial isomorphism, then the combinatorial isomorphism $f= g^{-1}$ witnesses that $g\in {\mathcal Z}$.

To see (c), let $h_1, h_2\in {\mathcal Z}$ be such that $h_1\circ h_2$ is defined. 
We aim to show that $h_1\circ h_2\in {\mathcal Z}$. Fix $f_1, f_2\in {\mathcal D}$ witnessing that $h_1, h_2\in {\mathcal Z}$, that is, 
\[
g_1= h_1\circ f_1,\; g_2=h_2\circ f_2\in {\mathcal W}.
\]
Note that $g_2$ and $f_1$ have the same codomain and 
apply Lemma~\ref{C:cam} to $g_2\in {\mathcal W}$ and $f_1\in {\mathcal D}$ to obtain $f\in {\mathcal N}\subseteq {\mathcal W}$ and $g \in {\mathcal D}$ such that 
$g_2\circ g = f_1\circ f$. Then notice that 
\[
(h_1\circ h_2)\circ (f_2\circ g) = h_1\circ g_2\circ g= h_1\circ f_1\circ f=g_1\circ f\in {\mathcal W}. 
\]
It follows that $f_2\circ g\in {\mathcal D}$ witnesses that $h_1\circ h_2\in {\mathcal Z}$.

We now show (d), that is, we prove that all pure weld-division maps are in $\mathcal Z$. 
To this end, we start with the following claim analyzing maps of the form $\pi^{\vec{q}}_{p,t}$. 
\begin{claim}\label{Cl:1} 
Let $s, t$ be faces of $\vec{q}$ with $p\in t$ and $p\not\in s$. Consider the weld map $\pi^{\vec{q}}_{p,t}\colon t\,\vec{q}\to \vec{q}$ and its division $s\pi^{\vec{q}}_{p,t}\colon st\, \vec{q}\to s\,\vec{q}$. Then 
\[
(s\pi^{\vec{q}}_{p,t})\circ f =\pi_1\circ \pi_2, 
\]
for some $f\in {\mathcal D}$ and where, for some $t_1, t_2\in {\rm Fin}^+$ that are not vertices of $s\,\vec{q}$ and $t_1s\,\vec{q}$, respectively, we have $p\in t_1$, $p\in t_2$, and $\pi_1=\pi^{s\,\vec{q}}_{p,t_1}$ and $\pi_2=\pi^{t_1s\,\vec{q}}_{p,t_2}$. 
\end{claim} 

\noindent {\em Proof of Claim~\ref{Cl:1}.} We keep in mind that $p\in t$ and $p\not\in s$. 
Set 
\[
t(s) = (t\setminus s) \cup \{ s\}\;\hbox{ and }\; s(t)= (s\setminus t)\cup \{ t\}. 
\]
Since $s$ and $t$ are not vertices of $\vec{q}$ (since they are its faces), we have a combinatorial isomorphism of type 2
\[
h\colon t(s) ts\, \vec{q}\to s(t) s\, t\, \vec{q}.
\]
given by the assignment 
\begin{equation}\label{E:easig} 
t(s)\to s(t).
\end{equation}
Note further that, since $s(t)$ and $t(s)$ are not equal to either $s$ or $t$, $s$ is not equal to $t$, and $s(t), t(s), s, t$ are not vertices of $\vec{q}$ (since $s$ and $t$ are its faces), we have 
\begin{enumerate}
\item[] $s(t)$ is a not a vertex of $st\vec{q}$ and $t\in s(t)$, 

\item[] $t(s)$ is not a vertex of $ts\vec{q}$ and $p\in t(s)$, and 

\item[] $t$ is not a vertex of $s\vec{q}$ and $p\in t$.
\end{enumerate}
So, we can consider the weld maps 
\[
\begin{split}
\pi_0&= \pi_{t, s(t)}^{st\vec{q}}\colon s(t)st\vec{q}\to st\vec{q},\\
\pi_1&=  \pi^{s\vec{q}}_{p,t}\colon ts\vec{q}\to s\vec{q},\\
\pi_2 &= \pi_{p, t(s)}^{ts\vec{q}} \colon t(s) ts\vec{q}\to ts\vec{q}.
\end{split} 
\]
Now note that maps $s\pi^{\vec{q}}_{p,t}\circ (\pi_0\circ h)$ and $\pi_1\circ \pi_2$ have the same domains and codomains, which are $t(s) ts\vec{q}$ and $s\vec{q}$, respectively, and, by \eqref{E:easig} and the definition of weld maps, they are both given by the assignment 
\[
t \to p\;\hbox{ and }\; t(s)\to p.
\]
It follows that 
\[
s\pi^{\vec{q}}_{p,t}\circ (\pi_0\circ h) = \pi_1\circ \pi_2, 
\]
and $f=\pi_0\circ h$, $t_1=t$, and $t_2= t(s)$ are as required by the conclusion of Claim~\ref{Cl:1}. 

\medskip

\begin{claim}~\label{Cl:2}
Let $r\in {\rm Fin}^+$ not be a vertex of a sequence $\vec{q}$ and let $p\in r$. Consider the weld map $\pi^{\vec{q}}_{q,r}\colon r\,\vec{q}\to \vec{q}$. 
Let $s_1, \dots, s_m \in {\rm Fin}^+$ be such that $p\notin s_i$ and $s_i$ is a face of $s_{i+1}\cdots s_m \vec{q}$, for all $1\leq i\leq m$. Then, for each $1\leq i\leq m$, the following conditions hold. 
\begin{enumerate} 
\item[(i)] $s_i$ is a face of $s_{i+1}\cdots s_m r\,\vec{q}$ and it is the only face of $s_{i+1}\cdots s_m r\,\vec{q}$ mapped by $s_{i+1}\cdots s_m \pi^{\vec{q}}_{q,r}$ to $s_i$. 

\item[(ii)] If $g\colon \vec{r}\to s_{i+1}\cdots s_m r\,\vec{q}$ is in $\mathcal D$, for some sequence $\vec{r}$, then 
\[
s_1\cdots s_i( s_{i+1}\cdots s_m 
\pi^{\vec{q}}_{q,r}
\circ g) = (s_1\cdots s_m \pi^{\vec{q}}_{q,r})\circ (s_1\cdots s_i g).
\]
\end{enumerate} 
\end{claim}

\noindent {\em Proof of Claim~\ref{Cl:2}.} Point (ii) follows from (i) and Lemma~\ref{L:X}.  

To prove (i), note that $s_i$ is a face of $s_{i+1}\cdots s_m r\vec{q}$. This statement is proved by induction as follows. Assume $s_i$ is a face of $s_{i+1}\cdots s_m \vec{q}$. If $s_i$ is a face of $\vec{q}$, then, since $r\not\subseteq s_i$, $s_i$ is a face of $r\vec{q}$, and so a face of $s_{i+1}\cdots s_m r \vec{q}$, as required, since we have $s_j\not\subseteq s_i$, for $j>i$, as $s_i$ is assumed to be a face of $s_{i+1}\cdots s_m \vec{q}$. If $s_i$ is not a face of $\vec{q}$, then there is $m\leq j>i$ such that $s_j\in s_i$ and $\big(s_i\setminus \{ s_j\}\big)\cup s_j$ is a face of $s_{j+1}\cdots s_m \vec{q}$. So, by induction, we have that 
$\big(s_i\setminus \{ s_j\}\big)\cup s_j$ is a face of $s_{j+1}\cdots s_m r \vec{q}$. So $s_i$ is a face of $s_{i+1}\cdots s_m \vec{q}$. 

Furthermore, the maps $s_{i+1}\cdots s_m \pi^F_{q,r}$, $i<m$, are given by the assignment $r\to p$. Thus, to see that $s_i$ is the only face of $s_{i+1}\cdots s_m rF$ mapped by $s_{i+1}\cdots s_m \pi^F_{q,r}$ to $s_i$, it suffices to note that $p\not\in s_i$, as is assumed. The claim follows.

\smallskip

Now, we finish proving (d), that is, we show that each pure weld-division map is in $\mathcal Z$. By Lemma~\ref{L:chrpu} and by (b) and (c), it will be enough to prove the following statement.

\noindent {\em Let $s_1, \dots, s_m\in {\rm Fin}^+$, let $t$ be a face of $\vec{q}$, and let $p\in t$ and $p\not\in s_i$, for $i\leq m$. Then 
\[
s_1\cdots s_m \pi^{\vec{q}}_{p,t}\in {\mathcal Z}. 
\]}

\noindent The statement is proved by induction on $m$. Clearly we can assume that, for all $i\leq m$, $s_i$ is a face of $s_{i+1}\cdots s_m \vec{q}$ since otherwise we remove from the sequence $s_1\cdots s_m$ all $s_i$ such that $s_i$ is not a face of $s_{i+1}\cdots s_m\vec{q}$, and the resulting map is equal to $s_1\cdots s_m \pi^{\vec{q}}_{p,t}$. If $m=0$, the statement is obvious since $\pi_{p,t}\in {\mathcal Z}$ by definition of $\mathcal Z$. If $m=1$, the statement follows from Claim~\ref{Cl:1}, (a), and (c). Assume $m\geq 2$ and suppose that the statement holds for $m-1$.

Apply Claim~\ref{Cl:1} to $\pi^{\vec{q}}_{p,t}$ and $s=s_m$, so the conclusion of Claim~\ref{Cl:1} holds with some $f\in {\mathcal D}$ and $\pi_1$ and $\pi_2$. We get from Claim~\ref{Cl:2}
\begin{equation}\label{E:altg} 
\begin{split}
(s_1\cdots s_{m-1} s_m \pi^{\vec{q}}_{p,t}) \circ (s_1\cdots s_{m-1} f) &= (s_1\cdots s_{m-1}) \big((s_m \pi^{\vec{q}}_{p,t}) \circ f\big)\\ 
&= s_1\cdots s_{m-1} (\pi_1\circ \pi_2). 
\end{split}
\end{equation} 
Since $\pi_1=\pi^{s_m\vec{q}}_{p,t_1}$ with $p\in t_1$ and since $p\not\in s_i$ for $i<m$, Claim~\ref{Cl:2} gives 
\[
s_1\cdots s_{m-1} (\pi_1\circ \pi_2)=(s_1\cdots s_{m-1} \pi_1)\circ (s_1\cdots s_{m-1} \pi_2).
\]
Combining this equation with \eqref{E:altg} and applying (a) and (c) makes it clear that it suffices to see that 
\[
s_1\cdots s_{m-1} \pi_1,\; s_1\cdots s_{m-1} \pi_2\in {\mathcal Z}. 
\]
But this follows from $\pi_1$ being equal to $\pi^{s_m\vec{q}}_{p,t_1}$, $\pi_2$ being equal to $\pi^{t_1s_m\vec{q}}_{p,t_2}$, $p\not\in s_i$ for $i<m$ 
and the inductive assumption for $m-1$. Thus, point (d) is proved.

We move to proving (e). This part of the proof involves an application of Main Lemma. Let 
\[
\pi= \pi^{\vec{q}}_{p,t}\colon t\,\vec{q}\to \vec{q}. 
\]
where $t\in {\rm Fin}^+$ is not a vertex of $\vec{q}$, and $p\in t$. 
If $t$ is not a face of $\vec{q}$, then $\pi$ and, so also $S\pi$, are identity maps and there is nothing to prove. So assume $t$ is a face of $\vec{q}$.   
By (a), it will suffice to find $f\in {\mathcal D}$ such that 
\begin{equation}\label{E:cirf}
(S\pi) \circ f \in {\mathcal Z}. 
\end{equation} 
We proceed to construct such a map $f\in {\mathcal D}$.

Define 
\[
T = \{ t\}\cup \{ t\cup s\mid s\in S,\, t\cup s\hbox{ a face of } \vec{q}\,\}. 
\]
This is an additive family in $\vec{q}$ since $S$ is. Note that $p$ is an element in each set in $T$ and consider the map $\pi_{p,T}\colon T \vec{q}\to \vec{q}$. 
Define $\rho\colon T\vec{q}\to t\,\vec{q}$ by the assignment that maps each element of $T$ to $t$. Since $t\subseteq t'$, for each $t'\in T$, we can wriyr 
\[
T\vec{q} = t \,(T\setminus \{ t\}) \vec{q}.
\]
We see from Lemma~\ref{L:org6} with $m=1$ that $\rho$ is a weld-division map. Observe that 
\begin{equation}\label{E:ror}
\pi_{p,T} = \pi\circ \rho.
\end{equation} 
Now, from \eqref{E:ror} and Lemma~\ref{L:X}, we have 
\begin{equation}\label{E:alth}
S \pi_{p,T} = \big( S\pi\big)\circ \big( S' \rho\big),
\end{equation} 
where $S' = \pi^{-1}(S)$.
Note that $S' \rho$ is a weld-division map since $\rho$ is. It is immediate that $S$ and $T$ fulfill 
the assumptions of Main Lemma. Thus, by Main Lemma, $S \pi_{p,T}$ is a pure weld-division map. 
This assertion together with \eqref{E:alth} implies \eqref{E:cirf} with $f=S'\rho$ since 
each pure weld-division map is in $\mathcal Z$ by (d).

We prove (f). Fix $h\colon \vec{q}\to \vec{p}$ in ${\mathcal Z}$ and let $s\in {\rm Fin}^+$. We need to prove that $sh\in {\mathcal Z}$. 
We can assume that $s$ is a face of $\vec{p}$ since otherwise $sh=h$ and there is nothing to prove. 
Fix 
$f\in {\mathcal D}$ such that $g= h\circ f\in {\mathcal W}$. Set $S= h^{-1}(s)$. Then, by Lemma~\ref{L:X}, we have 
\[
sg=(sh)\circ (Sf).
\]
Thus, by (a), it will suffice to show that 
\begin{equation}\label{E:sgg} 
sg\in {\mathcal Z}. 
\end{equation}
Now, since $g\in {\mathcal W}$, we have 
\[
g= \pi_1\circ \pi_2\circ \cdots \circ \pi_n,
\]
for some weld maps $\pi_1, \dots, \pi_n$. It follows from Lemma~\ref{L:X} that 
\begin{equation}\label{E:prap} 
sg = (S_1 \pi_1)\circ (S_2 \pi_2)\circ \cdots \circ (S_n \pi_n),
\end{equation} 
where $S_1 = \{ s\}$ and, for $1\leq i<n$, $S_{i+1} = \pi_i^{-1}(S_i)$. Since each $S_i$ is additive in the codomain of $\pi_{i}$, \eqref{E:sgg} is a consequence of \eqref{E:prap}, (c), and (e). 

We showed (a)--(f) and the theorem is proved. 
\end{proof}

\section{Proof of Main Lemma}\label{S:mal}

This section is concerned with the map 
\[
S\pi^{\vec{q}}_{p,T} \colon \pi_{p,T}^{-1}(S)\, T\, \vec{q}\to S\,\vec{q},
\]
where $S$ and $T$ fulfill the assumptions in Main Lemma, that is, 
\begin{enumerate}
\item[(I)] $S$ and $T$ are additive families of faces of $\vec{q}$;

\item[(II)] for $s\in S$ and $t\in T$, if $s\cup t \hbox{ is a face of }\vec{q}$, then $s\cup t \in T$;

\item[(III)] $p\in t$, for each $t\in T$. 
\end{enumerate}
We will often refer to these properties in what follows.

To make our notation lighter, we set 
\[
\pi= \pi^{\vec{q}}_{p,T}. 
\]

\subsection{Definitions needed for analyzing the family $\pi^{-1}(S)$ and the map $S\pi$}

First, we set up notation that will lead to a useful description of the elements of $\pi^{-1}(S)$ and the map $S\pi$, which is given in Lemma~\ref{L:desco}.  
Let 
\begin{equation}\notag
S_p = \{ s\in S\mid p\in s\}, 
\end{equation} 
and let 
\[
{\mathcal T} = \{ X\subseteq T\mid X\not=\emptyset,\; X \hbox{ linearly ordered by }\subseteq\}. 
\]
For $X\in {\mathcal T}$, we write 
\[
\min X
\]
for the smallest with respect to inclusion element in $X$. 

It will be convenient to write 
\[
\un{s} = s\setminus \{ p\}, \hbox{ for }s\in S. 
\]
Note that $\un{s}=s$ if $s\in S\setminus S_p$.
For $X\in{\mathcal T}$ with $r\subseteq \min X$,  
set 
\begin{equation}\label{E:rxe}
r(X) = r\cup X.
\end{equation} 
We will use the above notation only for $r=s$ and $r= \un{s}$ with $s\in S_p$, $s\subseteq \min X$.

Below, we state a general lemma that is a consequence of Lemma~\ref{L:linear}. 

\begin{lemma}\label{L:linear2} Let $\vec{p}$ be a sequence, 
let $T$ be an additive set of faces of $\vec{p}$, let $x$ be a set, 
and let $\emptyset\not= X\subseteq T$ be such that 
$x\cup t\in T$ for each $t\in X$. Then  
\[
x\cup X \hbox{ is a face of } T\vec{p}
\]
if and only if the following conditions hold 
\begin{enumerate} 
\item[($\alpha$)] $t\not\subseteq x$ for each $t\in T$;

\item[($\beta$)] $X$ is linearly ordered by $\subseteq$;

\item[($\gamma$)] $x\subseteq t$,  for each $t\in X$.
\end{enumerate}
\end{lemma} 

\begin{proof}
The implication $\Leftarrow$ is obvious by Lemma~\ref{L:linear} and Proposition~\ref{P:conr}; the implication $\Rightarrow$ is almost equally obvious; to get $(\gamma)$ apply (d) of Lemma~\ref{L:linear} and the assumption that $x\cup t\in T$, for each $t\in X$. 
\end{proof}

We now come back to the proof of Main Lemma. We have a lemma that gives a description of elements of the family $\pi^{-1}(S)$ and the map $S\pi$. 

\begin{lemma}\label{L:desco} 
\begin{enumerate}
\item[(i)] $\pi^{-1}(S)$ is equal to the union of the following, pairwise disjoint families of faces of $T\vec{q}$
\begin{enumerate} 
\item[---] $S\setminus T$; 

\item[---] $\{ s(X) \mid s\in S_p\setminus T,\, X\in {\mathcal T},\, s\subseteq \min X\}$;

\item[---] $\{ \un{s}(X) \mid s\in S_p,\, X\in {\mathcal T},\, s\subseteq \min X\}$.
\end{enumerate} 

\item[(ii)] $S\pi\colon \pi^{-1}(S)\, T\,\vec{q}\to S\,\vec{q}$ 
is a map given by the assignment 
\begin{enumerate} 
\item[---] $T \ni t\to p$; 

\item[---] $s(X)\to s$,\, for $s\in S_p\setminus T$, $X\in {\mathcal T}$, $s\subseteq \min X$; 

\item[---] $\un{s}(X)\to s$,\, for $s\in S_p$, $X\in {\mathcal T}$, $s\subseteq \min X$.
\end{enumerate} 
\end{enumerate} 
\end{lemma} 

\begin{proof} (i) Each face of $T\vec{q}$ is of the form 
\[
x\cup X,
\]
with $x\subseteq {\rm vr}(\vec{q})$ and $X\subseteq T$. 
By inspection, noticing that if $x\cup X$ is a face of $TA$, then 
\[
\pi(x\cup X) = \pi(x) \cup \pi(X) = x \cup \pi(X)
\]
and 
$\pi(t)=p$ for all $t\in X$, we see that of those the ones mapped by $\pi$ to some $s\in S$ must be of the form 
\begin{equation}\label{E:cand}
\begin{split}
&s = s\cup \emptyset,\, \hbox{ for }s\in S,\\
&s\cup X\,\hbox{ and }\, \un{s} \cup X,\, \hbox{ for } s\in S_p\hbox{ and }\emptyset\not= X\subseteq T.
\end{split} 
\end{equation} 
Now observe using (II) that 
\[
s\in S\setminus T \Longleftrightarrow s\in S\hbox{ and } t\not\subseteq s, \hbox{ for all }t\in T.
\]
Keeping this in mind, using assumptions (II) and (III), and applying Lemma~\ref{L:calc}(ii) and Lemma~\ref{L:linear2}, we see that of the sets in \eqref{E:cand} precisely the 
ones in the conclusion of (i) are faces of $T\vec{q}$. 

(ii) is checked by inspection using (i). Note that by our convention for assignments, each $s\in S\setminus T$ is mapped to $s$. 
\end{proof}

\subsection{First reduction}\label{Su:redu1} 

In this section, we show that, without loss of generality, we can impose additional conditions on $S$ and $T$, namely: for each $t\in T$, there is $s\in S$ with $s\subseteq t$; for each $s\in S$, there is a $t\in T$ with $s\subseteq t$; and $p\in s$, for each $s\in S$. These conditions are stated as (IV), (V), (VI) in Section~\ref{Su:linord}.

We will need a general lemma on the interaction of two additive families of faces. 

\begin{lemma}\label{L:inadd} 
Let $\vec{t}$ be a sequence, and let $U$, $V$, and $W$ be additive families of faces of $\vec{t}$. 
\begin{enumerate}
\item[(i)] If $U\subseteq W$ and no element of $U$ is a subset of an element of $W\setminus U$, then $W\setminus U$ is an additive family of faces of $U\vec{t}$ and 
\[
(W\setminus U)\, U \vec{t} = W\,\vec{t}.
\]

\item[(ii)] If, for all $u\in U$ and $v\in V$, $u\cup v$ is not a face of $\vec{t}$, then $U$ is an additive family of faces of $V\vec{t}$, $V$ is an additive family of faces of $U\vec{t}$, and 
\[
UV \vec{t} = VU \vec{t}.
\] 
\end{enumerate} 
\end{lemma}

\begin{proof} (i) The definition of face and the conditions on $U$ and $W$ give that elements $W\setminus U$ are faces of $U\vec{t}$. To see that this family is additive, let $v_1, v_1\in W\setminus U$ and assume that $v_1\cup v_2$ is a face of $U\vec{t}$. Since $v_1\cup v_2\subseteq {\rm vr}(\vec{t}\,)$ (as $v_1$ and $v_2$ are faces of $\vec{t}$\,), by Lemma~\ref{L:nonfa}(i), we have that $v_1\cup v_2$ is a face of $\vec{t}$. So, we get $v_1\cup v_2\in W$. If $v_1\cup v_2\in U$, then, by Lemma~\ref{L:nonfa}(ii), we have that 
$v_1\cup v_2$ is not a face of $U\vec{t}$, which is a contradiction. So, $v_1\cup v_2\not\in U$, and we checked additivity of $W\setminus U$.

The conditions on $U$ and $W$ make it possible to find an nondecreasing enumeration $v_0\cdots v_n$ of $W$ such that, for some $m\leq n$, we have $W\setminus U = \{ v_i\mid i\leq m\}$ and $U= \{ v_i\mid m<i\leq n\}$, which proves the second part of the conclusion by Lemma~\ref{L:dedf}.

(ii) Note first that no element of $U$ is included as a subset in an element of $V$; in particular, $U\cap V=\emptyset$. Indeed, if $u\in U$ and $v\in V$ and $u\subseteq v$, then $u\cup v= v$ is a face of $\vec{t}$ contradicting our assumptions. Similarly, no element of $V$ is included as a subset in an element of $U$. Moreover, clearly $U\cup V$ is an additive family of faces of $\vec{t}$. 
Thus, applying (i) with $W= U\cup V$ and noting that $U= W\setminus V$ and $V= W\setminus U$, we see that $U$ is an additive family of faces of $V\vec{t}$, $V$ is an additive family of faces of $U\vec{t}$, and 
\[
UV\vec{t}= (U\cup V)\, \vec{t}= VU\vec{t}.
\]
Point (ii) is proved. 
\end{proof} 

We will need another general lemma on dividing pure weld-division maps by additive families of faces. 

\begin{lemma}\label{L:purediva} 
Let $f\colon \vec{t}\to \vec{r}$ be a pure weld-division map and let $S$ be an additive family of faces of $\vec{r}$ such that, for each $s\in S$, we have $s\subseteq {\rm u}(f)$. Then $Sf\colon f^{-1}(S)\, \vec{t}\to S\,\vec{r}$ is a pure weld-division map. 
\end{lemma} 

\begin{proof} The proof is done by induction of the size of $S$. If $S=\emptyset$, the conclusion is evident. Assume that $S\not=\emptyset$. Let $s_0\in S$ be minimal with respect to inclusion, and let $S_0= S\setminus \{ s_0\}$. Note that $S_0$ is an additive family of faces of $\vec{r}$ and that the assumption of the lemma holds for $S_0$. So, by induction, we have that $S_0 f$ is a pure weld-division map. Note further that $s_0$ is a face of $S_0 \vec{r}$ since no element of $S_0$ is a subset of $s_0$. Further, by Lemma~\ref{L:difsa}(iii), we have 
$Sf= s_0(S_0f)$. Thus, to show that $Sf$ is a pure weld-division map, it suffices to see that $s_0\subseteq {\rm u}(S_0f)$. But notice that 
\begin{equation}\label{E:dogh}
{\rm u}(S_0f) \supseteq {\rm u}(f)\setminus \{ x\mid \{ x\}\in S_0\}. 
\end{equation} 
If $x\in s_0$ for some $x$ with $\{ x\}\in S_0$, then $\{ x\}\subseteq s_0$, which contradicts the assertion that no element of $S_0$ is a subset of $s_0$. Thus, $x\not\in s_0$ for each $x$ with $\{ x\}\in S_0$. So, since, by assumption, $s_0\subseteq {\rm u}(f)$, it follows from \eqref{E:dogh} that $s_0\subseteq {\rm u}(S_0f)$, as required. 
\end{proof}

We now start the proof of the reduction of this section. 
Let 
\[
T' = \{ t\in T\mid s\not\subseteq t\hbox{ for all } s\in S\}. 
\]
We aim to show that, without loss of generality, we can assume that $T'=\emptyset$.

First we have the following lemma. 

\begin{lemma}\label{L:usle} 
$T\setminus T'$ is an additive set of faces of $\vec{q}$, $T'$ is an additive set of faces of $ \pi^{-1}(S)\, (T\setminus T') \,\vec{q}$, the family $\pi^{-1}(S)$ consists of faces of  $(T\setminus T')\, \vec{q}$, and 
\begin{equation}\label{E:pitt}
\pi^{-1}(S)\, T\, \vec{q} = T' \pi^{-1}(S)\, (T\setminus T')\, \vec{q}. 
\end{equation} 
\end{lemma} 

\begin{proof} The family $T\setminus T'$ is an additive family of faces of $\vec{q}$ by additivity of $T$ and the definition of $T'$. 
Clearly, no element of $T\setminus T'$ is included in an element of $T'$; thus, by Lemma~\ref{L:inadd}(i), $T'$ is an additive set of faces of $(T\setminus T')\,\vec{q}$ and 
\begin{equation}\label{E:ttp} 
T\vec{q} = T' (T\setminus T')\, \vec{q}. 
\end{equation} 

By the condition $s\subseteq \min X$ in Lemma~\ref{L:desco}(i), $\pi^{-1}(S)$ consists of faces of the sequence $(T\setminus T')\,\vec{q}$. 
We claim that for $u\in \pi^{-1}(S)$ and $t\in T'$, the set $u\cup t$ is not a face of $(T\setminus T')\,\vec{q}$. Let $s\in S$ be such that  $\un{s}\subseteq u$. Such an $s$ exists by Lemma~\ref{L:desco}(i). Then $\un{s}\cup t = s\cup t$ since $p\in t$. Thus, $s\cup t\subseteq u\cup t$, and it suffices to show that $s\cup t$ is not 
a face of $(T\setminus T')\, \vec{q}$. By Lemma~\ref{L:nonfa}(i), if $s\cup t$ is not a face of $\vec{q}$, then it 
is not a face of $(T\setminus T')\,\vec{q}$ as $s\cup t$ is a subset of ${\rm vr}(\vec{q}\,)$. On the other hand, if $s\cup t$ is a face of $\vec{q}$, 
then, by our assumption (II), $s\cup t\in T$ and it clearly contains an element of $S$. So $s\cup t\in T\setminus T'$, hence it is not 
a face of $(T\setminus T')\, \vec{q}$ by Lemma~\ref{L:nonfa}(ii). 

The assertion that was just proved and Lemma~\ref{L:inadd}(ii) show that 
$T'$ is an additive set of faces of $ \pi^{-1}(S)\, (T\setminus T') \,\vec{q}$ and that \eqref{E:pitt} follows from \eqref{E:ttp}. 
\end{proof}

Set 
\[
\pi'= \pi^{\vec{q}}_{p, T\setminus T'} \colon (T\setminus T')\, \vec{q} \to \vec{q}
\]
and observe that the families $S$ and $T\setminus T'$ fulfill conditions (I), (II), and (III). Thus, Lemma~\ref{L:desco} applies to $\pi'$, $S$, and $T\setminus T'$. Comapring Lemma~\ref{L:desco}(i) for $\pi, S, T$ and for $\pi', S, T\setminus T'$, we see that 
\[
(\pi')^{-1}(S) = \pi^{-1}(S), 
\]
from which we get 
\[
S\pi' \colon \pi^{-1}(S) (T\setminus T')\, \vec{q} \to \vec{q}.
\]
Set 
\[
\vec{Q} =  \pi^{-1}(S) (T\setminus T')\, \vec{q}, 
\]
so, by Lemma~\ref{L:usle}, 
\[
S\pi' \colon \vec{Q} \to \vec{q}\;\hbox{ and }\; S\pi\colon T'\vec{Q}\to \vec{q}.
\]
Now, the map 
\[
\pi^{\vec{Q}}_{T',p} \colon T' \vec{Q} \to \vec{Q}
\]
is a composition of weld maps by Lemma~\ref{L:ioco}. Note also that by Lemma~\ref{L:desco}(ii) applied to $\pi, S, T$ and to $\pi', S, T\setminus T'$, we have 
\[
S \pi = S\pi'\circ \pi^{\vec{Q}}_{T',p} 
\]
Thus, to prove Main Lemma it suffices to show that the map $S\pi'$  is a pure weld-division map. That is, we can assume  that $T = T\setminus T'$,
so $T'=\emptyset$. Therefore, from this point on we assume that 
\begin{equation}\label{E:redst} 
\hbox{for each }t\in T,\hbox{ there exists } s\in S\hbox{ with } s\subseteq t.
\end{equation}

Let now 
\[
S' = \{ s\in S\mid s\not\subseteq t\hbox{ for all }t\in T\}. 
\]
Our aim is to show that we can assume that $S'$ is empty. 

\begin{lemma}\label{L:ggaa} 
$S'$ is an additive family of faces of $\vec{q}$, $T$ and $S\setminus S'$ are additive families of faces of $S'\vec{q}$, and 
\begin{equation}\label{E:sqsq} 
S\vec{q}= (S\setminus S')\,S'\vec{q} \;\hbox{ and }\;    \pi^{-1}(S)\, T \vec{q} = \pi^{-1}(S\setminus S')\, TS' \vec{q}.
\end{equation}
\end{lemma} 

\begin{proof} 
The family $S'$ is an additive family of faces of $\vec{q}$ since it is closed upwards in $S$ and $S$ is additive. It is clear that no element of $S'$ is a subset of an element of 
$S\setminus S'$, hence, by Lemma~\ref{L:inadd}(i), $S\setminus S'$ is an additive family of faces of $S'\vec{q}$ and we have the first part of \eqref{E:sqsq}.

It remains to prove the second part of \eqref{E:sqsq}.

Note that, for $s\in S'$ and $t\in T$, the set $s\cup t$ is not a face of $\vec{q}$ since otherwise, by (II), it would be in $T$ and it would contain $s$ contradicting the assumption that $s\in S'$. So, from Lemma~\ref{L:inadd}(ii), $T$ is an additive family of faces of $S'\vec{q}$, $S'$ is an additive family of faces of $T\vec{q}$, and we have 
\begin{equation}\label{E:cract} 
S'T\vec{q} = TS'\vec{q}. 
\end{equation}

Clearly $S'\subseteq S\setminus T$, so, from Lemma~\ref{L:desco}(i), we have 
\begin{equation}\notag 
S'\subseteq \pi^{-1}(S). 
\end{equation} 
Furthermore, again by Lemma~\ref{L:desco}(i), since $S'$ is upward closed in $S$ and since $\min X$ in the statement of Lemma~\ref{L:desco}(i) is in $T$, we see that 
no element of the family $\pi^{-1}(S)\setminus S'$ contains a face of $S'$ as a subset. Also, as proved above, $S'$ is additive in $T\vec{q}$. So, by Lemma~\ref{L:inadd}(i), we can write 
\begin{equation}\label{E:pidi} 
\pi^{-1}(S)\, T \vec{q} = (\pi^{-1}(S)\setminus S')\, S'\, T \vec{q}. 
\end{equation}

Combining \eqref{E:cract} and \eqref{E:pidi}, we get 
\begin{equation}\label{E:pidi2}
\pi^{-1}(S)\, T \vec{q} = (\pi^{-1}(S)\setminus S')\, T\,S' \vec{q}. 
\end{equation} 

By Lemma~\ref{L:desco}(i) since $\min X$ in this lemma are in $T$, we have $\pi^{-1}(S')= S'$, so 
\[
\pi^{-1}(S)\setminus S' = \pi^{-1}(S\setminus S').
\]
This equality and \eqref{E:pidi2} imply the second part of \eqref{E:sqsq}. 
\end{proof}

Observe that conditions (I), (II), and (III) continue to hold for $S'\vec{q}$, $S\setminus S'$, and $T$. Indeed, because of Lemma~\ref{L:ggaa}, we only need to check (II). 
For $s\in S\setminus S'$ and $t\in T$, if $s\cup t$ is a face of $S'\vec{q}$, then it is a face of $\vec{q}$, so it belongs to $T$ since (II) holds for $S$ and $T$. 
Set 
\[
\vec{Q} = S'\vec{q} \;\hbox{ and } \pi' = \pi^{\vec{Q}}_{p,T}\colon T\vec{Q}\to \vec{Q}. 
\]
By what was observed above, Lemma~\ref{L:desco} applies to $\pi'$, $S\setminus S'$, and $T$. By Lemma~\ref{L:desco}(ii), we see that 
\[
(\pi')^{-1}(S\setminus S') = \pi^{-1}(S\setminus S'), 
\]
from which we get 
\begin{equation}\label{E:sspi} 
(S\setminus S')\, \pi'\,\colon \pi^{-1}(S\setminus S') \,T \,\vec{Q} \to (S\setminus S')\,\vec{Q}.
\end{equation}
Now by \eqref{E:sqsq}, we see that 
\[
(S\setminus S')\, \pi'\colon \pi^{-1}(S)\, T \vec{q} \to S\vec{q}. 
\]
Applying Lemma~\ref{L:desco}(ii) to $\pi$ and $\pi'$, we we have
\[
(S\setminus S')\, \pi' = S\, \pi.
\]
Thus, to see that $S\pi$ is a pure weld-division map, it suffices to show that $(S\setminus S')\, \pi'$ as in \eqref{E:sspi} is a pure weld-division map. This amounts to being able to assume that $S'=\emptyset$. Consequently, from this point on, we will assume that 
\begin{equation}\label{E:redst2}
\hbox{for each }s\in S,\hbox{ there exists } t\in T\hbox{ with } s\subseteq t.
\end{equation}

Let now 
\[
S_0 = \{ s\in S\mid p\not\in s\}. 
\]
We show that we can assume that $S_0=\emptyset$. 

\begin{lemma}\label{L:cccd} 
$S\setminus S_0$ is an additive family of faces of $\vec{q}$, $S_0$ is an additive family of faces of $(S\setminus S_0)\, \vec{q}$, $S_0$ is an additive family of faces of $\pi^{-1}(S\setminus S_0) \, T \vec{q}$, and 
\begin{equation}\label{E:gart}
S\,\vec{q}= S_0 \,(S\setminus S_0) \,\vec{q}  \;\hbox{ and }\;  \pi^{-1}(S)\, T \vec{q} = S_0 \, \pi^{-1}(S\setminus S_0) \, T\, \vec{q}.  
\end{equation} 
\end{lemma} 

\begin{proof} Since $S$ is an additive set of faces of $\vec{q}$, so is $S\setminus S_0$. Also, clearly, no element of $S\setminus S_0$ is included in an element of $S_0$. So  Lemma~\ref{L:inadd}(i) implies that $S_0$ is an additive family of faces of $(S\setminus S_0) \, \vec{q}$ and that the first part of \eqref{E:gart} holds. 

By (iii) and Lemma~\ref{L:desco}(i), $S_0\subseteq S\setminus T \subseteq \pi^{-1}(S)$ and no element of $\pi^{-1}(S)\setminus S_0$ is included in an element of $S_0$. So, by Lemma~\ref{L:inadd}(i), we have that $S_0$ is an additive family of faces of $\pi^{-1}(S\setminus S_0) \, T \vec{q}$ and 
\begin{equation}\notag
\pi^{-1}(S)\,T \vec{q} = S_0 \, \big(\pi^{-1}(S)\setminus S_0\big) \, T\, \vec{q}.
\end{equation} 
Since, by Lemma~\ref{L:desco}(i), we also have $\pi^{-1}(S_0) = S_0$, the second part of \eqref{E:gart} follows. 
\end{proof}

Note that $\pi$, $S\setminus S_0$, and $T$ fulfill conditions (I), (II), (III), so Lemma~\ref{L:desco} is applicable. 
Consider the map 
\begin{equation}\label{E:sso}
(S\setminus S_0)\, \pi \colon \pi^{-1}(S\setminus S_0) \, T \vec{q}\to (S\setminus S_0)\, \vec{q}.
\end{equation} 

Observe that each $s\in S_0$ is a face of $(S\setminus S_0)\,\vec{q}$, since $s$ is a face of $\vec{q}$ and no element of $S\setminus S_0$ is included in $s$, and, by Lemma~\ref{L:desco}(ii), 
\[
s\subseteq {\rm u}\big( (S\setminus S_0)\, \pi \big).
\]
Further, note that, by definition of division, $s$ is a face of $\pi^{-1}(S\setminus S_0) \, T \vec{q}$ as, by (III), no element of $T$ and, by Lemma~\ref{L:desco}(i), no element of $\pi^{-1}(S\setminus S_0)$ is included in $s$. 
It follows, by Lemma~\ref{L:purediva}, that the map 
\[
S_0  \big((S\setminus S_0)\, \pi \big) \colon S_0 \, \pi^{-1}(S\setminus S_0)\,\vec{q} \to S_0 \,(S\setminus S_0) \vec{q}
\]
is a pure weld-division map if $(S\setminus S_0)\, \pi$ is a pure weld-division. Now, by Lemma~\ref{L:cccd}, the maps $S_0  \big((S\setminus S_0)\, \pi \big)$ and $S\,\pi$ have the same domains and codomains, and by Lemma~\ref{L:desco}(ii) applied to $S\,\pi$ and $(S\setminus S_0)\,\pi$, we have that 
\[
S_0  \big((S\setminus S_0)\, \pi \big) = S\,\pi. 
\]
Therefore, to see that $S\,\pi$ is a pure weld-division map, it suffices to see that the map $(S\setminus S_0)\, \pi$ in \eqref{E:sso} is a pure weld-division map, that is, we can assume that $S= S\setminus S_0$, so $S_0=\emptyset$. From this point on, we assume $S_0=\emptyset$, that is, 
\begin{equation}\label{E:pins} 
\hbox{for each }s\in S,\, p\in s. 
\end{equation}

\subsection{Assumptions on $S$ and $T$}\label{Su:linord} 

This section collects the assumptions on $S$ and $T$ we will be making based in the first reduction, 

We have two additive families $S$ and $T$ of faces of the sequence $A$ and a vertex $p$ of $A$ fulfilling conditions (I), (II), and (III).  
In the light of the statements \eqref{E:redst}, \eqref{E:redst2}, and \eqref{E:pins}, 
we will be making the following additional assumptions: 
\begin{enumerate}
\item[(IV)] $\hbox{for each }t\in T,\hbox{ there exists } s\in S\hbox{ with } s\subseteq t$; 

\item[(V)] $\hbox{for each }s\in S,\hbox{ there exists } t\in T\hbox{ with } s\subseteq t$; 
\item[(VI)] $\hbox{for each }s\in S,\, p\in s$.
\end{enumerate}
Note that (III) follows from (IV) and (VI) above.

\subsection{Concatenating sequences}\label{Su:conc}

Before we formulate and prove the second reduction, we need to introduce some conventions regarding concatenations 
of sequences of sets.

We fix a linear order $\preceq$ on $S\cup T$ that is a linear extension of $\subseteq$. 
Later on, we will impose some additional conditions on $\preceq$, but at this point we only need it to be 
an extension of the inclusion relation. 

It will be convenient to introduce the following piece of notation. Let $X\subseteq S\cup T$, and let $r_1\preceq \cdots \preceq r_l$ be an 
increasing enumeration of $X$. For a function $\phi$ defined on $X$ whose values are sequences, let 
\[
\prod_{r\in X} \phi(r)\;\hbox{ and }\; \prod^*_{r\in X} \phi(r)
\]
stand for the sequences 
\[
\phi(r_1)\cdots \phi(r_l)\;\hbox{ and }\; \phi(r_l)\cdots \phi(r_1),
\]
respectively, that is, $\prod_{r\in X} \phi(r)$ stands for the concatenation of the sequences $\phi(r_1), \dots, \phi(r_l)$ and $\prod^*_{r\in X} \phi(r)$ 
stands for the concatenation in the reverse order of the same sequences. Furthermore, if $X$ is given by a property $P(r)$ for $r\in S\cup T$, that is, 
\[
X = \{ r\in S\cup T\mid P(r)\hbox{ holds}\},
\]
we write 
\[
\prod_{r: P(r)} \phi(r)\;\hbox{ and }\; \prod^*_{r: P(r)} \phi(r), 
\]
for $\prod_{r\in X} \phi(r)$ and $\prod^*_{r\in X} \phi(r)$, respectively.

\subsection{A list of symbols for the calculations that follow} 

We list pieces of notation that will be used in the computation below. In the list, $s$ varies over $S$, $t$ and $\tau$ vary over $T$, $X$ varies over $\mathcal T$, $r$ is a set, and $\preceq$ is the fixed linear order on $S\cup T$ that extends the inclusion relation. The precise definitions of the notation listed below will be given later. The list is intended as a reference for ease of reading of what follows. 

\[
\begin{split}
s^t &= \hbox{ the largest with respect to $\preceq$  element $s$ of $S$ with $s\subseteq t$};\\
t_s &= \hbox{ the smallest with respect to $\preceq$ element $t$ of $T$ with $s\subseteq t$};
\end{split}
\]

\begin{align*}
&\un{s} = s\setminus \{ p\};\\
&r(t) = r\cup \{ t\}; 
&&\lfloor r \rfloor = \prod^*_{t: r\subseteq t} r(t);\\
&r(X) = r\cup X; 
&&r[t] = \{ r(X)\mid X\in {\mathcal T}, \, t=\min X\};
&&\lceil r\rceil = \prod^*_{t:r\subseteq t} r[t];\\ 
&\ov{t} = (t\setminus s^t)\cup \{ s^t\};
&&\hat{t} = \ov{t}\cup \{ p\};\\
&\lfloor \un{s}\rfloor_t = \prod^*_{\tau :t\preceq \tau,\, s\subseteq \tau} \un{s}(\tau);
&&\lfloor \un{s}\rfloor_{\succ t} = \prod^*_{\tau :t\prec \tau,\, s\subseteq \tau} \un{s}(\tau).
\end{align*}

\subsection{Second reduction}\label{Su:redu2}

A linear order $\preceq$ on $S\cup T$ that extends the relation of inclusion remains fixed. We also continue to assume that $s$ varies over $S$ and $t$, $\tau$ vary over $T$.

By Lemma~\ref{L:desco}(ii), we are working towards proving that the map 
\[
S\pi\colon \pi^{-1}(S)\, T A\to SA
\]
given by the assignment 
\begin{enumerate} 
\item[---] $T \ni t\to p$; 

\item[---] $s(X)\to s$,\, for $s\in S\setminus T$, $X\in {\mathcal T}$, $s\subseteq \min X$; 

\item[---] $\un{s}(X)\to s$,\, for $s\in S$, $X\in {\mathcal T}$, $s\subseteq \min X$
\end{enumerate} 
is a pure weld-division map. We are using (VI) here to see the above assignment as equal to the one from Lemma~\ref{L:desco}(ii), as under (VI) we have $S_p=S$. 

For $s\in S$, let 
\begin{equation}\label{E:tsde}
t_s =\hbox{ the smallest with respect to $\preceq$ element $t$ of $T$ with $s\subseteq t$.}  
\end{equation} 
The face $t_s$ exists by (V).
We define certain sequences needed for Lemma~\ref{L:hath} below. 
For $s\in S$ and $t\in T$ with $s\subseteq t$, we write 
\[
\un{s}(t) = (s\setminus \{ p\})\cup \{ t\}, 
\]
that is $\un{s}(t)$ is equal to $\un{s}(X)$ from \eqref{E:rxe}, where $X= \{ t\}$. 
We let 
\begin{equation}\label{E:unde} 
\lfloor \un{s}\rfloor= \prod^*_{t: s\subseteq t} \un{s}(t). 
\end{equation}

The following lemma constitutes the reduction of this section.

\begin{lemma}\label{L:hath} 
The map 
\begin{equation}\label{E:pso}
\pi^{-1}(S)\, T \vec{q}\to \big(\prod_{t\in T} P_t\, t\big)\, \vec{q}, 
\end{equation} 
where 
\[
P_t= \prod_{s: t_s=t} \lfloor \un{s}\rfloor\, s, 
\]
given by the assignment 
\begin{equation}\label{E:gis}
\begin{cases}
s(X)\to s,& \hbox{ for } s\in S\setminus T,\, X\in {\mathcal T},\, s\subseteq \min X;\\
\un{s}(X) \to \un{s}(\min X),& \hbox{ for } s\in S,\, X\in {\mathcal T},\, s\subseteq \min X.
\end{cases} 
\end{equation} 
is a pure weld-division map. 
\end{lemma}

Before we proceeding to the proof of Lemma~\ref{L:hath}, we state a result identifying a pure weld-division map. It is a rewording of Lemma~\ref{L:org6} with a stronger conclusion. As explained below the proof of Lemma~\ref{L:org6} gives this stronger conclusion.

\begin{lemma}\label{L:org6pure} 
Let $\vec{q}$ be a sequence.  Let $\vec{p}$ be a sequence of faces of $\vec{q}$ and,  for $i=1, \dots, m$, let 
$r_i, s_{i1}, \dots, s_{in_i}$ be faces of $\vec{q}$ such that $r_i \subseteq s_{ij}$, for all $j\leq n_i$. Assume that 
the sequence 
\begin{equation}\notag
\vec{p}\, (r_1s_{11} \cdots s_{1n_1}) \cdots  (r_ms_{m1} \cdots s_{mn_m})
\end{equation} 
is nondecreasing.
Then $s_{ij}\not= s_{i'j'}$, if $i\not= i'$ or $j\not= j'$, and the map 
\begin{equation}\notag 
\vec{p}\,(r_1s_{11} \cdots s_{1n_1}) \cdots  (r_ms_{m1} \cdots s_{mn_m})\, \vec{q} \to \vec{p}\,r_1\cdots r_m\vec{q}
\end{equation} 
given by the assignment 
\begin{equation}\notag
s_{ij}\to r_i,\;\hbox{ for }i\leq m\hbox{ and }j\leq n_i, 
\end{equation} 
is a pure weld-division map. 
\end{lemma}

\begin{proof} 
We note that pure weld-division maps are closed under composition. So, it is enough to show that the function \eqref{E:later} from the proof of Lemma~\ref{L:org6} is a pure weld-division map. But this is what is done in that proof---all the divisions in it are pure divisions as witnessed by inclusions \eqref{E:later1} and \eqref{E:later2}. 
\end{proof}

\begin{proof}[Proof of Lemma~\ref{L:hath}] Recall that Lemma~\ref{L:desco}(i) lists all the faces of $\pi^{-1}(S)$, and note that, by our convention for assignments, it is understood that 
formula \eqref{E:gis} is augmented by 
\[
S\cup T \ni r\to r\in S\cup T. 
\]

The first step of the proof consists of reorganizing of the two sequences in \eqref{E:pso}. This is done in Claims~\ref{C:org5} and \ref{C:org4} below.

In order to state Claim~\ref{C:org5}, we modify the sequence $P_t$. We define  
\begin{equation}\label{E:ppri} 
P'_t=
\begin{cases}
\big( \prod_{s:s\in S\setminus T, t_s=t} \lfloor \un{s}\rfloor\, s\big) \lfloor \un{t} \rfloor,& \hbox{ if } t\in T\cap S;\\
\prod_{s:s\in S\setminus T, t_s=t} \lfloor \un{s}\rfloor\, s ,& \hbox{ if }t\in T\setminus S. 
\end{cases}
\end{equation} 
Note that for $s\in S$ we have that $s\in T$ if and only if $s= t_s$, so  
\[
t_s=t 
\Leftrightarrow \big( (s\in S\setminus T\hbox{ and } t_s=t)\hbox{ or } s=t\big). 
\]
It follows that 
\begin{equation}\label{E:pt} 
\begin{split} 
P_t &= P_t'\, t,\hbox{ if } t\in T\cap S;\\
P_t &= P_t', \hbox{ if }t\in T\setminus S. 
\end{split}
\end{equation}

\setcounter{claim}{0}

\begin{claim}\label{C:org5} 
We have 
\begin{equation}\label{E:ps2}
\big(\prod_{t\in T} P_t\, t \big)\,  \vec{q} = \big( \prod_{t\in T} P_t' \big)\, T \vec{q}.
\end{equation} 
\end{claim} 

\noindent {\em Proof of Claim~\ref{C:org5}.} It follows immediately from \eqref{E:pt} and the definition of equivalence of sequences that the lefthand side of 
\eqref{E:ps2} is equal to 
\[
\big(\prod_{t\in T} P_t'\, t \big)\,  \vec{q}. 
\]
Thus, by the definition of equivalence among sequences, to see \eqref{E:ps2}, it suffices to show that, for each $t'\in T$, we have
\begin{equation}\label{E:gra}
t' \big( \prod_{t: t\succ t'} P'_t\big)  \big( \prod_{t: t\succ t'} t\big) \vec{q} = \big( \prod_{t: t\succ t'} P'_t\big) \, t' \big( \prod_{t: t\succ t'} t\big) \vec{q}. 
\end{equation} 
If we set 
\[
T_{\succ t'}=\prod_{t: t\succ t'} t, 
\]
then \eqref{E:gra} becomes 
\begin{equation}\label{E:gra2} 
t' \big( \prod_{t: t\succ t'} P'_t\big)\,  T_{\succ t'} \vec{q} = \big( \prod_{t: t\succ t'} P'_t\big) \, t' \,  T_{\succ t'}\vec{q}. 
\end{equation} 

In the proof of \eqref{E:gra2}, we will be repeatedly using Lemma~\ref{L:nonfa}(i) and Proposition~\ref{P:trse}(ii). For $t\succ t'$, by the definition of $P_t'$, each entry $r$ of $P_t'$ is equal to one of the following
\begin{enumerate}
\item[($\alpha$)] $s$ with $t_s=t$ and $s\in S\setminus T$;

\item[($\beta$)] $\un{s}(\tau)$ with $s\subseteq \tau$ and $t_s=t$ and $s\in S\setminus T$;

\item[($\gamma$)] $\un{t}(\tau)$ with $t\subseteq \tau$ if $t\in S$.
\end{enumerate}
Note that for $\tau$ as in ($\beta$) and ($\gamma$), we have $t\preceq \tau$, so, for each set $r=s, \,\un{s}(\tau),\, \un{t}(\tau)$ as in ($\alpha$), ($\beta$), and ($\gamma$), we have that $t'\cup r\subseteq{\rm vr}(T_{\succ t'} \vec{q}\,)$. Thus, to see \eqref{E:gra2}, by the definition of the equivalence of sequences and by Lemma~\ref{L:nonfa}(i), 
it is enough to check that for each such $r$, we have 
\[
r\not\in t',\; t'\not\in r, \hbox{ and } t'\cup r \hbox{ is not a face of }T_{\succ t'} \vec{q}.
\]
We give an argument for this statement. In it, we use letters $s$ and $\tau$ to denote sets as in ($\alpha$), ($\beta$), and ($\gamma$). Since $t'$ is a face of $\vec{q}$, it is clear that for $r$ as in ($\alpha$), ($\beta$), and ($\gamma$), we have $r\not\in t'$ since $s, \tau\not\in {\rm tc}(t')$. If $t'\in r$, then, since $t'\not\in s, t$, we have that for some $\tau$, $t'=\tau$ and (there is $s\in S\setminus T$ with 
$s\subseteq \tau$ and $t_s=t$) or ($t\subseteq \tau$ and $t\in S$). In either case, we have $s\in S$ with $t_s=t$ and $s\subseteq \tau$. Since $\tau\in T$, 
by the definition of $t_s$, we get $t= t_s\preceq \tau=t'$. Hence we get $t\preceq t'$ contradicting our assumption $t'\prec t$. 
It remains to show that $t'\cup r$ is not a face of $T_{\succ t'} \vec{q}$. To do this, we note that all $r$ in question contain $\un{s}$ as a subset for some $s$ with $t_s=t$. Furthermore, since $p\in t'$, 
we have 
\[
\un{s}\cup t'= s\cup t'.
\]
So, in order to get the conclusion, it will be enough to show that $s\cup t'$ is not a face of $T_{\succ t'} \vec{q}$ provided that 
$t_s\succ t'$. 
If $s\cup t'$ is not a face of $\vec{q}$, then it is not a face of $T_{\succ t'} \vec{q}$ since $s\cup t'\subseteq {\rm vr}(\vec{q}\,)$. So assume that 
$s\cup t'$ is a face of $\vec{q}$. In this case, we have $s\cup t'\in T$ by assumption (II). 
So we will be done if we show that $s\cup t'\succ t'$ as then $s\cup t'$ is 
an entry of the sequence $T_{\succ t'}$. Since $s\cup t'\supseteq t'$, we have $s\cup t'\succeq t'$. 
Thus, we only need to rule out the possibility 
that $s\cup t'= t'$, that is, $s\subseteq t'$. But, by the defining property of $t_s$, 
this inclusion implies that $t_s\preceq t'$, which contradicts the assumption 
$t_s\succ t'$, and the claim follows.

\medskip

The goal now is to give a convenient for our arguments non-decreasing enumeration of $\pi^{-1}(S)$. This will be accomplished in Claim~\ref{C:org4}. 
Our immediate goal is to define sequences \eqref{E:rpri} below in analogy with \eqref{E:ppri}. 
These sequences will be used to provided the desired enumeration of $\pi^{-1}(S)$. 

For $t\in T$ and $r\subseteq t$, define 
\[
\begin{split}
r[t] \;\; = \{ r(X)\mid X\in {\mathcal T}, \, t=\min X\}. 
\end{split} 
\]
We will use this definition only for $r= \un{s}$ and $r=s$ with $s\in S$. 
We note that, by Lemma~\ref{L:linear}, $r[t]$ are additive sets of faces of $T\vec{q}$, so the choices of their non-decreasing enumerations are not material when using them to 
implement divisions. 
Observe further that assignment \eqref{E:gis} maps all elements of $s[t]$ to $s$, for $s\in S\setminus T$, 
and all elements of $\un{s}[t]$ to $\un{s}(t)$, for $s\in S$. 

Define for $r$ such that $r\subseteq t$, for some $t\in T$, 
\[
\lceil r\rceil= \prod^*_{t:r\subseteq t} r[t]. 
\]
Observe that 
\begin{equation}\label{E:nod1}
\lceil r\rceil\;\hbox{ is non-decreasing.} 
\end{equation}
Indeed it suffices to show that if $r\subseteq t, t'$ and $t\prec t'$, then no set in $r[t]$ is included in a set in $r[t']$. Sets in $r[t]$ have the form $r(X)$ with $\min X = t$ and sets in $r[t']$  have the form $r(X')$ with $\min X'=t'$. If $r(X)\subseteq r(X')$, then, since $t\not\in r$, we have that $t$ is in $X'$ and so $t'=\min X'\subseteq t$ contradicting $t\prec t'$.

For $t\in T$,  define  
\begin{equation}\label{E:rpri}
R'_t =
\begin{cases}
\big(\prod_{s:s\in S\setminus T, t_s=t}\, \lceil \un{s}\rceil\, s\, \lceil s\rceil\big) \lceil \un{t}\rceil, &\hbox{ if }t\in S;\\
\prod_{s:s\in S\setminus T, t_s=t}\, \lceil \un{s}\rceil\, s\, \lceil s\rceil, &\hbox{ if } t\not\in S. 
\end{cases}
\end{equation} 
Note that 
\begin{equation}\label{E:nod2}
R'_t \hbox{ is non-decreasing.} 
\end{equation}
Indeed, fix $t$. It follows from \eqref{E:nod1} and the obvious observation using (VI) that $s\not\subseteq \un{s}(X)$ and $s(X)\not\subseteq s$, for $X\in {\mathcal T}$, that $\lceil \un{s}\rceil\, s\, \lceil s\rceil$ is nondecreasing. So, it suffices to show that if $s\prec s'$ and $t_s= t_{s'}=t$, then no set in $\lceil \un{s'}\rceil\, s'\, \lceil s'\rceil$ or in $\lceil \un{t}\rceil$
 is contained in a set in $\lceil \un{s}\rceil\, s\, \lceil s\rceil$. Note that each set in $\lceil \un{s'}\rceil\, s'\, \lceil s'\rceil$ contains $\un{s'}$. So if some set in $\lceil \un{s'}\rceil\, s'\, \lceil s'\rceil$ was contained in some set in $\lceil \un{s}\rceil\, s\, \lceil s\rceil$, we would have $\un{s'}\subseteq s$. Since, by (VI), $p\in s$, we would then have $s'\subseteq s$ contradicting $s\prec s'$. Similarly, if some set in $\lceil \un{t}\rceil$ is contained in a set in $\lceil \un{s}\rceil\, s\, \lceil s\rceil$, then $t\subseteq s$. Since $s'\subseteq t_{s'}=t$, we would again have $s'\subseteq s$ leading to a contradiction with $s\prec s'$.

Now, we observe that 
\begin{equation}\label{E:rrnod} 
\prod_{t\in T} R_t'\hbox{ is nondecreasing}. 
\end{equation}
To see this we use \eqref{E:nod2}, and we check that if $t\prec t'$, then no set in 
 $R'_{t'}$ is included in a set in $R'_t$, that is, no set in $\lceil \un{s'}\rceil\, s'\, \lceil s'\rceil$ with $t_{s'}= t'$ or in $\lceil \un{t'}\rceil$, if $t'\in S$,  is included in a set in $\lceil \un{s}\rceil\, s\, \lceil s\rceil$ with $t_{s}= t$ or in $\lceil \un{t}\rceil$, if $t\in S$. It suffices to see that no 
 set of the form $\un{s'}$ with $t_{s'}= t'$ or of the form $\un{t'}$ is included in a set of the form $s$ with $t_s=t$ or of the form $t$. But this is clear, since otherwise, in the first case, we would have $s'\subseteq s$ or $s'\subseteq t$, which, by the definition of $t_{s'}$, would imply $t'= t_{s'}\preceq t_s=t$ or $t'=t_{s'}\preceq t$, reaching a contradiction with $t\prec t'$. In the second case, we would have $t'\subseteq s$ or $t'\subseteq t$ from which we again a contradiction with $t\prec t'$.

\begin{claim}\label{C:org4} 
We have 
\begin{equation}\label{E:ps}
\pi^{-1}(S)\, T \vec{q}= \big( \prod_{t} R_t'\big)\, T \vec{q}. 
\end{equation} 
\end{claim} 

\noindent {\em Proof of Claim~\ref{C:org4}.}
Since $\pi^{-1}(S)$ is additive in $T\vec{q}$, all non-decreasing enumerations of it give the same result when applied to $T\vec{q}$. 
Using Lemma~\ref{L:desco}(i), by inspection, we see that the sequence $\prod_{t\in T} R_t'$ lists all elements of $\pi^{-1}(S)$. The claim follows from \eqref{E:rrnod}.

\medskip

Claims~\ref{C:org5} and \ref{C:org4} reduce proving the lemma to showing that the map 
\begin{equation}\label{E:srtp} 
 \big( \prod_{t} R_t'\big)\, T\vec{q} \to \big(\prod_{t} P_t' \big)\,T \vec{q}, 
\end{equation} 
defined by the assignment \eqref{E:gis} is a pure weld-division map. This is what we show in the remainder of this proof. 
It is helpful to keep in mind that assignment \eqref{E:gis} maps all elements of $\lceil s\rceil$ to $s$ and it maps 
elements of $\lceil \un{s}\rceil$ to elements of $\lfloor \un{s}\rfloor$; therefore, 
\eqref{E:gis} maps $R_t'$ to $P_t'$.

We will use Lemma~\ref{L:org6pure} and the closure of pure weld-division maps under composition to obtain the following claim. 

\begin{claim}\label{C:org7} 
Let $P$ and $Q$ be sequences of faces of $T\vec{q}$ and entries of $s\lceil s\rceil$ and $\lceil \un{s}\rceil$ are faces of $QT\vec{q}$. Assume that 
\[
P\, s\lceil s\rceil\, Q T\vec{q}\;\hbox{ and }\; P\, \lceil \un{s}\rceil\, QT\vec{q}
\]
are non-decreasing. 
\begin{enumerate}
\item[(i)] For $s\in S\setminus T$, the map 
\[
P\, s\lceil s\rceil\, Q T\vec{q} \to P\, s \,QT\vec{q}
\]
given by the restriction of \eqref{E:gis} to $\lceil s\rceil$ is a pure weld-division map.

\item[(ii)] For $s\in S$, the map 
\[
P\, \lceil \un{s}\rceil\, QT\vec{q} \to P\,\lfloor \un{s}\rfloor\, QT\vec{q}, 
\]
given by the restriction of \eqref{E:gis} to $\lceil \un{s}\rceil$ is a pure weld-division map.
\end{enumerate}
\end{claim} 

\noindent {\em Proof of Claim~\ref{C:org7}.} (i)  Recall that 
\[
\lceil s\rceil= \prod^*_{t:s\subseteq t} s[t]
\]
and note that the set $s$ is a subset of each set in $\lceil s\rceil$. 
Thus, point (i) follows directly from Lemma~\ref{L:org6pure} with $m=1$ and $r_1=s$. 

(ii) Recall that 
\[
\lceil \un{s}\rceil= \prod^*_{t:s\subseteq t} \un{s}[t],
\]
and observe that, for each $t\supseteq s$, the set $\un{s}(t)$ is included in each entry of the sequence $\un{s}[t]$ and so it is also 
the first entry of this sequence. Thus, by an application of Lemma~\ref{L:org6pure}, we get that the map 
\[
P\, \lceil \un{s}\rceil\, QT\vec{q} \to P \big(\prod^*_{t:s\subseteq t} \un{s}(t)\big) QT\vec{q}= P \,  \lfloor \un{s}\rfloor\,QT\vec{q}, 
\]
given by the assignment \eqref{E:gis} ispure weld-division map, and the claim follows.

\smallskip

Now, by inspecting the definitions of $P_t'$ and $R_t'$ from \eqref{E:ppri} and \eqref{E:rpri}, we check that an iterative 
application of Claim~\ref{C:org7} combined with closure of pure weld-division maps under composition imply that the map \eqref{E:srtp} is a pure weld-division map. This iteration is implemented as follows. In this paragraph, the words "largest" and "smallest" refer to the order $\preceq$. 
Starting with the largest $t\in T$ and proceeding to the smallest $t\in T$, we map elements of $R_t'$ to elements of $P_t'$, that is, we produce maps 
\[
 \big( \prod_{t'\prec t} R_{t'}'\big)\, R_t'\big( \prod_{t\prec t'} P_{t'}'\big)\, T\vec{q} \to  \big( \prod_{t'\prec t} R_{t'}'\big)\big(\prod_{t\preceq t'} P_{t'}' \big)\,T \vec{q}. 
\]
Fix $t\in T$, and consider $R_t'$. If $t\in T\cap S$, 
\[
R_t'= \big(\prod_{s:s\in S\setminus T, t_s=t}\, \lceil \un{s}\rceil\, s\, \lceil s\rceil\big) \lceil \un{t}\rceil.
\]
First, we map sets in the sequence to $\lceil \un{t}\rceil$ to sets in the sequence $\lfloor \un{t}\rfloor$ using Claim~\ref{C:org7}(ii). Then we proceed from the largest $s\in S\setminus T$ with $t_s=t$ to the smallest such $s$. We map sets in $s\lceil s\rceil$ to $s$ using Claim~\ref{C:org7}(i) and then sets in $\lceil \un{s}\rceil$ to sets in $\lfloor \un{s}\rfloor$ using Claim~\ref{C:org7}(ii). This procedure produces 
\[
P_t' = \big( \prod_{s:s\in S\setminus T, t_s=t} \lfloor \un{s}\rfloor\, s\big) \lfloor \un{t} \rfloor.
\]
If $t\in T\setminus S$, then 
\[
R_t'= \prod_{s:s\in S\setminus T, t_s=t}\, \lceil \un{s}\rceil\, s\, \lceil s\rceil,
\]
and we proceed as above except that we omit the first step dealing with $ \lceil \un{t}\rceil$ as this entry does not appear in $R_t'$ for $t\not\in S$. The application of Claim~\ref{C:org7} requires that the sequences of the form 
\[
 \big( \prod_{t'\prec t} R_{t'}'\big) \big(\prod_{t\preceq t'} P_{t'}' \big)\,T \vec{q}
\]
are non-decreasing, which is handled by the argument used to prove \eqref{E:rrnod}. 
The lemma is proved. 
\end{proof}

\subsection{Stronger conditions on the order $\preceq$}

Recall that $\preceq$ is a linear order on $S\cup T$. The only property of $\preceq$ that has been used so far is that $\preceq$ extends the inclusion relation. Now, we will need to require an additional condition of $\preceq$. We spell it out in \eqref{E:ordad} and prove in Lemma~\ref{L:orb} that such an order can be constructed. 
In analogy with the definition of $t_s$ in \eqref{E:tsde}, for $t\in T$, set 
\begin{equation}\label{stde} 
s^t = \hbox{the largest with respect to $\preceq$  element $s$ of $S$ with $s\subseteq t$}. 
\end{equation} 
Observe that by additivity of $S$ and the fact that $\preceq$ extends $\subseteq$, we have 
\[
s^t = \hbox{the largest with respect to $\subseteq$  element $s$ of $S$ with $s\subseteq t$}. 
\]
Such an $s^t$ exists by (IV).

\begin{lemma}\label{L:orb} 
There exists a linear order $\preceq$ on $S\cup T$ that extends the inclusion relation and 
is such that, for $s\in S$ and $t\in T$, we have 
\begin{equation}\label{E:ordad} 
t\prec t_s\; \Leftrightarrow\; t\prec s\; \Leftrightarrow\; s^t\prec s,\;\hbox{ for all }s\in S,\, t\in T. 
\end{equation} 
\end{lemma} 

\begin{proof} We produce a linear order $\preceq$ on $S\cup T$ extending the inclusion relation and fulfilling 
\begin{equation}\label{E:slare} 
s\preceq t \Rightarrow (s\preceq s^t \hbox{ and } t_s\preceq t). 
\end{equation} 
We claim such an order fulfills \eqref{E:ordad}. Indeed, we get $t\prec t_s\; \Rightarrow\; t\prec s$ from \eqref{E:slare}, while $t\prec t_s\; \Leftarrow\; t\prec s$ from $s\subseteq t_s$ and from $\preceq$ being an extension of $\subseteq$. 
Similarly, we get $t\prec s\; \Leftarrow\; s^t\prec s$ from \eqref{E:slare}, while $t\prec s\; \Rightarrow\; s^t\prec s$ from $s^t\subseteq t$ and from $\preceq$ being an extension of $\subseteq$.

To define such an ordering, let $S_1$ consists of all $s\in S$ such that 
\[
\hbox{for some } t\in T, \, s \hbox{ is the largest with respect to $\subseteq$ element of $S$ with } s\subseteq t, 
\]
that is, $S_1=\{ s^t\mid t\in T\}$. 
Let $\preceq_1$ be a linear order on $S_1$ extending the inclusion relation. 

For $s\in S_1$, let 
\[
A_s =\{ t\in T\mid s \hbox{ is the largest with respect to $\subseteq$ element of $S$ with $s\subseteq t$}\}.
\]
Note that the sets $A_s$, $s\in S_1$, are pairwise disjoint and their union is equal to $T$ by (IV) and additivity of $S$. Furthermore, for $s'\in S_1$, we have 
$s'\in A_s$ if and only if $s'=s\in T$. 

For $s\in S_1$, let 
\[
B_s= \{ s'\in S\mid s \hbox{ is the smallest with respect to $\preceq_1$ element of $S_1$ with $s'\subseteq s$}\}.
\]
Note that the sets $B_s$, $s\in S_1$, are pairwise disjoint, their union is equal to $S$ by (V). Furthermore, for $s'\in S_1$, we have 
$s'\in B_s$ if and only if $s'=s$.

From the observations above, we see that
\begin{equation}\label{E:lin1}
\begin{split}
 &\hbox{the sets $A_s\setminus \{ s\}$, $s\in S_1$, partition $T\setminus S_1$;}\\  
 &\hbox{the sets $B_s\setminus \{ s\}$, $s\in S_1$, partition $S\setminus S_1$;}\\
 &\hbox{the sets $\{ s\}$, $s\in S_1$, partition $S_1$.}
 \end{split}
\end{equation} 

Further, we make the following observation. For $s_1, s_2\in S_1$, we have 
\begin{equation}\label{E:lin2}
\begin{split}
\big(s_1\prec_1 s_2,\; t\in A_{s_1},\; s\in B_{s_1}, \; t'\in A_{s_2},\; & s'\in B_{s_2}\big) \Rightarrow\\
\big( &t'\not\subseteq t,\; t'\not\subseteq s,\; s'\not\subseteq t,\; s'\not\subseteq s\big). 
\end{split}
\end{equation} 
We prove the above non-inclusions by contradiction. If $t'\subseteq t$, then, since $t'\in A_{s_2}$, we have $s_2\subseteq t$, from which we get $s_2\subseteq s_1$ since $t\in A_{s_1}$. Then, $s_2\preceq_1 s_1$, contradicting $s_1\prec_1 s_2$. 
If $t'\subseteq s$, then $s_2\subseteq s$ since $t'\in A_{s_2}$, so $s_2\subseteq s_1$ since $s\in B_{s_1}$, and we again get $s_2\preceq_1 s_1$ reaching a contradiction. If $s'\subseteq t$, then $s'\subseteq s_1$ since $t\in A_{s_1}$. So, since $s'\in B_{s_2}$, we get $s_2\preceq s_1$, which gives a contradiction with $s_1\prec_1 s_2$. Finally, if $s'\subseteq s$, then $s'\subseteq s_1$ since $s\in B_{s_1}$, which gives the same contradiction as in the last sentence. Thus, \eqref{E:lin2} is proved.

Note that, for each $s\in S_1$, we have 
\begin{equation}\label{E:lin3}
(t'\in A_s,\; s'\in B_s) \Rightarrow s'\subseteq s\subseteq t'.
\end{equation} 

Finally observe 
\begin{equation}\label{E:lin4} 
S\cap T\subseteq S_1.
\end{equation}

Taking into account \eqref{E:lin1}, \eqref{E:lin2}, \eqref{E:lin3}, \eqref{E:lin4}, we see that we can define 
$\preceq$ to be the linear order on $S\cup T$ extending $\preceq_1$ on $S_1$, extending the inclusion relation, 
and having the following two properties for each $s\in S_1$: 
\begin{enumerate}
\item[---] $A_s$ is an interval with respect to $\preceq$ such that $s$ is the largest element of $S\cup T$ with respect to $\preceq$ with 
$s\preceq s'$ for all $s'\in A_s$;

\item[---] $B_s$ is an interval with respect to $\preceq$ such that $s$ is the smallest element of $S\cup T$ 
with respect to $\preceq$ with $s'\preceq s$ for all $s'\in B_s$.
\end{enumerate} 

We check that this linear order fulfills \eqref{E:slare}. 

Assume, towards a contradiction, that there is $s\in S$ and $t\in T$ such that $s^t\prec s\preceq t$. Note that $t\in A_{s^t}$. So, by the definition of $\preceq$, $s\in A_{s^t}$, which means, by the definition of $A_{s^t}$, that $s\in T\cap S$ and $s=s^t$, contradicting $s^t\prec s$. We proved 
\begin{equation}\label{E:inef1}
s\preceq t\Rightarrow s\preceq s^t,\;\hbox{ for all }s\in S,\; t\in T.
\end{equation}

Assume, towards a contradiction, that there are $s\in S$ and $t\in T$ such that $s\preceq t\prec t_{s}$. 
Then $s\in B_{s'}$, for some $s'\in S_1$. Note that $t\not\in B_{s'}\setminus \{s'\}$ since this last set does not contain elements of $T$, which follws from \eqref{E:lin4} and the observation that the only element of $B_{s'}\cap S_1$ is $s'$. By definition of $\preceq$, we have $s'\preceq t$.
So, $t$ is an element of $T$ such that $s'\preceq t\prec t_s$. 
From this assertion, by the definition of $\preceq$, we get an element $t'$ of $T$ with $t'\in A_{s'}$ and $t'\prec t_s$. But then $s\subseteq s'\subseteq t'$, which implies $t_{s}\preceq t'$ from the definition of $t_s$, and we reached a contradiction. Thus, we proved 
\begin{equation}\label{E:inef2}
s\preceq t\Rightarrow t_s\preceq t,\;\hbox{ for all }s\in S,\; t\in T.
\end{equation}
The lemma follows from \eqref{E:inef1} and \eqref{E:inef2}. 
\end{proof}

\subsection{Proof of Main Lemma}

It may be helpful to recall the definitions of $\lfloor \un{s}\rfloor$ from \eqref{E:unde} and $P_t$ from Lemma~\ref{L:hath}.  

By Lemma~\ref{L:hath}, with the notation from that lemma, it suffices to show that the map 
\begin{equation}\label{E:spa}
\big(\prod_{t} P_t\, t\big)\vec{q}\to  S\vec{q}
\end{equation} 
given by the assignment 
\begin{equation}\label{E:lasa} 
\begin{split}
&t\to p,\\
&\un{s}(t) \to s,\hbox{ for } s\subseteq t
\end{split}
\end{equation}
is a composition of combinatorial isomorphisms and weld maps. It may be convenient to fully spell out this assignment; in addition to the rules \eqref{E:lasa}, we have 
\begin{equation}\label{E:lasa2}
S\setminus T \ni s\to s\in S\setminus T. 
\end{equation} 
Note that the sets occurring as entries in the sequence $\prod_{t\in T} P_t\, t$ are precisely the sets in the family 
\[
S\cup T\cup \{ \un{s}(t) \mid s\in S, \, t\in T,\, s\subseteq t\};
\]
thus, the assignment given by \eqref{E:lasa} and \eqref{E:lasa2} does induce a map with the domain and codomain as in \eqref{E:spa}.

From this point on, we assume that variables $t, \tau$ run over the set $T$ and $s,\sigma$ over the set $S$.

Define, for $t\in T$, 
\begin{equation}\label{E:uneqa}
\lfloor \un{s}\rfloor_t = \prod^*_{\tau :t\preceq \tau,\, s\subseteq \tau} \un{s}(\tau)\;\hbox{ and }\;  \lfloor \un{s}\rfloor_{\succ t} = 
\prod^*_{\tau :t\prec \tau,\, s\subseteq \tau} \un{s}(\tau).
\end{equation} 
We note that 
\begin{equation}\label{E:flfl}
\lfloor \un{s}\rfloor_t =  \lfloor \un{s}\rfloor,\; \hbox{ if } t\preceq t_s.
\end{equation} 
To see \eqref{E:flfl}, recall the definition \eqref{E:unde} of $ \lfloor \un{s}\rfloor$ and note that it suffices to see that $t\preceq \tau$ follows from $s\subseteq \tau$ and $t\preceq t_s$. But 
$s\subseteq \tau$ implies $t_s\preceq \tau$ by the definition \eqref{E:tsde} of $t_s$, and $t\preceq \tau$ follows. 
Note further that 
\begin{equation}\label{E:ungr} 
\lfloor \un{s}\rfloor_t = 
\begin{cases} 
\lfloor \un{s}\rfloor_{\succ t}\, \un{s}(t), &\hbox{ if }s\subseteq t;\\
\lfloor \un{s}\rfloor_{\succ t}, &\hbox{ if } s\not\subseteq t.
\end{cases} 
\end{equation} 
Finally, it is easy to see that the sequences $\lfloor \un{s}\rfloor_t$ and $\lfloor \un{s}\rfloor_{\succ t}$ are non-decreasing.

Define  
\[
\ov{t}= (t\setminus s^t)\cup \{ s^t\}\;\hbox{ and }\; \hat{t}= (t\setminus \un{s}^t)\cup \{ s^t\}. 
\]
With this notation, for $s\subseteq t$, let 
\[
\overline{t}(s) = \overline{t}\cup \{ s\}.
\]
Note that $\ov{t}(s^t) = \ov{t}$.

\begin{lemma}\label{L:cruc} Fix $t$. Let 
\[
\vec{b_t} =  \big( \prod_{\tau: t\prec \tau} P_\tau\, \tau\big) \, \vec{q}.
\]
We have a combinatorial isomorphism 
\begin{equation}\label{E:coim} 
\big(\prod_{s:s\preceq t} \lfloor \un{s}\rfloor_t\, s\big)\, t\, \vec{b_t} 
\to
\big( \prod^*_{s: s\subseteq t} \overline{t}(s)\big)\, \hat{t}\,\big( \prod_{s: s\preceq t} \lfloor \un{s}\rfloor_{\succ t} \, s\big)\,\vec{b_t}
\end{equation} 
given by the assignment 
\begin{equation}\label{E:finassi} 
t\to \hat{t}\;\hbox{ and }\;\un{s}(t)\to \ov{t}(s),\hbox{ for $s$ with $s\subseteq t$}. 
\end{equation} 
\end{lemma}

\begin{proof} The face $t\in T$ will remain fixed in this proof, and we set $\vec{b}= \vec{b_t}$. To ease reading the formulas in this proof, we introduce the following piece of notation. 
Let $\vec{c}$ be a sequence. For two sequences $\vec{s}$ and $\vec{t}$, we write 
\[
\vec{s} \ec \vec{t}
\]
for $\vec{s}\,\vec{c}= \vec{t}\,\vec{c}$, and, similarly, we write 
\[
\vec{s}\ac \vec{t}
\]
to indicate a map $\vec{s}\,\vec{c}\to \vec{t}\,\vec{c}$.

Let $L$ be 
\[
L= \big(\prod_{s:s\preceq t} \lfloor \un{s}\rfloor_t\, s\big) t, 
\]
so $L\vec{b}$ is the left hand side of \eqref{E:coim}.
The proof consists of transforming $L$ until the right hand side is reached. 
We start with rewriting $L$, using \eqref{E:ordad} to get the first equality and using \eqref{E:ungr} and the relation $s^t\subseteq t$ to get the second one, 
\begin{equation}\label{E:rew}
\begin{split} 
L &\eb
\big(\prod_{s:s\prec s^t} \lfloor \un{s}\rfloor_t\, s\big)  \big(  \lfloor \un{s^t}\rfloor_t\, s^t\big) t\\ 
&\eb \big(\prod_{s:s\prec s^t} \lfloor \un{s}\rfloor_{\succ t}\, M_{s}\big)  \big( \lfloor \un{s^t}\rfloor_{\succ t} \, \un{s}^t(t)\, s^t\, t\big), 
\end{split} 
\end{equation} 
where 
\[
M_{s}= \begin{cases}
\un{s}(\,t\,) s,&\text{ if }s\subseteq t;\\
s, &\text{ if } s\not\subseteq t.
\end{cases}
\]

Observe that $t$ is a face of $\vec{b}$ since $t$ is a face of $\vec{q}$, no $\tau$ with $t\prec \tau$ is included in $t$, and no entry of $P_\tau$ with $t\prec \tau$ is included in $t$. Further, note that the relations $s^t\subseteq t$ and $\ov{t}(s^t)= \ov{t}$. So, we can apply Lemma~\ref{L:commut}(ii) (with $r= \un{s}^t$, $s= s^t$, $t=t$) to the rightmost three entries in the sequence in \eqref{E:rew}. This way we get, for each sequence $\vec{c}$ extending $\vec{b}$, a combinatorial isomorphism 
\begin{equation}\label{E:neve} 
 \un{s}^t(t)\, s^t\, t \ab \ov{t}(s^t)\, \hat{t}\, s^t
\end{equation} 
implemented by the assignment 
\begin{equation}\label{E:nevea} 
 t\to \hat{t}   \;\hbox{ and }\;   \un{s}^t(t)\to \ov{t}(s^t). 
\end{equation}
We observe that in the sequence in \eqref{E:rew} no entry to the left of $\un{s}^t(t)\, s^t\, t$ contains $\un{s}^t(t)$ as an element and the only such entry containing $t$ as an element is $\un{s}(t)$, which appears in $M_s$ if $s\subseteq t$. Thus, Lemma~\ref{L:mapdif}, the combinatorial isomorphism \eqref{E:neve} given by the assignment \eqref{E:nevea} gives the combinatorial isomorphism 
\begin{equation}\label{E:rewb} 
L \ab \big(\prod_{s\prec s^t} \lfloor \un{s}\rfloor_{\succ t}\, N_{s}\big) \big(  \lfloor \un{s^t}\rfloor_{\succ t} \big) \big(  \ov{t}(s^t)\, \hat{t}\, s^t\big)
\end{equation} 
via the assignment 
\begin{equation}\label{E:assi1}
t\to \hat{t}, \;\; \un{s}^t(t) \to \ov{t}(s^t),\;\; \un{s}(t) \to \un{s}(\hat{t}\,),\hbox{ for }s\prec s^t,\, s\subseteq t,  
\end{equation}
where 
\[
N_{s}= \begin{cases}
\un{s}(\,\hat{t}\,) s,&\text{ if }s\subseteq t;\\
s, &\text{ if } s\not\subseteq t.
\end{cases}
\]
Note that $N_s$ is obtained from $M_s$ using Lemma~\ref{L:mapdif} by applying $t\to \hat{t}$.

\setcounter{claim}{0}
The following claim lists the needed commutativity relations for $\ov{t}(\sigma)$ and $\hat{t}$. 
We say that a sequence $\vec{c}$ {\bf extends $\vec{b}$} if 
\[
\vec{c}= \vec{b}'\vec{b}\,\hbox{ for some sequence }\vec{b}'.
\]

\begin{claim}\label{Cl:tasi}  
Fix $s \preceq t$. Let $\vec{c}$ be a sequence extending $\vec{b}$. 

The following relations hold for $\ov{t}(\sigma)$ with $s\preceq \sigma\subseteq t$:  
\begin{enumerate} 
\item[(a)] 
$\un{s}(\tau )\,  \ov{t}(\sigma )\ec  \ov{t}(\sigma )\, \un{s}(\tau )$,  for $\tau$ with $t\prec \tau$ and $s\subseteq \tau$;

\item[(b)] $\un{s}(\,\hat{t}\,)\, \ov{t}(\sigma ) \ec  \ov{t}(\sigma )\, \un{s}(\,\hat{t}\,)$;

\item[(c)] $s\, \ov{t}(\sigma ) \ec  \ov{t}(\sigma )\, s$,\, if $s\not=\sigma$.
\end{enumerate} 

The following relations hold for $\hat{t}$: 
\begin{enumerate}
\item[(d)] $\un{s}(\tau )\,  \hat{t} \ec  \hat{t} \, \un{s}(\tau )$, for $\tau$ with $t\prec \tau$ and $s\subseteq \tau $;

\item[(e)] $s\,  \hat{t} \ec  \hat{t}\, s$,\, if $s\not\subseteq t$.
\end{enumerate} 
\end{claim}

\noindent {\em Proof of Claim~\ref{Cl:tasi}.} 
Recall the definition of combinatorial equivalence of sequences from Section~\ref{Su:combs}.  
Recall also that, for $r_1, r_2\in S\cup T$, we have $r_1\not\in {\rm tc}(r_2)$ as $r_1$ and $r_2$ are faces of $\vec{q}$. 
We will use this property freely in the proof below. 

Assume first that $s\subseteq s^t$. We need to check (a), (b), (c) and (d) (since in (e) the condition $s\not\subseteq t$ contradicts $s\subseteq s^t$). 
This check is based on disjointness of the appropriate faces implying combinatorial equivalence of sequences (as in point (b) of the definition of combinatorial equivalence of sequences in Section~\ref{Su:combs}). 
For (a) and (d) note that since $\sigma \subseteq t$, we have $\sigma \subseteq s^t\subseteq t\prec \tau $, so $\sigma \not= \tau \not= s^t$. 
Using this observation and our case assumption $s\subseteq s^t$, we see that $\un{s}(\tau )$ and $\ov{t}(\sigma )$ are disjoint as are $\un{s}(\tau )$ and $\hat{t}$, and (a) and (d) 
follow. In (b), since $\sigma\not= \hat{t}\not= s^t$ and $s\subseteq s^t$, we see that $\un{s}(\,\hat{t}\,)$ and $\ov{t}(\sigma )$ are disjoint. For (c), 
the disjointness of $s$ and  $\ov{t}(\sigma )$ is immediate by the case assumption $s\subseteq s^t$. 

Assume now $s\not\subseteq s^t$. The check of (a), (b), (c), (d) and (e) 
is based on the union of two appropriate faces not being a face implying combinatorial equivalence of sequences (as in point (b) of the definition of combinatorial equivalence of sequences in Section~\ref{Su:combs}). 
We will use the following property of the sequence $\vec{b}$\,: $\sigma$ is an entry of $\vec{b}$ if and only if $s^t\prec \sigma$. 
This property holds since, by \eqref{E:ordad}, for each $\sigma$ the condition 
$t\prec t_\sigma$ is equivalent to $s^t\prec \sigma$. 
Note now that for the pairs of the relevant faces in (a), (b), (c), (d), (e), their unions, that is, 
\begin{equation}\label{E:nofa}
 \un{s}(\tau)\cup  \ov{t}(\sigma), \; \un{s}(\,\hat{t}\,)\cup \ov{t}(\sigma ),\; s\cup \ov{t}(\sigma ),\; \un{s}(\tau )\cup \hat{t}, \; s\cup \hat{t},
\end{equation} 
all contain $\un{s}\cup \{ s^t\}$. So it is enough to show that $\un{s}\cup \{ s^t\}$ is not a face of $\vec{c}$. Let $\vec{c}= c_0\cdots c_n$. If $\un{s}\cup \{ s^t\}$ is a face of $\vec{c}$, then, by the definition of being a face, there exists $i\leq n$ such that $c_i= s^t$ and $\un{s}\cup s^t = s\cup s^t$ is a face of $c_{i+1}\cdots c_n$. Observe that, since $s^t$ is not listed in $\vec{b}$, we see that $\vec{b}$ is an initial segment of $c_{i+1}\cdots c_n$. Since 
\[
s\cup s^t \subseteq {\rm vr}(\vec{q}\,),
\]
we see, by Lemma~\ref{L:nonfa}(i), that $s\cup s^t$ is a face of $\vec{b}$ and of $\vec{q}$. 
Since $S$ is an additive family of faces of $\vec{q}$, it follows that $s\cup s^t$ is in $S$. Since, by our case assumption, $s\cup s^t$ properly contains $s^t$, so $s^t\prec s\cup s^t$, and therefore, it is an entry of $\vec{b}$. It follows, by Lemma~\ref{L:nonfa}(ii), that $s\cup s^t$ is not a face of $\vec{b}$, a contradiction. Thus, the claim is proved.

\smallskip 

Observe that by the definition \eqref{E:uneqa} of $\lfloor \un{s}\rfloor_{\succ t}$ and Claim~\ref{Cl:tasi}~(a) (with $\sigma=s=s^t$) and (d) (with $s=s^t$), we get the following identities, for each sequence $\vec{c}$ extending $\vec{b}$, 
\begin{equation}\notag
\big( \lfloor \un{s}^t\rfloor_{\succ t}\big) \, \ov{t}(s^t) \ec  \ov{t}(s^t)\,  \big( \lfloor \un{s}^t\rfloor_{\succ t}\big)\;\hbox{ and }\; 
\big( \lfloor \un{s}^t\rfloor_{\succ t}\big) \, \hat{t} \ec  \hat{t}\,  \big( \lfloor \un{s}^t\rfloor_{\succ t}\big).
\end{equation}
An application of these identities to \eqref{E:rewb} gives a combinatorial isomorphism 
\begin{equation}\label{E:rewc} 
L\ab  \big(\prod_{s\prec s^t} \lfloor \un{s}\rfloor_{\succ t}\, N_{s}\big) \, \ov{t}(s^t)\, \hat{t} \, \big(  \lfloor \un{s}^t\rfloor_{\succ t} \, s^t\big)
\end{equation}
given by assignment \eqref{E:assi1}.

We now streamline the information gathered in Claim~\ref{Cl:tasi} to apply it to \eqref{E:rewc}. 

\begin{claim}\label{Cl:stre} Let $\vec{c}$ be a sequence extending $\vec{b}$. 
\begin{enumerate}
\item[(i)] If $s\prec \sigma\subseteq t$, then 
\[
\Big( \big( \lfloor \un{s}\rfloor_{\succ t}\big)\,  N_s\Big) \, \ov{t}(\sigma) \ec
\ov{t}(\sigma)\, \Big( \big( \lfloor \un{s}\rfloor_{\succ t}\big)\, N_s\Big).
\]

\item[(ii)] if $s\not\subseteq t \hbox{ and }s\preceq t$, then 
\[
\lfloor \un{s}\rfloor_{\succ t}\, N_{s}\hat{t} \ec \hat{t}\, \big( \big( \lfloor \un{s}\rfloor_{\succ t}\big)\, s\big).
\]

\item[(iii)] if $s\subseteq t$ and $s$ and $\hat{t}$ are not vertices of $\vec{c}$, then 
the map 
\[
\lfloor \un{s}\rfloor_{\succ t}\, N_{s}\hat{t}  \ac \ov{t}(s)\, \hat{t}\, \big(\big( \lfloor \un{s}\rfloor_{\succ t}\big)\, s\big) 
\]
given by 
\begin{equation}\label{E:assi2}
\un{s}(\hat{t})\to \ov{t}(s)
\end{equation} 
is a combinatorial isomorphism. 
\end{enumerate} 
\end{claim}

\noindent {\em Proof of Claim~\ref{Cl:stre}.} We get (i) from Claim~\ref{Cl:tasi}~(a), (b), and (c).

Using Claim~\ref{Cl:tasi}~(d) and (e), we obtain 
\begin{equation}\notag
\Big( \big( \lfloor \un{s}\rfloor_{\succ t}\big)  s\Big) \, \hat{t} \ec 
\hat{t}\, \Big( \big( \lfloor \un{s}\rfloor_{\succ t}\big)\, s\Big), \; \hbox{ if }s\not\subseteq t \hbox{ and }s\preceq t.
\end{equation} 
which is (ii). 

To see (iii), note that, since 
\[
s\setminus \hat{t} = \un{s}\;\hbox{ and }\; \hat{t}\setminus s = \ov{t}, 
\]
the map 
\begin{equation}\notag
 \un{s}(\,\hat{t}\,)\, s\, \hat{t} \ac \ov{t}(s)\, \hat{t}\, s 
\end{equation} 
given by the assignment \eqref{E:assi2} is a combinatorial isomorphism of type 2 (with $s=s$, $t=\hat{t}$). Thus, by Lemma~\ref{L:mapdif}, since no entry of $\big( \lfloor \un{s}\rfloor_{\succ t}\big)$ has $\un{s}(\hat{t})$ as an element, we see that the map 
\begin{equation}\notag
\Big(\big( \lfloor \un{s}\rfloor_{\succ t}\big)\,  \un{s}(\,\hat{t}\,)\, s\Big)\, \hat{t} 
\ac \big( \lfloor \un{s}\rfloor_{\succ t} \big)\, \ov{t}(s)\, \hat{t}\, s 
\end{equation} 
given again by \eqref{E:assi2} is a combinatorial isomorphism. Then we use Claim~\ref{Cl:tasi}~(a) and (d) to get 
\[
\big( \lfloor \un{s}\rfloor_{\succ t} \big)\, \ov{t}(s)\, \hat{t}\, s 
\ec \ov{t}(s)\, \hat{t}\, \big(  \lfloor \un{s}\rfloor_{\succ t}\, s\big),\;\hbox{ if } s\subseteq t.
\]
Point (iii) and the claim follow. 

\smallskip

First apply Claim~\ref{Cl:stre}(i) (with $\sigma= s^t$) to the right hand side of \eqref{E:rewc} to see that 
\[
\big(\prod_{s\prec s^t} \lfloor \un{s}\rfloor_{\succ t}\, N_{s}\big) \, \ov{t}(s^t)\, \hat{t} \, \big(  \lfloor \un{s}^t\rfloor_{\succ t}  \, s^t\big) \eb
\ov{t}(s^t)\,\big(\prod_{s\prec s^t} \lfloor \un{s}\rfloor_{\succ t}\, N_{s}\big) \,  \hat{t} \, \big( \lfloor \un{s}^t\rfloor_{\succ t} \, s^t\big),
\]
which together with \eqref{E:rewc} gives a combinatorial isomorphism 
\begin{equation}\label{E:rewc2}
L\ab  \ov{t}(s^t)\, \big(\prod_{s\prec s^t} \lfloor \un{s}\rfloor_{\succ t}\, N_{s}\big) \,  \hat{t} \, \big(  \lfloor \un{s}^t\rfloor_{\succ t} \, s^t\big)
\end{equation}
given by assignment \eqref{E:assi1}. 
We claim that, to get, from \eqref{E:rewc2}, the desired combinatorial isomorphism 
\begin{equation}\label{E:almth} 
L\ab \big( \prod^*_{s:s\subseteq t} \overline{t}(s)\big)\, \hat{t}\,\big( \prod_{s:s\preceq t}  \lfloor \un{s}\rfloor_{\succ t} \, s\big)
\end{equation} 
given by assignment \eqref{E:finassi}, it will suffice to produce a combinatorial isomorphism 
\begin{equation}\label{E:prze} 
\begin{split}
\ov{t}(s^t)\, \big(\prod_{s\prec s^t} &\lfloor \un{s}\rfloor_{\succ t}\, N_{s}\big) \,  \hat{t} \, \big(  \lfloor \un{s}^t\rfloor_{\succ t} \, s^t\big)\\
&\ab\, \ov{t}(s^t)\,  \big( \prod^*_{s:s\subseteq t, s\prec s^t} \overline{t}(s)\big)\, \hat{t}\,\big( \prod_{s:s\prec s^t} \lfloor \un{s}\rfloor_{\succ t} \, s\big)\, \big(  \lfloor \un{s}^t\rfloor_{\succ t} \, s^t\big).
\end{split}
\end{equation} 
given by the assignment \begin{equation}\label{E:haha}
\un{s}(\hat{t}\,)\to \ov{t}(s), \;\hbox{ for }s\prec s^t,\, s\subseteq t. 
\end{equation} 
Indeed, note that by \eqref{E:ordad}, $s\preceq t$ is equivalent to $s\preceq s^t$, so 
the composition of \eqref{E:rewc2} and \eqref{E:prze} gives a combinatorial isomorphism as in \eqref{E:almth} that is given by 
the composition of the assignments \eqref{E:assi1} and \eqref{E:haha}, which composition is equal to the assignment \eqref{E:finassi}, as required.

From this point on, we work on constructing a combinatorial isomorphism \eqref{E:prze} given by \eqref{E:haha}. We will produce a combinatorial isomorphism 
\begin{equation}\label{E:prze2} 
\big(\prod_{s\prec s^t} \lfloor \un{s}\rfloor_{\succ t}\, N_{s}\big) \,  \hat{t} 
\ac \big( \prod^*_{s:s\subseteq t, s\prec s^t} \overline{t}(s)\big)\, \hat{t}\,\big( \prod_{s:s\prec s^t} \lfloor \un{s}\rfloor_{\succ t} \, s\big)
\end{equation} 
given by  \eqref{E:haha}, where $\vec{c}$ is a sequence extending $\vec{b}$ and such that $\hat{t}$ and $s$ with $s\prec s^t$ are not vertices of $\vec{c}$. This will be sufficient. Indeed, consider the sequence 
\begin{equation}\label{E:cdefa} 
\vec{c} =  \big(  \lfloor \un{s}^t\rfloor_{\succ t} \, s^t\big)\,\vec{b}. 
\end{equation} 
Note that $\hat{t} = (t\setminus \un{s}^t)\cup \{ s^t\}$ is a face of $s^t\vec{b}$, so it is not a vertex of $s^t\vec{b}$, and it is not an entry of $\lfloor \un{s}^t\rfloor_{\succ t}$. Thus, $\hat{t}$ is not a vertex of $\vec{c}$. Also $s$ with $s\prec s^t$ is a face of $\vec{q}$, so it is not a vertex of $\vec{q}$ and it is not an entry of the sequence 
\[
\lfloor \un{s}^t\rfloor_{\succ t}\, s^t\big( \prod_{\tau: t\prec \tau} P_\tau \tau\big).
\]
Thus, by the definition of $\vec{b}$, $s$ is not a vertex of $\vec{b}$. It follows that the formula \eqref{E:prze2} can be applied to $\vec{c}$ from \eqref{E:cdefa} above 
producing a combinatorial isomorphism 
\begin{equation}\notag 
\big(\prod_{s\prec s^t} \lfloor \un{s}\rfloor_{\succ t}\, N_{s}\big) \,  \hat{t} \, \big(  \lfloor \un{s}^t\rfloor_{\succ t} \, s^t\big)
\ab\, \big( \prod^*_{s:s\subseteq t, s\prec s^t} \overline{t}(s)\big)\, \hat{t}\,\big( \prod_{s:s\prec s^t} \lfloor \un{s}\rfloor_{\succ t} \, s\big)\, \big(  \lfloor \un{s}^t\rfloor_{\succ t} \, s^t\big).
\end{equation} 
given by  \eqref{E:haha}. This immediately yields \eqref{E:prze} by Lemma~\ref{L:mapdif}  after noticing that none of the sets in \eqref{E:haha} is an element of $\ov{t}(s^t)$.

We will iteratively apply Claim~\ref{Cl:stre} to produce the combinatorial isomorphism \eqref{E:prze2}. 
We let $s_0\cdots s_m$  be a non-decreasing enumeration of the set $\{ s\in S\mid s\prec s^t\}$. Then the left hand side of \eqref{E:prze2} becomes 
\begin{equation}\label{E:seqex} 
\big(\lfloor \un{s_0}\,\rfloor_{\succ t}\, N_{s_0}\big) \cdots \big( \lfloor \un{s_m}\,\rfloor_{\succ t}\, N_{s_m}\big)\,  \hat{t}.
\end{equation} 
Using Claim~\ref{Cl:stre} we move $\hat{t}$ to the left in the above expression as follows. To make the assumptions of Claim~\ref{Cl:stre} satisfied, we keep in mind that 
\[
s_0\prec s_1\prec \cdots \prec s_m\prec t, 
\]
where the last inequality comes from $s_m\prec s^t\subseteq t$. Let $i\leq m$. If $s_i\not\subseteq t$, then, by Claim~\ref{Cl:stre}(ii), we have 
 \begin{equation}\label{E:frt}
 \begin{split} 
\big(\lfloor \un{s_0}\,\rfloor_{\succ t}\, N_{s_0}\big) \cdots \big( \lfloor \un{s_{i-1}}\,\rfloor_{\succ t}\, N_{s_{i-1}}\big) \, \big(\lfloor \un{s_i}\,\rfloor_{\succ t}\, N_{s_i}\big) \,  &\hat{t}\\ \ec
\big( \lfloor \un{s_0}\,\rfloor_{\succ t}\, N_{s_0} \big) \cdots \big( \lfloor \un{s_{i-1}}\,\rfloor_{\succ t}\, N_{s_{i-1}}\big) \, &\hat{t}\, \big(\lfloor \un{s_i}\,\rfloor_{\succ t}\, s_i\big)
\end{split}
\end{equation}
If $s_i\subseteq t$, then, by applying first Claim~\ref{Cl:stre}(iii) and then Claim~\ref{Cl:stre}(i) (with $\sigma= s_i$), we get the following combinatorial isomorphism, for $\vec{c}$ extending $\vec{b}$ and such that $\hat{t}$ and $s_i$ are not vertices of $\vec{c}$, 
 \begin{equation}\label{E:scd} 
 \begin{split} 
\big( \lfloor \un{s_0}\,\rfloor_{\succ t}\, N_{s_0}\big) \cdots  \big(\lfloor \un{s_{i-1}}\,\rfloor_{\succ t}\, N_{s_{i-1}} \big) \, \big(\lfloor \un{s_i}\,\rfloor_{\succ t}\, N_{s_i}\big)\,  &\hat{t}\\ 
\ac\big(\lfloor \un{s_0}\,\rfloor_{\succ t}\, N_{s_0} \big) \cdots  \big(\lfloor \un{s_{i-1}}\,\rfloor_{\succ t}\, N_{s_{i-1}}\big) \,  \ov{t}(s_i)\, &\hat{t}\, \big(\lfloor \un{s_i}\,\rfloor_{\succ t}\, s_i\big)\\\ec
\ov{t}(s_i)\, \big( \lfloor \un{s_0}\,\rfloor_{\succ t}\, N_{s_0}\big)  \cdots  \big(\lfloor \un{s_{i-1}}\,\rfloor_{\succ t}\, N_{s_{i-1}} \big) \, &  \hat{t}\, \big(\lfloor \un{s_i}\,\rfloor_{\succ t}\, s_i\big), 
\end{split}
\end{equation}
where the combinatorial isomorphism in \eqref{E:scd} 
is given be the assignment \eqref{E:assi2}, that is, 
\begin{equation}\label{E:ssiia}
\un{s_i}(\hat{t}\,)\to \ov{t}(s_i).
\end{equation} 

By applying recursively \eqref{E:frt} and \eqref{E:scd} for $i= m, m-1, \dots, 0$, choosing \eqref{E:frt} or \eqref{E:scd} depending on whether $s_i\subseteq t$ or $s_i\not\subseteq t$, we obtain a combinatorial isomorphism of the form $\ac$ from the sequence \eqref{E:seqex} to the sequence 
\[
\ov{t}(s_{i_k})\, \ov{t}(s_{i_{k-1}})\cdots \ov{t}(s_{i_0})\, \hat{t}\, \big(\lfloor \un{s_0}\,\rfloor_{\succ t}\, s_0\big)\cdots  \big(\lfloor \un{s_{m-1}}\,\rfloor_{\succ t}\, s_{m-1}\big)
\big(\lfloor \un{s_m}\,\rfloor_{\succ t}\, s_m\big), 
\]
where $s_{i_0}, \dots, s_{i_k}$, for $i_0<\cdots <i_k$, enumerate all $s_i$, $i\leq m$, with $s_i\subseteq t$. The isomorphism is given by assignments \eqref{E:ssiia} for all $i= i_0, \dots, i_k$. Thus, we proved that \eqref{E:prze2} given by \eqref{E:haha} is a combinatorial isomorphism, and the lemma follows. 
\end{proof}

We are now ready to give the final argument.

\setcounter{claim}{0}

\begin{proof}[Proof of Main Lemma] 
We will prove that the map in \eqref{E:spa} is a composition of combinatorial isomorphisms and weld maps, which suffices to show Main Lemma. 

Recall that $S$, $T$, and a sequence $\vec{q}$ are fixed in the statement of Main Lemma. Recall also that 
\[
P_t= \prod_{s: t_s=t} \lfloor \un{s}\rfloor\, s.
\]
Define
\[
Q_t=  \prod^*_{s: s\subseteq t} \ov{t}(s). 
\]
As before, for two sequences $\vec{r_1}$ and $\vec{r_2}$, we will write 
\[
\vec{r_1}\eq \vec{r_2}\;\hbox{ and }\; \vec{r_1}\aq \vec{r_2}
\]
to indicate $\vec{r_1}\,\vec{q}= \vec{r_2}\,\vec{q}$ and a map $\vec{r_1}\,\vec{q}\to \vec{r_2}\,\vec{q}$, respectively.

\begin{claim*}\label{Cl:last}
The map 
\begin{equation}\notag
\prod_{t} P_t\, t \aq \big( \prod_{t} Q_t\, \hat{t}\,\big)\, S
\end{equation} 
given by the assignment 
\[
t\to \hat{t},\;\; \un{s}(t)\to \ov{t}(s),\hbox{ for } s \hbox{ with }s\subseteq t 
\]
is a combinatorial isomorphism.
\end{claim*} 

\noindent{\em Proof of Claim.}
For notational convenience, we add a new element $0$ to $S\cup T$ and require that $0$ is $\prec$-smaller than each element 
of $S\cup T$.  We claim that for all $t'\in \{ 0\}\cup T$, we have a combinatorial isomorphism 
\begin{equation}\label{E:prov}
\prod_{t} P_t\, t \aq
\big( \prod_{t: t\preceq t'} Q_t\, \hat{t}\,\big) 
\big( \prod_{s:s\preceq t'} \lfloor \un{s}\rfloor_{\succ t'}\, s\big) \big(\prod_{t: t'\prec t} P_t\, t\big)
\end{equation} 
via the assignment 
\[
\begin{split}
&t\to \hat{t},\; \hbox{ for }t\preceq t',\\
&\un{s}(t)\to \ov{t}(s),\; \hbox{ for }s\subseteq t\preceq t'.
\end{split}
\]

Taking $t'$ to be the largest element of $T$, formula \eqref{E:prov} becomes the formula in the claim. Indeed, for this $t'$, 
the sequence $\lfloor \un{s}\rfloor_{\succ t'}$ is empty as is the sequence $\prod_{t: t'\prec t} P_t\, t$. 
Further, obviously $t\preceq t'$ for all $t\in T$ but also $s\preceq t'$ for each $s\in S$ by (V). Thus, the right hand side of \eqref{E:prov} 
is equal to 
\[
\big( \prod_{t} Q_t\, \hat{t}\,\big) \big( \prod_{s} s\big) = \big( \prod_{t} Q_t\, \hat{t}\,\big)\, S. 
\]

Now, we prove \eqref{E:prov} by induction on $t'\in \{ 0\} \cup T$ starting with $t'=0$ and going upwards with respect to the order 
$\preceq$. For $t'=0$, formula \eqref{E:prov} is a tautology since the sequences in the first two pairs of parentheses on the right hand side are empty. Assume it holds for some $t'\in \{ 0\}\cup T$ that is not the largest element of $T$. 
Let $t''$ be the immediate successor of $t'$ in $\{ 0\}\cup T$. We prove \eqref{E:prov} for $t''$.  We first observe that 
\begin{equation}\label{E:trans} 
\big( \prod_{s:s\preceq t'} \lfloor \un{s}\rfloor_{\succ t'}\, s\big) \big(\prod_{t: t'\prec t} P_t\, t\big) 
= \big( \prod_{s:s\preceq t'} \lfloor \un{s}\rfloor_{\succ t'}\, s\big) \big( P_{t''} t''\big)\big(\prod_{t: t''\prec t} P_t\, t\big).
\end{equation} 
We continue with the first two sequences on the right hand side of \eqref{E:trans}. We observe that \eqref{E:flfl} and the definition of $t_s$ give 
\begin{equation}\label{E:lasl} 
 \lfloor \un{s}\rfloor= \lfloor \un{s}\rfloor_{t''}\hbox{ and } s\preceq t'',\;\hbox{ if }t_s=t''.
\end{equation}
and \eqref{E:ordad} gives 
\begin{equation}\label{E:lask}
t_s=t'', \hbox{ if }t'\prec s\preceq t''.
\end{equation}
We use \eqref{E:lasl} and \eqref{E:lask} to get the second equality in \eqref{E:trans2} below
\begin{equation}\label{E:trans2} 
\begin{split}
\big( \prod_{s:s\preceq t'} \lfloor \un{s}\rfloor_{\succ t'}\, s\big) \big( P_{t''} t''\big)
&= \big( \prod_{s:s\prec t''} \lfloor \un{s}\rfloor_{t''}\, s\big) \Big( \big( \prod_{s: t_s=t''} \lfloor \un{s}\rfloor\, s\big)\, t''\Big)\\
&=  \big( \prod_{s:s\preceq t''} \lfloor \un{s}\rfloor_{t''}\, s\big) \, t''.
\end{split}
\end{equation} 
Putting together \eqref{E:trans} and \eqref{E:trans2} gives 
\begin{equation}\label{E:afbr}
\big( \prod_{s:s\preceq t'} \lfloor \un{s}\rfloor_{\succ t'}\, s\big) \big(\prod_{t: t'\prec t} P_t\, t\big) 
=  \big( \prod_{s:s\preceq t''} \lfloor \un{s}\rfloor_{t''}\, s\big) \, t''\,
\big(\prod_{t: t''\prec t} P_t\, t\big).
\end{equation}

On the other hand, Lemma~\ref{L:cruc} (with $t=t''$) gives a combinatorial isomorphism 
\[
\big( \prod_{s:s\preceq t''} \lfloor \un{s}\rfloor_{ t''}\, s\big)\, t''\, \big(\prod_{t: t''\prec t} P_t\, t\big) \aq Q_{t''} \, \widehat{t''}\,
\big( \prod_{s: s\preceq t''}  \lfloor \un{s}\rfloor_{\succ t''} \, s\big)\big(\prod_{t: t''\prec t} P_t\, t\big), 
\]
which, after taking into account \eqref{E:afbr}, yields a combinatorial isomorphism 
\[
\big( \prod_{s:s\preceq t'} \lfloor \un{s}\rfloor_{\succ t'}\, s\big) \big(\prod_{t: t'\prec t} P_t\, t\big) \aq Q_{t''} \, \widehat{t''}\,
\big( \prod_{s: s\preceq t''} \lfloor \un{s}\rfloor_{\succ t''} \, s\big)\big(\prod_{t: t''\prec t} P_t\, t\big)
\]
via the assignment 
\begin{equation}\label{E:aslasl} 
t''\to \widehat{t''}\;\hbox{ and }\;\un{s}(t'')\to \ov{t''}(s),\hbox{ for $s$ with $s\subseteq t''$}. 
\end{equation} 
The above combinatorial isomorphism gives the combinatorial isomorphism 
\[
\begin{split}
\big( \prod_{t: t\preceq t'} Q_t\, \hat{t}\,\big) &\big( \prod_{s:s\preceq t'} \lfloor \un{s}\rfloor_{\succ t'}\, s\big) \big(\prod_{t: t'\prec t} P_t\, t\big)\\ 
&\aq \big( \prod_{t: t\preceq t''} Q_t\, \hat{t}\,\big)\, \big( \prod_{s: s\preceq t''} \lfloor \un{s}\rfloor_{\succ t''} \, s\big)\big(\prod_{t: t''\prec t} P_t\, t\big)
\end{split} 
\]
implemented by the assignment \eqref{E:aslasl}. Precomposing this map with the combinatorial isomorphism given by \eqref{E:prov} for $t'$, which exists by our inductive assumption, 
we get \eqref{E:prov} for $t''$ via the assignment 
\[
\begin{split}
&t\to \hat{t},\; \hbox{ for }t\preceq t'',\\
&\un{s}(t)\to \ov{t}(s),\; \hbox{ for }s\subseteq t\preceq t'', 
\end{split}
\]
and the claim is proved. 

\smallskip

In light of Claim, our goal of checking that the map in \eqref{E:spa} is a pure weld-division map will be achieved if we show 
that the map 
\begin{equation}\label{E:mapq}
\big( \prod_{t} Q_t\, \hat{t}\,\big)\, S \aq S
\end{equation} 
given by the assignment 
\begin{equation}\label{E:lastassig}
\hat{t}\to p\;\hbox{ and }\;\ov{t}(s) \to s,\; \hbox{ for }s\subseteq t, 
\end{equation} 
is a composition of weld maps. For this conclusion, note that writing $Q_t$ explicitly, we get 
\begin{equation}\label{E:qata}
\big( \prod_{t} Q_t\, \hat{t}\,\big)\, S  \eq \Big(\prod_{t}  \big(\prod^*_{s: s\subseteq t} \ov{t}(s)\big)\,\hat{t}\Big) S.
\end{equation} 

We work with the right hand side of equality \eqref{E:qata} above. Let $t'\in T$.  We show that the map 
\begin{equation}\label{E:plco}
f_{t'}\colon \Big(\prod_{t'\preceq t}  \Big( \big(\prod^*_{s: s\subseteq t} \ov{t}(s)\big)\,\hat{t}\Big)\Big) S\aq
\Big(\prod_{t'\prec t}  \Big(\big(\prod^*_{s: s\subseteq t} \ov{t}(s)\big)\,\hat{t}\Big)\Big) S
\end{equation} 
given by the assignment 
\begin{equation}\label{E:plcoa}
 \hat{t'}\to p,\;\; \ov{t'}(s)\to s \hbox{ for }s\subseteq t'
\end{equation} 
is a composition of weld maps. Represent the domain sequence in \eqref{E:plco} as 
\[
\Big(\big(\prod^*_{s: s\subseteq t'} \ov{t'}(s)\big)\,\widehat{t'}\Big)
\Big(\prod_{t'\prec t}  \Big(\big(\prod^*_{s: s\subseteq t} \ov{t}(s)\big)\,\hat{t}\Big)\Big)S
\]
Now, notice that $s\in \ov{t'}(s)$ and using maps of the form $\pi_{s, \ov{t'}(s)}$, map to $s$ each element of the form $\ov{t'}(s)$ in the sequence 
\[
\prod^*_{s: s\subseteq t'} \ov{t'}(s)
\]
starting with the $\preceq$-largest $s$ with $s\subseteq t'$ and ending with the $\preceq$-smallest $s$ of this sort. This produces a composition of weld maps  
\[
\Big(\big(\prod^*_{s: s\subseteq t'} \ov{t'}(s)\big)\,\widehat{t'}\Big)
\Big(\prod_{t'\prec t}  \Big(\big(\prod^*_{s: s\subseteq t} \ov{t}(s)\big)\,\hat{t}\Big)\Big)S\aq 
\widehat{t'}
\Big(\prod_{t'\prec t}  \Big(\big(\prod^*_{s: s\subseteq t} \ov{t}(s)\big)\,\hat{t}\Big)\Big)S.
\]
Then compose this map with the weld map of the form $\pi_{p, \widehat{t'}}$ (after noticing that $p\in \widehat{t'}$)
\[
 \widehat{t'}\Big(\prod_{t'\prec t}  \Big(\big(\prod^*_{s: s\subseteq t} \ov{t}(s)\big)\,\hat{t}\Big)\Big)S \aq
\Big(\prod_{t'\prec t}  \Big(\big(\prod^*_{s: s\subseteq t} \ov{t}(s)\big)\,\hat{t}\Big)\Big)S 
\]
mapping $\widehat{t'}$ to $p$. This composition is given by the assignment \eqref{E:plcoa} and, therefore, is equal to the map $f_{t'}$ as in \eqref{E:plco}.

Let now $t_0\prec \cdots \prec t_n$ enumerate $T$ in a non-decreasing manner. In light of the argument above, the map 
\[
f_{t_n}\circ \cdots \circ f_{t_0}
\]
is a composition of weld maps and is given by the assignment \eqref{E:lastassig}, so it is equal to the map \eqref{E:mapq}. So, this latter map 
is a composition of weld maps, and Main Lemma is proved. 
\end{proof}

\newpage 

\part{Proofs of the main theorems from amalgamation for $\mathcal D$}\label{P:prfd}

In this part, we work in the setup of Section~\ref{S:divset}. As explained in that section, this assumption causes no loss of generality. 
However, we will be assuming that all divided complexes are finite, in fact, we assume that the set $\rm Ur$ is finite. This assumption can by bypassed with some additional work in Section~\ref{S:proamp}, but it becomes important in the following sections.

\section{The projective amalgamation theorem}\label{S:proamp}

The goal of this section is to prove the projective amalgamation theorem, Theorem~\ref{T:frai2}. 

We recall the definitions involved in the amalgamation theorem. We fix a complex $\mathbf A$. We assume that $\mathbf A$ is finite and grounded. 
Then $\langle {\mathbf A}\rangle$ is the family of all complexes of obtained from $\mathbf A$ by iterative division. Without loss of generality, we will assume that ${\rm Vr}({\mathbf A})= {\rm Ur}$.

For $A$ in $\langle {\mathbf A}\rangle$, $s\in {\rm Fin}^+$ that is not a vertex of $A$, and $p\in s$, recall the weld map 
\begin{equation}\notag
\pi^A_{p,s} \colon  s A \to A;
\end{equation} 
see Appendix~\ref{Su:weldpro}  for the definition and basic properties. 
{\bf Weld-division maps} among complexes in $\langle {\mathbf A}\rangle$ is the smallest class of maps that 
\begin{enumerate}
\item[---] contains all weld maps $\pi^A_{p,s}$ with $p\in s\in {\rm Fin}^+$, $s\not\in {\rm Vr}(A)$, and $A\in \langle{\mathbf A}\rangle$,  

\item[---] contains all grounded isomorphisms; 

\item[---] is closed under division of simplicial maps, and  

\item[---] is closed under composition.
\end{enumerate} 
We use ${\mathcal D}({\mathbf A})$ to denote the category whose objects are complexes in $\langle {\mathbf A}\rangle$ and whose morphisms 
are weld-division maps among them.

Let $A\in \langle {\mathbf A}\rangle$. Let $S$ be an additive family of faces of $A$ and let $\iota\colon S\to {\rm Vr}(A)$ be such that $\iota(s)\in s$, for each $s\in S$. Recall the map 
\[
\pi^A_\iota\colon SA\to A;
\]
see Appendix~\ref{Su:weldpro}  for the definition and basic properties. We say that $\iota$ is based on $S$. 
We say that $S$ is {\bf upward closed in $A$} if for each $s\in S$ and $t\in A$ with $s\subseteq t$, we have $t\in S$. Note that each upward closed family is additive. The class of {\bf neatly composed weld maps} among complexes in $\langle {\mathbf A}\rangle$ 
is the smallest class of simplicial maps among complexes in $\langle {\mathbf A}\rangle$ that 
\begin{enumerate} 
\item[---] contains all maps of the form $\pi_\iota$ with $\iota$ based on an upward closed family 

\item[---] is closed under taking compositions. 
\end{enumerate}
The class of neatly composed weld maps will be denoted by $\mathcal{N}({\mathbf A})$. 

We observe that 
\[
{\mathcal N}({\mathbf A}) \subseteq {\mathcal D}({\mathbf A}).
\]

Now, we have the following amalgamation theorem.

\begin{theorem}\label{T:frai2} 
For $f',\, g'\in {\mathcal D}({\mathbf A})$ with the same codomain, there exist $f \in \mathcal{N}({\mathbf A})$ and $g\in {\mathcal D}({\mathbf A})$ 
such that 
\begin{equation}\notag 
f'\circ f = g'\circ g.
\end{equation} 
\end{theorem}

The proof amounts to translating Theorem~\ref{T:frai} to the above statement. To implement the translation, we need some additional notions and some preliminary results. 
For a grounded complex $A$ with ${\rm Vr}(A)= {\rm Ur}$, let 
\[
A^c =\{ s\in {\rm Fin}^+\mid s\subseteq {\rm Ur},\, s\not\in A\},
\]
and let 
\[
\vec{A^c}
\]
be a non-decreasing enumeration of $A^c$. The family $A^c$ is an additive family of faces of the empty sequence, so it is not material which non-decreasing enumeration of $A^c$ we choose. The introduction of the sequence $\vec{A^c}$ is needed for  the definition of the extension operation $f\to f^A$ below. Observe that 
\begin{equation}\label{E:acaa} 
\vec{A^c}A= A. 
\end{equation}
For a sequence $\vec{t}= t_1\cdots t_n$, we write 
\[
{\rm sp}(\vec{t}\,)= \{ {\rm sp}(t_1), \dots , {\rm sp}(t_n)\}. 
\]

\begin{lemma}\label{L:facar} 
Let $A$ be a grounded complex. Let $\vec{t}$ be a sequence. 
\begin{enumerate}  
\item[(i)] Then $x$ is a face of $\vec{t}\,A$ if and only if $x$ is a face of $\vec{t}\, {\vec A^c}$ and ${\rm sp}(x)\in A$. In particular,  
\[
{\rm Vr}\big( \vec{t}\,A\big) \subseteq {\rm vr}\big( \vec{t}\,\vec{A^c}\big). 
\]

\item[(ii)] Assume that ${\rm sp}(\vec{t}\,)\subseteq A$. Then $x$ is a face of $\vec{t}\,\vec{A^c}$ if and only if 
\[
x= u\cup \{ v_1, \cdots v_l\},
\]
where $v_1, \cdots, v_l\in A^c$ are such that $ v_1\subseteq \cdots \subseteq v_l$ and $u=\emptyset$ or 
$u$ is a face of $\vec{t}A$ and ${\rm sp}(u) \subseteq v_1$.
In particular, we have 
\[
{\rm vr}\big( \vec{t}\,\vec{A^c}\big) = {\rm Vr}\big( \vec{t}\,A\big)\cup A^c. 
\]
\end{enumerate} 
\end{lemma}

\begin{proof} (i) By Proposition~\ref{P:conr}, it suffices to show that $x$ is a face of $\vec{t}$ and ${\rm sp}(x)\in A$ if and only if $x$ is a face of $\vec{t}\, {\vec A^c}$ and ${\rm sp}(x)\in A$. The proof of this equivalence is done by an easy induction of the length of $\vec{t}$ and is left to the reader. 

(ii) Again, this proof is done by induction on the length of $\vec{t}$. In what follows, we use the definition of faces of sequence and Lemma~\ref{L:conr}.

Assume $\vec{t}$ is the empty sequence. Obviously, each face of $\vec{A^c}$ is of the form $u\cup X$, where $u\subseteq {\rm Ur}$ and $X\subseteq A^c$. Lemma~\ref{L:linear2} immediately implies the equivalence in (ii) for the empty sequence $\vec{t}$. 

We now show the inductive steps in the proofs of the two implications of the equivalence in (ii). 

We start with implication $\Rightarrow$. 
Assume the implication holds for $\vec{t}$ with ${\rm sp}(\vec{t}\,)\subseteq A$. We prove it for $t\,\vec{t}$ for any $t$ with ${\rm sp}(t)\in A$. Let $x$ be a face of $t\,\vec{t}\vec{A^c}$. 
We suppose that $t$ is a face of $\vec{t}\vec{A^c}$ since otherwise $t\,\vec{t}\vec{A^c}= \vec{t}\vec{A^c}$ and there is nothing to prove. 

Suppose first that $x$ is a face of $\vec{t}\vec{A^c}$. Then, by our inductive assumption, $x= u\cup \{ v_1, \cdots v_l\}$ with $u$ and $v_1, \dots, v_l$ as in (ii). Since $x$ is a face of $t\,\vec{t}\vec{A^c}$, we see that $t\not \subseteq x$, so $t\not\subseteq u$. It follows that $u=\emptyset$ or $u$ is a face of $t\,\vec{t}A$. Thus, we get the condition in (ii) for $t\,\vec{t}$. 

Suppose now, that $x$ is not a face of $\vec{t}\vec{A^c}$. Then $t\in x$ and $\big( x\setminus \{ t\}\big)\cup t$ is a face of $\vec{t}\vec{A^c}$. By our inductive assumption we get 
\[
\big( x\setminus \{ t\}\big)\cup t =u\cup \{ v_1, \dots, v_l\}
\]
with $u$ and $v_1, \dots, v_l$ satisfying the conditions in (ii). The condition ${\rm sp}(t)\in A$ implies that 
\[
t\cap \{ v_1, \dots, v_l\}=\emptyset,
\]
so, 
\begin{equation}\label{E:ulala} 
t\subseteq u\;\hbox{ and }\; x\setminus \{ t\} = (u\setminus t)\cup \{ v_1, \dots, v_l\}.
\end{equation}
Note that, in particular, $u\not=\emptyset$, so $u$ is a face of $\vec{t}A$. Using the second part of \eqref{E:ulala}, we see that 
\[
x = \big( (u\setminus t)\cup \{ t\}\big) \cup \{ v_1, \dots, v_l\}.
\]
We claim that this representation of $x$ certifies that $x$ has the properties in (ii), that is, we need to check that $(u\setminus t)\cup \{ t\}$ is a face of $t\,\vec{t} A$ and that 
${\rm sp}\big( (u\setminus t)\cup \{ t\}\big) \subseteq v_1$.
By the first part of \eqref{E:ulala}, $(u\setminus t)\cup t =u$ and, by the choice of $u$, $u$ is a face of $\vec{t}A$. Additionally, 
by Proposition~\ref{P:conr}, $t$ is a face of $\vec{t}A$ since $t$ is a face of $\vec{t}\vec{A^c}$, $\vec{t}\vec{A^c}A= \vec{t}A$, by \eqref{E:acaa}, and ${\rm sp}(t)\in A$. 
Thus, $(u\setminus t)\cup \{ t\}$ is a face of $t\,\vec{t} A$. Further, by the first part of \eqref{E:ulala}, the choice of $u$, and the properties of $\rm sp$ spelled out in \eqref{E:sppr}, we have 
\[
{\rm sp}\big( (u\setminus t)\cup \{ t\}\big) = {\rm sp}(u)\subseteq v_1.
\]
The implication $\Rightarrow$ is proved for $t\,\vec{t}$. 

The proof of the inductive step of the implication $\Leftarrow$ is similar. Again we assume it holds for $\vec{t}$ and we prove it for $t\,\vec{t}$ assuming that $t$ is a face of $\vec{t}A$. 
So suppose that 
\[
x= u\cup \{ v_1, \dots, v_l\}.
\]
and $u$ and $v_1, \dots, v_l$ have the properties from (ii). We can assume that $u$ is a face of $t\,\vec{t}A$, as the case $u=\emptyset$ is obvious. We now split the argument into two cases: $u$ is a face of $\vec{t}A$ and $u$ is not a face of $\vec{t}A$. In the first case, we note that $t\not\subseteq u$ and we apply our inductive assumption, from which we get that $x$ is a face of $\vec{t}\vec{A^c}$. Note that $t\not\subseteq x$ as $t\not\subseteq u$ and $t\cap \{ v_1, \dots, v_l\}=\emptyset$ by our assumption that ${\rm sp}(t)\in A$. It follows that 
$x$ is a face of $t\,\vec{t}\vec{A^c}$, as required. In the second case, we have $t\in u$ and $\big(u\setminus \{ t\}\big)\cup t$ is a face of $\vec{t}A$. By $t\in u$, we get 
\[
{\rm sp}\big( (u\setminus\{ t\})\cup t\big) = {\rm sp}(u)\subseteq v_1,
\]
and, by our inductive assumption, we see that 
\begin{equation}\label{E:ususa} 
\big( (u\setminus\{ t\})\cup t\big) \cup \{ v_1, \dots, v_l\}\hbox{ is a face of $\vec{t}\vec{A^c}$}. 
\end{equation}
Now, note that, by Proposition~\ref{P:conr}, $t$ is a face of $\vec{t}\vec{A^c}$ since $t$ is a face of $\vec{t}\vec{A^c} A$ and, by \eqref{E:acaa}, $\vec{t}\vec{A^c} A = \vec{t}A$. Thus, after noticing that $t\not= v_i$, for all $1\leq i\leq l$, as ${\rm sp}(t)\in A$, we see that \eqref{E:ususa} 
implies, by the definition of face, that $x$ is a face of $t\,\vec{t}\vec{A^c}$. 
\end{proof}

Let $A$ be a grounded complex. Lemma~\ref{L:facar} allows us to define the following two operations.

Let $\vec{t}$ and $\vec{s}$ be sequences with ${\rm sp}(\vec{t}),\, {\rm sp}(\vec{s})\subseteq A$. For a grounded simplicial map 
$f\colon \vec{t}A\to \vec{s}A$, by Lemma~\ref{L:facar}(ii), we can define 
\[
f^A\colon \vec{t}\vec{A^c}\to \vec{s}\vec{A^c}
\]
by letting 
\[
\begin{split} 
f^A\res {\rm Vr}\big( \vec{t}\,A\big) &= f\\
f^A\res A^c &= {\rm id}_{A^c}. 
\end{split} 
\]
It is easy to check, using Lemma~\ref{L:facar}(ii), that $f^A$ is a grounded simplicial map. We list the properties of this operation that are relevant in the proof of Theorem~\ref{T:frai2}. We have 
\begin{equation}\label{E:A1}
\begin{split}
&f \hbox{ is a grounded isomorphism }\Rightarrow\; f^A \hbox{ is a grounded isomorphism},\\
&\big(\pi_{p,s}^{\vec{t}A}\big)^A= \pi_{p,s}^{\vec{t}\vec{A^c}}, \hbox{ for }s\in \vec{t}A,\, p\in s,\\
&(sf)^A= sf^A,\hbox{ for }s\in {\rm Fin}^+\hbox{ with }{\rm sp}(s)\in A;
\end{split}
\end{equation} 
additionally, if $g\colon \vec{u}A\to \vec{t}A$ is grounded simplicial and ${\rm sp}(\vec{u})\subseteq A$, then 
\begin{equation}\label{E:A2} 
(f\circ g)^A= f^A\circ g^A.
\end{equation} 
We leave the easy check of these properties based on Lemma~\ref{L:facar}(ii) to the reader.

Now, let $\vec{s}$ and $\vec{t}$ be sequences. For a grounded simplicial map $f\colon \vec{t}\vec{A^c}\to \vec{s}\vec{A^c}$,  by Lemma~\ref{L:facar}(i), we can define 
\[
f_A\colon  \vec{t}A\to \vec{s}A
\]
by letting 
\[
f_A= f\res {\rm Vr}\big( \vec{t}A\big). 
\]
Again it is easy to see using Lemma~\ref{L:facar}(i) that $f_A$ is a grounded simplicial map. 
We have the following properties 
\begin{equation}\label{E:B1}
\begin{split}
&f \hbox{ is a grounded isomorphism }\Rightarrow\; f_A \hbox{ is a grounded isomorphism},\\
&\big(\pi_{p,s}^{\vec{t}\vec{A^c}}\big)_A= \pi_{p,s}^{\vec{t}A}, \hbox{ for }s\in \vec{t}\,\vec{A^c},\, p\in s,\\
&(sf)_A= sf_A, \hbox{ for }s\in {\rm Fin}^+,\\
&\big(\pi_\iota^{\vec{t}\vec{A^c}}\big)_A= \pi_{\iota'}^{\vec{t}A},
\end{split}
\end{equation} 
where in the last line $\iota\colon S\to \bigcup S$, for an additive family $S$ of faces of $\vec{t}\,\vec{A^c}$, is such that $\iota(s)\in s$
and $\iota' = \iota\res \big(S\cap \vec{t}\,A)$; 
furthermore, if $g\colon \vec{u}\vec{A^c}\to \vec{t}\vec{A^c}$ is grounded simplicial, then 
\begin{equation}\label{E:B2}
(f\circ g)_A= f_A\circ g_A.
\end{equation} 
For the second line of \eqref{E:B1}, we point out that if $s\in \vec{t}\,\vec{A^c}$, then, by Proposition~\ref{P:trse}, $s$ is not a vertex of $\vec{t}\,\vec{A^c}$, so it is not a vertex of $\vec{t}\,A$ and $\pi_{p,s}^{\vec{t}A}$ is defined. For the fourth line of \eqref{E:B1}, we note that if $S$ is an additive family of faces of $\vec{t}\,\vec{A^c}$, then $S\cap \vec{t}\,A$ is an additive family of faces of $\vec{t}\,A$; and if $S$ is upward closed, so is $S\cap \vec{t}\,A$. Again, we leave the check of \eqref{E:B1} and \eqref{E:B2} to the reader.

Finally, we note that, by the very definition of the operations $f\to f^A$ and $f\to f_A$, we have that, for a grounded simplicial map $f\colon \vec{t}A\to \vec{s}A$ with ${\rm sp}(\vec{t}),\, {\rm sp}(\vec{s})\subseteq A$,
\begin{equation}\label{E:B3}
(f^A)_A= f.
\end{equation}

\begin{proof}[Proof of Theorem~\ref{T:frai2}] Let $f'\colon \vec{s}\,{\mathbf A}\to \vec{r}\,{\mathbf A}$ and $g'\colon \vec{t}\,{\mathbf A}\to \vec{r}\,{\mathbf A}$ be maps in ${\mathcal D}({\mathbf A})$. We can assume that 
\begin{equation}\label{E:supinc} 
{\rm sp}(\vec{r}\,),\, {\rm sp}(\vec{s}\,),\, {\rm sp}(\vec{t}\,)\subseteq {\mathbf A}.
\end{equation} 
Indeed, by Propostion~\ref{P:conr}, each face of a complex obtained by iterative division of $\mathbf A$ has its support in $\mathbf A$, therefore, if $\vec{u}$ is a sequence and $\vec{u}'$ is obtained from $\vec{u}$ by deleting all the entries of $\vec{u}$ with support not in $\mathbf A$, then $\vec{u}\,{\mathbf A} = \vec{u}'{\mathbf A}$.

Consider $(f')^{\mathbf A}\colon \vec{s}\,\vec{{\mathbf A}^c}\to \vec{r}\,\vec{{\mathbf A}^c}$ and $(g')^{\mathbf A} \colon \vec{t}\,\vec{{\mathbf A}^c}\to \vec{r}\,\vec{{\mathbf A}^c}$ and observe that, by \eqref{E:A1}, \eqref{E:A2}, and \eqref{E:supinc}, these two maps are in $\mathcal D$. By Theorem~\ref{T:frai}, there are $f\in {\mathcal D}$ and $g\in {\mathcal N}$ such that 
\begin{equation}\label{E:fafapr} 
(f')^{\mathbf A} \circ f = (g')^{\mathbf A}\circ g.
\end{equation} 
Note that since $g$ is a composition of welds, its domain, and so also the domain of $f$, is of the form $\vec{u}\,\vec{t}\,\vec{{\mathbf A}^c}$. 
Then, by \eqref{E:B1} and \eqref{E:B2}, we get that $f_{\mathbf A}\colon  \vec{u}\, \vec{t}\,{\mathbf A}\to \vec{s}\,{\mathbf A}$ and $g_A \colon \vec{u}\, \vec{t}\,{\mathbf A}\to \vec{t}\,{\mathbf A}$ 
are in ${\mathcal D}({\mathbf A})$ and ${\mathcal N}({\mathbf A})$, respectively. Furthermore, by \eqref{E:fafapr}, \eqref{E:B3}, and \eqref{E:B2}, we get 
\[
f'\circ f_{\mathbf A} = \big( (f')^{\mathbf A}\big)_{\mathbf A}  \circ f_{\mathbf A} = \big( (g')^{\mathbf A}\big)_{\mathbf A}\circ g_{\mathbf A}= g' \circ g_{\mathbf A},
\]
which shows that $f_{\mathbf A}$ and $g_{\mathbf A}$ witness the amalgamation property for $f'$ and $g'$. 
\end{proof} 

\begin{proposition}\label{P:inob} 
The complex $\mathbf A$ is an initial object in ${\mathcal D}({\mathbf A})$, that is, 
for each complex $A$ in $\langle {\mathbf A}\rangle$, there exists a composition of weld maps, so, in particular, a weld-division map, from $A$ to ${\mathbf A}$.   
\end{proposition} 

\begin{proof} Let 
\[
A = s_1\cdots s_l {\mathbf A}, 
\]
for some sets $s_i\in {\rm Fin}^+$, for $i=1, \dots, l$, of which we can assume that $s_i$ is a face of $s_{i+1}\cdots s_l{\mathbf A}$. Pick $p_i\in s_i$. For $i=0, 1, \dots, l$, set $B_i= s_{i+1} \cdots s_l {\mathbf A}$. 
Then, for $i=1, \dots, l$, 
\[
B_{i-1}= s_i B_i\;\hbox{ and }\;\pi^{B_{i}}_{p_i, s_i}\colon B_{i-1}\to B_{i}.
\]
It follows that 
\[
\pi^{B_{l}}_{p_l, s_l}\circ\cdots \circ \pi^{B_{1}}_{p_1, s_1}\colon B_0\to B_l
\]
is in ${\mathcal D}({\mathbf A})$, which proves the lemma as $B_0=A$ and $B_l= {\mathbf A}$. 
\end{proof}

\section{${\mathcal D}({\mathbf A})$ as a transitive projective Fra{\"i}ss{\'e} class}\label{S:frrel}

The reader may consult Appendix~\ref{Su:prfr} for background information on projective Fra{\"i}ss{\'e} classes and their limits. 

Given $A$ from $\langle {\mathbf A}\rangle$ we define a binary relation $R^A$ on ${\rm Vr}(A)$ by making $x,y \in {\rm Vr}(A)$ related by $R^A$ if $x$ and $y$ belong to a face of $A$, that is, 
\begin{equation}\label{E:interR} 
xR^A y \;\Leftrightarrow\; \{ x,y\}\in A.
\end{equation} 
We note that $R^A$ is reflexive and symmetric. Clearly each grounded simplicial map $f\colon B\to A$, for $A,B\in \langle {\mathbf A}\rangle$, is an epimorphism when considered as a function from the structure $({\rm Vr}(B), R^B)$ to $({\rm Vr}(A), R^A)$. We now consider the class of structures of the form $({\rm Vr}(A), R^A)$ for $A\in \langle {\mathbf A}\rangle$ taken with the class of epimorphisms that are weld-division maps. We denote this category by 
\[
{\mathcal D}_R({\mathbf A}). 
\]

\begin{theorem}\label{T:trprc2} 
${\mathcal D}_R({\mathbf A})$ is a projective Fra{\"i}ss{\'e} class.
\end{theorem} 

\begin{proof}
The projective amalgamation property for ${\mathcal D}_R({\mathbf A})$, that is, point (ii) in the definition of projective Fra{\"i}ss{\'e} class,  is immediate from Theorem~\ref{T:frai2}. The joint projective property, that is, point (i) of the definition follows from point (ii) along with Proposition~\ref{P:inob}.
\end{proof}

To set notation, let $(A_n, f_n)$ be a generic sequence for ${\mathcal D}_R({\mathbf A})$, so 
\[
f_n\colon {\rm Vr}(A_{n+1})\to {\rm Vr}(A_n).
\]
Let ${\mathbb A}$
be the inverse limit of this inverse system and, for $n\in {\mathbb N}$, let ${\rm pr}_n\colon {\mathbb A}\to A_n$ be the projection. 
We equip $\mathbb A$ with a binary relation $R^{\mathbb A}$ by letting, for $x,y\in {\mathbb A}$, 
\[
xR^{\mathbb A}y 
\]
precisely when ${\rm pr}_n(x) R^{A_n}{\rm pr}_n(y)$, for each $n$, that is, for each $n$, there is a face $s$ in $A_n$ such that ${\rm pr}_n(x), {\rm pr}_n(y)\in s$. 
Then $\mathbb A$ with $R^{\mathbb A}$ is the projective Fra{\"i}ss{\'e} limit of ${\mathcal D}_R({\mathbf A})$. 

Let $x\in {\mathbb A}$. We point out that since each $f_n$ is grounded simplicial and 
\[
f_n\big( {\rm pr}_{n+1}(x)\big) = {\rm pr}_{n}(x), 
\]
we have 
\begin{equation}\label{E:mont}
{\rm sp}\big( \{ {\rm pr}_n(x)\} \big)\subseteq  {\rm sp}\big( \{ {\rm pr}_{n+1}(x)\} \big), \hbox{ for each }n.
\end{equation}
The function $\rm sp$ is defined on faces of $A_n$, hence the need to consider the faces $\{ {\rm pr}_n(x)\}$ and $\{ {\rm pr}_{n+1}(x)\}$ in formula \eqref{E:mont} rather than the vertices  ${\rm pr}_n(x)$ and ${\rm pr}_{n+1}(x)$. 
The observation \eqref{E:mont} and finiteness of $\mathbf A$ imply that the sequence $\big({\rm sp}( \{ {\rm pr}_n(x)\} )\big)_n$ stabilizes and the set 
\[
\bigcup_n{\rm sp}\big( \{ {\rm pr}_n(x)\} \big)
\]
is a face of $\mathbf A$. It can be thought of as the support in $\mathbf A$ of $x\in {\mathbb A}$.

\begin{theorem}\label{T:trans2} Let $\mathbb A$ equipped with $R^{\mathbb A}$ be the projective Fra{\"i}ss{\'e} limit of ${\mathcal D}_R({\mathbf A})$.
\begin{enumerate}
\item[(i)] $R^{\mathbb A}$ is a transitive relation, so it is a compact equivalence relation on $\mathbb A$. 

\item[(ii)] For all $x,y\in {\mathbb A}$, if $xR^{\mathbb A}y$, then 
\[
\bigcup_n{\rm sp}\big( \{ {\rm pr}_n(x)\}\big) = \bigcup_n {\rm sp}\big(\{ {\rm pr}_n(y)\} \big).
\] 
\end{enumerate}
\end{theorem} 

\begin{proof} By its very definition $R^{\mathbb A}$ is reflexive, symmetric, and compact.

Let $x,y,z\in {\mathbb A}$. Assume that $xR^{\mathbb A} y$ and $y R^{\mathbb A} z$. To show transitivity of $R^{\mathbb A}$, we need to see that $x R^{\mathbb A} z$. Set, for $n\in {\mathbb N}$, 
\[
x_n = {\rm pr}_n(x), \, y_n = {\rm pr}_n(y), \, z_n = {\rm pr}_n(z). 
\]
Then $x_n,y_n,z_n\in {\rm Vr}(A_n)$ and, for each $n$ there is a face of $A_n$ containing $x_n$ and $y_n$ and there is a face of $A_n$ containing $y_n$ and $z_n$.

We first make an observation that will make the notation in the argument below easier to handle. If, for infinitely many $n$, we have $x_n=y_n$ or $y_n=z_n$ or $x_n=z_n$, then $x=y$ or $y=z$ or $x=z$ and there is nothing to prove. So, we can assume that, for large enough $n$, we have $x_n\not= y_n$ and $y_n\not= z_n$ and $x_n\not= y_n$. For ease of notation,  we will assume that this happens for all $n\in {\mathbb N}$.

So, we assume that 
\begin{equation}\label{E:xypo} 
\{ x_n, y_n\}\in A_n, \hbox{ for all }n.
\end{equation}
and 
\begin{equation}\label{E:xypo2} 
\{ y_n, z_n\}\in A_n, \hbox{ for all }n,
\end{equation}

We fix $k$. 
Let $s= \{ x_k, y_k\}$. By \eqref{E:xypo}, $s\in A_k$. Consider 
\[
\pi = \pi^{A_k}_{y_k, s}\colon sA_k\to A_k. 
\]
Since $\pi$ is in ${\mathcal D}({\mathbf A})$, by property (iii) of projective Fra{\"i}ss{\`e} limits (see Appendix~\ref{Su:prfr}), 
there are $l\geq k$ and a map $f\colon A_l\to sA_k$ in ${\mathcal D}({\mathbf A})$ such that 
\begin{equation}\label{E:whok} 
\pi\circ f\circ {\rm pr}_l = {\rm pr}_k.
\end{equation} 
Since the only vertex mapped by $\pi$ to $x_k$ is $x_k$, we see from \eqref{E:whok} that 
\begin{equation}\label{E:xlk} 
f(x_l)= x_k,
\end{equation}  
and, similarly, since the only vertex mapped by $\pi$ to $z_k$ is $z_k$, we have 
\begin{equation}\label{E:zlk}
f(z_l)=z_k. 
\end{equation}
The only vertices mapped by $\pi$ to $y_k$ are $s$ and $y_k$, so $f(y_l)= s$ or $f(y_l)= y_k$. By \eqref{E:xlk}, the possibility $f(y_l)=y_k$ would give 
\begin{equation}\label{E:prefo}
f\big( \{ x_l, y_l\} \big) = \{ x_k, y_k\}. 
\end{equation} 
But, by \eqref{E:xypo}, 
$\{ x_l, y_l\}\in A_l$ while, by the definition of $sA_k$, $\{ x_k, y_k\}\not\in sA_k$. So \eqref{E:prefo} would contradict $f$ being a grounded simplicial map, which follows from $f$ being in ${\mathcal D}({\mathbf A})$. So we have 
\begin{equation}\label{E:yls} 
f(y_l)= s. 
\end{equation} 
We observe that, since $f$ is grounded simplicial, by \eqref{E:yls}, so we have 
\[
{\rm sp}(\{ y_l\}) \supseteq {\rm sp}(\{ s\}) = {\rm sp}(\{ x_k, y_k\}), 
\]
and, therefore, for each $k$, there exists $l\geq k$ with 
\begin{equation}\label{E:supli} 
{\rm sp}(\{ y_l\})  \supseteq {\rm sp}(\{ x_k\}). 
\end{equation}

The argument above showed that, for each $k$, there exists $l\geq k$ with \eqref{E:zlk}, \eqref{E:yls}, and \eqref{E:supli} holding. Now assume towards a contradiction that 
\begin{equation}\label{E:xyzne} 
\{ x_k, y_k, z_k\}\not\in A_k, \hbox{ for some }k.
\end{equation} 
For this $k$, fix $l\geq k$ with \eqref{E:zlk}, \eqref{E:yls}, and \eqref{E:supli}. By \eqref{E:xyzne}, $\{ s, z_k\} \not\in sA_k$ for $s= \{ x_k, y_k\}$, and therefore, by \eqref{E:zlk} and \eqref{E:yls}, we get $\{ y_l, z_l\}\not\in A_l$ contradicting \eqref{E:xypo2}.
Therefore, we have 
\begin{equation}\notag
\{ x_n, y_n, z_n\}\in A_n, \hbox{ for all }n.
\end{equation} 
Thus, (i) is proved. Now, observe that \eqref{E:supli} implies ${\rm Sp}(y)\supseteq {\rm Sp}(x)$, so, by symmetry of the relation $R^{\mathbb A}$, we get (ii).  
\end{proof}

\section{Identification of the canonical quotient} 

In light of Theorem~\ref{T:trans2}, for $a \in {\mathbb A}/R^{\mathbb A}$, we write 
\[
{\rm Sp}_{\rm fr}^{\mathbf A}(a) = \bigcup_n {\rm sp}\big( \{ {\rm pr}_n(x)\} \big) \hbox{ for any }x\in {\mathbb A}\hbox{ with } a= x/R^{\mathbb A}. 
\]
Note that ${\rm Sp}_{\rm fr}^{\mathbf A}(a)\in {\mathbf A}$. So for $s\in {\mathbf A}$, the elements $a$ of ${\mathbb A}/R^{\mathbb A}$ with ${\rm Sp}_{\rm fr}^{\mathbf A}(a)=s$ can be considered to be supported by $s$, that is, one can consider the set of all $a\in {\mathbb A}/R^{\mathbb A}$ with ${\rm Sp}_{\rm fr}^{\mathbf A}(a)=s$ to be a face of ${\mathbb A}/R^{\mathbb A}$ corresponding to $s$.

Recall the definition \eqref{E:spgr} of support  ${\rm Sp}_{\rm go}^{\mathbf A}(x)$ of an element $x$ of the geometric realization $\|{\mathbf A}\|$ of $\mathbf A$. Note that ${\rm Sp}_{\rm go}^{\mathbf A}(x)\in {\mathbf A}$. Again, given $s\in {\mathbf A}$, one can consider the set of all $x\in \|{\mathbf A}\|$ with ${\rm Sp}_{\rm go}^{\mathbf A}(x)=s$ to constitute a face of $\|{\mathbf A}\|$ corresponding to $s$. 

The theorem below asserts that there is a homeomorphism between ${\mathbb A}/ R^{\mathbb A}$ and $\|{\mathbf A}\|$ that respects the notions of faces corresponding to $s\in {\mathbf A}$ in these two spaces.

\begin{theorem}\label{T:prli2}
Let $\mathbb A$ equipped with $R^{\mathbb A}$ be the projective Fra{\"i}ss{\'e} limit of ${\mathcal D}_R({\mathbf A})$. Then 
there exists a homeomorphism 
\[
g\colon {\mathbb A}/ R^{\mathbb A}\to \|{\mathbf A}\| 
\]
such that, for each $a\in {\mathbb A}/R^{\mathbb A}$, 
\begin{equation}\label{E:aligsp} 
{\rm Sp}_{\rm fr}^{\mathbf A}(a) = {\rm Sp}_{\rm go}^{\mathbf A}\big(g(a)\big). 
\end{equation} 
\end{theorem} 

For the proof of the theorem above, we will need a generic sequence of ${\mathcal D}({\mathbf A})$ with additional properties. The following lemma gives such a sequence.

\begin{lemma}\label{L:specge} 
There exists a generic sequence $(A_n, \pi_n)$ for ${\mathcal D}({\mathbf A})$ such that 
\begin{enumerate}
\item[(i)] for each $n$, there exists $s_n\in A_n$ and $p_n\in s_n$ such that 
\[
\pi_n= \pi^{A_n}_{p_n,s_n}; 
\]

\item[(ii)] there are infinitely many $n$, for which there exist $n'\geq n$ and $\iota\colon A_n\to {\rm Vr}(A_n)$ with $\iota(s)\in s$, for each $s\in A_n$, such that 
\[
\pi_n\circ\cdots \circ \pi_{n'} = \pi_{\iota}^{A_n};
\]

\item[(iii)] $A_0={\mathbf A}$.
\end{enumerate} 
\end{lemma}

\begin{proof} We refer in this proof to the construction of a generic sequence as described at the end of Appendix~\ref{Su:prfr}. 

We note that to perform steps (A) and (B) in this construction in the case of ${\mathcal D}({\mathbf A})$, that is, to find maps $f$, $f_n$, $f'$, and $f_{n'}$, we use the amalgamation theorem, Theorem~\ref{T:frai2}. This theorem guarantees that the maps  $f_n$ and $f_{n'}$ are compositions of weld maps, which in turn guarantees (i) of the current lemma. 

When constructing the generic sequence $(A_n, \pi_n)$ recursively on $n$, it is enough to perform steps (A) and (B) only at infinitely many $n$, leaving out infinitely many $n$ to satisfy (ii). For any $n$ of this latter kind, with $A_n$ given, we let 
\[
n'= n+ \#(A_n)-1, 
\]
where $\#(A)$ is the number of faces of $A_n$, and we pick an arbitrary $\iota\colon A_n\to {\rm Vr}(A_n)$ with $\iota(s)\in s$, for each $s\in A_n$. By Lemma~\ref{L:ioco2}, $\pi^{A_n}_\iota$ is a composition of weld maps $\pi_k\colon A_{k+1}\to A_k$, $n\leq k \leq n'$ as in point (ii). We choose the pairs $(A_k, \pi_k)$, $n\leq k\leq n'$, to be in the generic sequence we are constructing. Since the maps $\pi_k$, $n\leq k\leq n'$, are weld maps, the construction as described in the preceding paragraph can be completed with this choice. 

Point (iii) can be arranged by Proposition~\ref{P:inob}.
\end{proof}

\begin{proof}[Proof of Theorem~\ref{T:prli2}] 
We will use the framework for the geometric realization set up in Appendix~\ref{S:Div}. Additionally, we introduce the following piece of notation. If $r$ is a realizing assignment for a complex $A$, for $s\in A$, we write 
\[
{\rm conv}^0\big(r(s)\big) = {\rm conv}\big(r(s)\big)\setminus  \bigcup_{\emptyset\not= s'\subsetneq s} {\rm conv}\big(r(s')\big).
\]

Given a realizing assignment $r$ for $A$ and $s\in A$, we form a realizing assignment $r'$ for $sA$ by letting it be equal to $r$ on 
${\rm Vr}(A)\setminus \{ x\}$ or ${\rm Vr}(A)$, depending on whether or not $s=\{ x\}$ for some $x\in {\rm Vr}(A)$, and setting 
\begin{equation}\label{E:aseq} 
r'(s )= \frac{1}{\#(s)} \sum_{v\in s} x_v, 
\end{equation} 
where $\#(s)$ stands for the number of elements in $s$. It is easy to check that this is a realizing assignment for $sA$, that is, that conditions 
\eqref{E:prog1} and \eqref{E:prog2} hold. Note that 
\begin{equation}\label{E:stas}
\| r\| = \| r'\|. 
\end{equation} 
We make the following easy but crucial observation. Let $p\in s\in A$ and let $\pi=\pi_{p,s}\colon sA \to A$ be a weld map. Then,
for $t\in sA$, there exists $t_0\in A$ with
\begin{align}
&{\rm sp}(t)={\rm sp}(t_0),\label{E:a1}\\ 
&\pi(t)\subseteq t_0,\label{E:a2}\\ 
&{\rm conv}\big( r'(t)\big) \subseteq {\rm conv}\big( r(t_0)\big),\label{E:a3}\\
&{\rm conv}^0\big( r'(t)\big)\subseteq {\rm conv}^0\big(r(t_0)\big).\label{E:a4} 
\end{align} 
Indeed, if $s \not\in t$, then take $t_0=t$; if $s\in t$, take $t_0= \big( t\setminus\{ s\}\big) \cup s$. 
Note that \eqref{E:a1}--\eqref{E:a4} constitute an essential use of $\pi$ being a weld map; a division of a weld map need not have these properties and neither does a grounded isomorphism. 

Let 
\[
{\mathbb A} =\varprojlim_n (A_n, \pi_n)
\]
be the projective Fra{\"i}ss{\'e} limit of ${\mathcal D}_R({\mathbf A})$ produced with the generic sequence from Lemma~\ref{L:specge}. Let ${\rm pr}_n\colon {\mathbb A}\to A_n$ be the canonical projections and, for $m\leq n$, let ${\rm pr}^n_m\colon A_n\to A_m$ be the composition $\pi_m\circ\cdots \circ \pi_{n-1}$. Furthermore, by Proposition~\ref{P:inob}, we can assume $A_0= {\mathbf A}$. 

We start with a realizing assignment for $\mathbf A$, for example, with $r_0\colon {\rm Vr}({\mathbf A}) \to {\mathbb R}^{{\rm Vr}({\mathbf A})}$ given by $r_0(v)=x^v$ as in Appendix~\ref{S:Div}. We identify the geometric realization $\| {\mathbf A}\|$ of $\mathbf A$ with $\| r_0\|$. Now we apply recursively the procedure described in \eqref{E:aseq} to produce realizing assignments $r_n$ for $A_n$. Note that by \eqref{E:stas}, we have 
\[
r_n\colon {\rm Vr}(A_n)\to |{\mathbf A}|_{\mathbb{R}}. 
\]
We observe two properties of the map $r_n$ stated in \eqref{E:epze} and \eqref{E:epzzz} below. 

First, for $n\leq n'$ as in Lemma~\ref{L:specge}(ii), $A_{n'}$ is obtained from $A_n$ by the barycentric division of $A_n$. So, by the standard estimate for barycentric subdivision \cite[Proof of Proposition~2.21]{Ha}, we have 
\[
\max_{t'\in A_{n'}} {\rm diam}\big(r_{n'}(t')\big) \leq \frac{d-1}{d}\, 
\max_{t\in A_n} {\rm diam}\big(r_{n}(t)\big), 
\]
where $\rm diam$ is computed  with respect to the Euclidean metric on $\|{\mathbf A}\|$ and  $d$ is the maximum value of $\#(s)$ with $s\in {\mathbf A}$.
Thus, by Lemma~\ref{L:specge}(ii), we get 
\begin{equation}\label{E:epze} 
\epsilon_n=\max_{t\in A_n} {\rm diam}(r_n(t)) \to 0.  
\end{equation}

The second property of $r_n$ is the following. Fix $n$. Since $\pi_n$ is a weld map, \eqref{E:a2} and \eqref{E:a3} immediately give that, for each $n$, 
\begin{equation}\label{E:epzzz}
\begin{split}
&\hbox{for each $t'\in A_{n+1}$, there exist $t\in A_n$ with}\\
&\pi_n(t')\subseteq t, \hbox{ and }\, {\rm conv}\big( r_{n+1}(t')\big) \subseteq {\rm conv}\big( r_n(t)\big), 
\end{split} 
\end{equation}
which implies by an easy inductive argument (with $t'=\{ v'\}$) that, for all $n'>n$ and $v'\in {\rm Vr}(A_{n'})$, there exists $t\in A_n$ with
\begin{equation}\label{E:b2}
{\rm pr}^{n'}_n(v')\in t\;\hbox{ and }\;r_{n'}(v') \in {\rm conv}\big( r_n(t)\big).
\end{equation}

Set
\[
g_n =r_n\circ {\rm pr}_n\colon  {\mathbb A}\to |{\mathbf A}|_{\mathbb{R}}.
\]
Each map $g_n$ is clearly continuous. 
By \eqref{E:b2}, we see that, for each $x\in {\mathbb A}$, for all $n'\geq n$, 
\begin{equation}\label{E:vvvv}
g_{n'}(x) \in {\rm conv}\big( r_n(t)\big)\hbox{ for some }t\in A_n\hbox{ with }{\rm pr}_n(x)\in t. 
\end{equation} 
Thus, since each $A_n$ is finite, for every $x\in {\mathbb A}$ and $n\leq n'$, we have that
\[
{\rm dist}\big(g_n(x), g_{n'}(x)\big)\leq \epsilon_n, 
\]
where ${\rm dist}$ is the Euclidean metric. So, by \eqref{E:epze}, the sequence of function $(g_n)$ converges uniformly. 
Let $g\colon {\mathbb A} \to |{\mathbf A}|_{\mathbb{R}}$ be the continuous function that is the uniform limit of 
the sequence $(g_n)$. 
Since the image of $g_n$ contains the image of $r_n$, we see that the image of 
$g_n$ is $\epsilon_n$-dense in $|{\mathbf A}|_{\mathbb{R}}$. Now, by \eqref{E:epze}, compactness of ${\mathbb A}$, 
and uniform convergence of $(g_n)$ to $g$, we get that  $g$ is surjective. 
Additionally, note, for a later use, that \eqref{E:vvvv} implies that, for each $n$ 
\begin{equation}\label{E:lateru}
g(x) \in {\rm conv}\big( r_n(t)\big)\hbox{ for some }t\in A_n\hbox{ with }{\rm pr}_n(x)\in t. 
\end{equation}

We check that, for $x,y\in {\mathbb A}$, 
\begin{equation}\label{E:invariants} 
g(x)=g(y)\;\hbox{ if and only if }\; xR^{\mathbb A}y. 
\end{equation}

If $xR^{\mathbb A}y$, then, for each $n$, there is $t\in A_n$ with ${\rm pr}_n(x), {\rm pr}_n(y)\in t$, and so with $g_n(x), g_n(y) \in {\rm conv}\big(r_n(t)\big)$. It follows that ${\rm dist}(g_n(x), g_n(y)) \leq \epsilon_n$, for each $n$; thus, $g(x)=g(y)$ by \eqref{E:epze}. 

Assume now that $\neg(x R^{\mathbb A} y)$. Let 
\[
p_n = {\rm pr}_n(x)\in {\rm Vr}(A_n)\;\hbox{ and }\; q_n= {\rm pr}_n(y)\in {\rm Vr}(A_n). 
\]
Since $R^{\mathbb A}$ is a transitive relation and we have $\neg(x R^{\mathbb A} y)$, finiteness of each $A_n$ guarantees the 
existence of $n$ such that, for all $s, t\in A_n$, 
\begin{equation}\label{E:std}
\hbox{ if }p_n\in s\hbox{ and }q_n\in t, \hbox{ then } s\cap t=\emptyset. 
\end{equation} 
Otherwise, we would find $s_n,t_n\in A_n$ such that, for each $n$, 
\[
p_n\in s_n,\; q_n\in t_n,\; s_n\cap t_n\not=\emptyset,\; \pi_n(s_{n+1})= s_n, \hbox{ and }\, \pi_n(t_{n+1})= t_n.
\] 
This, in turn, would allow us to find $v_v\in s_n\cap t_n$ such that $\pi_n(v_{n+1})=v_n$, for each $n$, that is, $(v_n)\in {\mathbb A}$. But then, we would have 
\[
x=(p_n) R^{\mathbb A} (v_n)R^{\mathbb A} (q_n)=y, 
\]
so $xR^{\mathbb A}y$, contradicting our case assumption. 

We fix $n$ as in \eqref{E:std}. By finiteness of $A_n$ and compactness of ${\rm conv}\big(r_n(s)\big)$ 
and ${\rm conv}\big(r_n(t)\big)$ for $s, t\in A_n$, it follows from \eqref{E:b2} that there are $s, t\in A_n$ with 
\begin{equation}\label{E:nstd}
p_n\in s,\; q_n\in t,\; g(x)\in {\rm conv}\big(r_n(s)\big), \hbox{ and }\, g(y) \in {\rm conv}\big(r_n(t)\big).
\end{equation} 
If $g(x)=g(y)$, then \eqref{E:nstd} would contradict \eqref{E:std} since 
\[
{\rm conv}\big( r_n(s)\big) \cap {\rm conv}\big( r_n(t)\big) = {\rm conv}\big( r_n(s\cap t)\big). 
\] 
Thus, $g(x)\not= g(y)$, as required. 

As a consequence of \eqref{E:invariants} and $g$ being a continuous surjection, we see that the map 
${\mathbb A}/R^{\mathbb A} \to |{\mathbf A}|_{\mathbb{R}}$ induced by $g$ is a homeomorphism. We denote this induced homeomorphism again by $g$.

It remains to check \eqref{E:aligsp}, that is, 
\[
{\rm Sp}_{\rm fr}^{\mathbf A}(a) = {\rm Sp}_{\rm go}^{\mathbf A}\big(g(a)\big).
\] 
Fix $a\in {\mathbb A}/R^{\mathbb A}$ and $z\in {\mathbb A}$ such that $a= z/R^{\mathbb A}$. 

By \eqref{E:a1} and \eqref{E:a3} and finiteness of $A_0= {\mathbf A}$, we see that there exists $t\in {\mathbf A}$ 
such that, for infinitely many $n$, 
\[
{\rm sp}(\{ {\rm pr}_n(z)\}) = {\rm sp}(t)=t
\;\hbox{ and }\;g_n(z)= {r_n}({\rm pr}_n(z))\in {\rm conv}(r_0(t)). 
\]
It now follows from \eqref{E:mont} and compactness of ${\rm conv}(r_0(t))$ that 
\[
 {\rm Sp}_{\rm fr}^{\mathbf A}(a)= t \;\hbox{ and }\; g(a) \in {\rm conv}(r_0(t)). 
\]
So, by definition of ${\rm Sp}_{\rm go}^{\mathbf A} \big(g(a)\big)$, we get 
\begin{equation}\label{E:kunda}
{\rm Sp}_{\rm go}^{\mathbf A} \big(g(a)\big)\subseteq  {\rm Sp}_{\rm fr}^{\mathbf A}(a). 
\end{equation}

We prove that the inclusion in \eqref{E:kunda} is not proper, which will establish \eqref{E:aligsp}. 
We claim that if $s\in A_n$, then 
\begin{equation}\label{E:funda}
{\rm Sp}^{\mathbf A}_{\rm go}\big(x\big)={\rm sp}(s)\; \hbox{ for all $x\in {\rm conv}^0(r_n(s))$}. 
\end{equation} 
Indeed, it is clear that \eqref{E:funda} holds for $n=0$. Assume it holds for $n$ and fix $s\in A_{n+1}$. We apply \eqref{E:a1} and \eqref{E:a4} to $\pi_n$ and $s$ to find $s_0\in A_n$ 
with 
\begin{align}
&{\rm sp}(s)={\rm sp}(s_0),\label{E:c1}\\
& {\rm conv}^0\big(r_{n+1}(s)\big)\subseteq  {\rm conv}^0\big(r_{n}(s_0)\big) .\label{E:c3}
\end{align} 
By our inductive assumption, we have 
\[
{\rm Sp}^{\mathbf A}_{\rm go}\big(x\big)={\rm sp}(s_0)\; \hbox{ for all $x\in {\rm conv}^0(r_n(s_0))$} 
\]
which implies \eqref{E:funda} for $s$ by \eqref{E:c1} and \eqref{E:c3}. So \eqref{E:funda} is proved. 

Set $t= {\rm Sp}_{\rm fr}^{\mathbf A}(a)$. Let $n_0$ be such that ${\rm sp}(\{ {\rm pr_n}(z)\})= t$ for all $n\geq n_0$. 
Now, we argue that we can find $n_1\geq n_0$ such that if $v\in {\rm Vr}(A_{n_1})$ and $vR^{A_{n_1}}{\rm pr}_{n_1}(z)$, then 
\begin{equation}\label{E:nono} 
{\rm sp}(\{v\})\supseteq t.  
\end{equation}
Indeed, if the above assertion failed, by a compactness argument using \eqref{E:mont}, we would find $y\in {\mathbb A}$ with $zR^{\mathbb A} y$ and such that 
$\bigcup_n{\rm sp}\big({\rm pr}_n(y)\big)$ does not contain $t$, in particular, 
\[
\bigcup_n{\rm sp}\big({\rm pr}_n(z)\big)\not= \bigcup_n{\rm sp}\big({\rm pr}_n(y)\big)
\]
contradicting Theorem~\ref{T:trans2}(ii). 
It follows from \eqref{E:nono} and \eqref{E:funda}, applied to $n=n_1$ and $s= \{ v\}$, that for each $v\in {\rm Vr}(A_{n_1})$ with $vR^{A_{n_1}}{\rm pr}_{n_1}(z)$, we have 
\begin{equation}\label{E:kurd}
{\rm Sp}^{\mathbf A}_{\rm go}\big(r_{n_1}(v)\big)\supseteq t.
\end{equation} 
Now, by \eqref{E:lateru}, there exists  $s_0\in A_{n_1}$ with ${\rm pr}_{n_1}(z)\in s_0$ such that 
\begin{equation}\label{E:locgg}
g(a)\in {\rm conv}(r_{n_1}(s_0)). 
\end{equation} 
By \eqref{E:kurd}, we get 
\[
{\rm Sp}^{\mathbf A}_{\rm go}\big(r_{n_1}(v)\big)\supseteq t,\;\hbox{ for each }v\in s_0.
\]
The definition of ${\rm Sp}^{\mathbf A}_{\rm go}$ and the above inclusion immediately imply that 
\[
r_{n_1}(s_0)\cap {\rm conv}(r_0(t'))=\emptyset, \;\hbox{ for each }t'\subsetneq t.
\]
So since, by \eqref{E:a3}, $r_{n_1}(s_0)$ is a subset of ${\rm conv}(r_0(t_0))$ for some $t_0\in {\mathbf A}$, we get
\[
{\rm conv}\big(r_{n_1}(s_0)\big)\cap {\rm conv}(r_0(t'))=\emptyset, \;\hbox{ for each }t'\subsetneq t.
\]
Thus, by \eqref{E:locgg}, we have 
\[
g(a)\not\in {\rm conv}\big(r_0(t')\big), \hbox{ for each }t'\subsetneq t.
\]
So, it follows that ${\rm Sp}^{\mathbf A}_{\rm go}\big(g(a)\big)$
is not a proper subset of $t$, and \eqref{E:aligsp} is proved. 
\end{proof}

\section{Questions}

We ask our questions in the framework of divided complexes from Section~\ref{S:divset}. Recall from Appendix~\ref{A:isome} the statements of the definitions of weld maps, grounded isomorphisms, and division of grounded simplicial maps  in this framework.

The first question is a rigidity problem for grounded isomorphisms asking if all grounded isomorphisms are combinatorial isomorphisms. 
\begin{question}
Let $A$, $B$ be divided complexes. Let $f\colon B\to A$ be a grounded isomorphism. Can $f$ be obtained from grounded isomorphisms described in Theorems~\ref{T:ord3} and \ref{T:ord} using composition and division of simplicial maps? 
\end{question}

The second question concerns a Ramsey statement in the weld-division category ${\mathcal D}({\mathbf A})$ for a finite grounded complex $\mathbf A$. Recall the class $\langle {\mathbf A}\rangle$ of all complexes obtained from $\mathbf A$ by iterating the division operation.

\begin{question} 
Let $b>0$ be a natural number. Let $A, B\in \langle {\mathbf A}\rangle$. Does there exist $C\in \langle {\mathbf A}\rangle$ such that for each coloring with $b$ colors of all weld-division maps from $C$ to $A$, there exists a weld-division map $g\colon C\to B$ such that the set 
\[
\{ f\circ g\mid f\colon B\to A\hbox{ a weld division map}\}
\]
is monochromatic? 
\end{question}

The definition of the class of weld-division maps describes this class from below, that is, the class is defined by closing a base family (weld maps and grounded isomorphisms) under certain operations (composition and division). 

\begin{question}
Is there a description of weld-division maps that characterizes them through combinatorial properties? 
\end{question}

\appendix

\newpage

\part{Appendices}

\section{Simplicial complexes, projective Fra{\"i}ss{\'e} classes, Set Theory} 

Two notions from the literature form the basis of the considerations in this paper: the notion of stellar subdivision of a simplicial complex 
and the notion of 
projective Fra{\"i}ss{\'e} class and its limit. In Appendix~\ref{S:Div}, we recall the standard background relating to simplicial complexes 
and their stellar subdivisions. We also fix notation that will be used throughout the paper; the notation is not entirely standard. 
In Appendix~\ref{Su:prfr}, we recall the projective Fra{\"i}ss{\'e} theory.

\subsection{Simplicial complexes}\label{S:Div}

A {\bf simplicial complex}, or {\bf complex} for short, is a family $A$ of non-empty finite sets closed under taking non-empty subsets, that is, if 
$\emptyset\not= s\subseteq t\in A$, then $s\in A$, and such that 
\begin{equation}\label{E:vrad} 
{\rm Vr}(A)\cap A=\emptyset, 
\end{equation} 
where ${\rm Vr}(A)$ is  the union of all sets in $A$. Sets in $A$ are called {\bf faces} of $A$ and elements of ${\rm Vr}(A)$ are called {\bf vertices} of $A$. Note that 
condition \eqref{E:vrad} is equivalent to 
\begin{equation}\label{E:vrad2} 
s\not\in t,\hbox{ for all }s,t\in A.
\end{equation} 

We describe a setup which will be used to define the geometric realization of a simplicial complex. The setup will be useful in the proofs. 
For a subset $X$ of some ${\mathbb R}^m$, we denote by 
\[
{\rm conv}(X) 
\]
the convex hull of $X$. 
We now fix $m$, and with each $v\in {\rm Vr}(A)$, we associate a point $r(v)\in {\mathbb R}^m$ so that the following 
conditions hold for all $s,t\in A$: 
\begin{align}
&\hbox{the points }r(v),\, v\in s, \hbox{ are in general position,}\label{E:prog1}\\ 
&{\rm conv}\big(\{ r(v)\mid v\in s\}\big)\cap {\rm conv}\big(\{ r(v)\mid v\in t\}\big) = {\rm conv}\big(\{ r(v)\mid v\in s\cap t\}\big).\label{E:prog2} 
\end{align} 
We then consider the union 
\begin{equation}\label{E:unr}
\bigcup_{s\in A} {\rm conv}\big(\{ r(v)\mid v\in s\}\big)\subseteq {\mathbb R}^m.
\end{equation} 
Note that the topological space defined as the above union is fully determined by the assignment 
\[
r\colon {\rm Vr}(A)\to {\mathbb R}^m
\]
fulfilling \eqref{E:prog1} and \eqref{E:prog2}.  We call such an assignment a {\bf realizing assignment for} $A$ and write 
\[
\| r\|
\]
for the topological space \eqref{E:unr}. 
Given two realizing assignments $r$ and $r'$ of the same complex $A$, it is easy to verify that the spaces $\| r\|$ and $\| r'\|$ are homeomorphic by considering the map  
\begin{equation}\label{E:relaf} 
r(v) \to r'(v), \hbox{ for }v\in {\rm Vr}(A), 
\end{equation} 
and extending this map in an affine manner to each set ${\rm conv}\big(\{ r(v)\mid v\in s\}\big)$ with $s\in A$. 
An example of a realizing assignment is the map 
\[
{\rm Vr}(A)\ni v\to x^v\in {\mathbb R}^{{\rm Vr}(A)}, 
\]
where, for $w\in {\rm Vr}(A)$, 
\[
x^v(w) = 
\begin{cases}
1, &\hbox{ if }w=v,\\
0, &\hbox{ if } w\not= v.
\end{cases}
\]
Given a realizing assignment $r$ of $A$, we view ${\rm Vr}(A)$ as embedded into $\| r\|$ by 
\begin{equation}\label{E:reaem} 
{\rm Vr}(A)\ni v\to r(v)\in \| r\|. 
\end{equation} 

A {\bf geometric realization} of a complex $A$ is the topological space $\| r\|$ for any realizing assignment $r$ of $A$ together with the map \eqref{E:reaem}. As explained above, any two geometric realizations $\| r\|$ and $\| r'\|$ of $A$ are related by a homeomorphism induced via the affine extension of the map \eqref{E:relaf}. This isomorphism type of geometric realizations is called {\bf the geometric realization} of $A$ and is denoted by 
\[
\|A\|.
\]

Let $x\in \|A\|$, that is, $x\in \| r\|$ for a realizing assignment $r$ for $A$. Note that, by \eqref{E:prog2}, for each $x\in \| r\|$, there exists a smallest, under the relation of inclusion, face $s$ of $A$ such that $x\in {\rm conv}(r(s))$. It is clear that $s$ does not depend on the realizing assignment $r$. We write, for $x\in \|A\|$, 
\begin{equation}\label{E:spgr} 
{\rm Sp}_{\rm go}^A(x)= s. 
\end{equation}

For more background on simplicial complexes the reader may consult \cite[Chapter 1]{Gl}, \cite{Ha}, and \cite[Sections~2.1 and 2.2]{Koz}.

\subsection{Projective Fra{\"i}ss{\'e} classes and their limits}\label{Su:prfr} 

We present here a framework that extracts in a canonical manner a topological space from a combinatorial/categorical background. The framework comes from \cite{IrSo} with 
some minor additions incorporated. In the present paper it will be applied to the family of all subdivisions of a given simplicial complex $A$ taken together with weld-division maps.
The extraction of the topological space from this combinatorial situation will produce the geometric realization of $A$, thereby giving a combinatorial presentation of 
this geometric/topological object.

For our presentation of the projective Fra{\"i}ss{\'e} framework, we fix a symbol $R$. By an {\bf interpretation of} $R$ on a set $X$ we understand a binary relation 
$R^X$ on $X$, that is, $R^X\subseteq X\times X$. 
We say that $R^X$ is a {\bf reflexive graph} if it is reflexive and symmetric as a binary relation. Assume $X$ and $Y$ are equipped with interpretations $R^X$ 
and $R^Y$ of $R$. 
Then a function $f\colon X\to Y$ is called a {\bf strong homomorphism} 
if
\begin{enumerate} 
\item[---] for all $x_1, x_2\in X$, $x_1R^X x_2$ implies $f(x_1)R^Y f(x_2)$;  

\item[---] for all $y_1, y_2\in Y$, $y_1R^Y y_2$ implies that there exist $x_1, x_2\in X$ with $y_1= f(x_1)$, $y_2=f(x_2)$, and $x_1 R^X x_2$. 
\end{enumerate}
By reflexivity of $R^Y$, a strong homomorphism is a surjective function from $X$ to $Y$.

Now, we assume we have a category $\mathcal C$ each of whose objects is a finite set equipped with an interpretation of $R$ as a reflexive graph 
and all of whose morphisms are strong homomorphisms. We stress that, in general, not all finite sets equipped with such interpretations of $R$ are objects of $\mathcal C$ and not 
all strong homomorphisms are morphisms of $\mathcal C$. In fact, this will be the case for the category considered in this paper.

Given $\mathcal C$, we define a new category ${\mathcal C}^\omega$ consisting of projective, that is, inverse, 
limits of sequences $(A_n, f_n)_n$, where each $A_n$ is an object of $\mathcal C$ and $f_n\colon A_{n+1}\to A_n$ is a morphism of $\mathcal C$. 
Each such projective limit 
\[
X=\varprojlim_n (A_n,f_n) = \{ (x_n)\in \prod_n A_n\mid f_n(x_{n+1})=x_n\hbox{ for all }n\}
\]
is equipped with the compact zero-dimensional topology it inherits from $\prod_n A_n$, where the objects $A_n$ are 
given the discrete topology, and with the natural projections ${\rm pr}^X_n\colon X\to A_n$. Furthermore, $X$ is naturally equipped with an interpretation of $R$ by letting, 
for $x,y\in X$, 
\[
x\,R^X y \;\hbox{ iff }\; {\rm pr}^X_n(x)\,R^{A_n} {\rm pr}^X_n(y)\hbox{ for all } n. 
\]
One easily checks that $R^X$ is a reflexive graph, that $R^X$ is compact when viewed as a set of pairs in $X\times X$, and that ${\rm pr}^X_n\colon X\to A_n$ is a continuous 
strong homomorphisms for each $n$. Given two such inverse limits $X=\varprojlim_n (A_n, f_n)$ and $Y=\varprojlim_n (B_n, g_n)$, 
we declare a continuous function $F\colon Y\to X$ to be a {\bf morphism in ${\mathcal C}^\omega$} if 
for each $m$, there exists $n$ and a morphism $f\colon B_n\to A_m$ in $\mathcal C$ such that 
\[
{\rm pr}_m^X\circ F = f\circ {\rm pr}_n^Y.
\]

So, objects of ${\mathcal C}^\omega$ are, just like objects of $\mathcal C$,  
sets carrying an interpretation of $R$ as a reflexive graph and its morphisms, like morphisms in $\mathcal C$, are also strong homomorphisms with respect to interpretations of $R$. The 
main addition is that objects of ${\mathcal C}^\omega$ carry a non-trivial topology and that interpretations of $R$ and morphisms respect this topology. 
We view $\mathcal C$ as a full subcategory of ${\mathcal C}^\omega$ by identifying each object $A$ of $\mathcal C$ 
with the projective limit $\varprojlim_n(A_n, f_n)$, where $A_n=A$ and $f_n$ is the identity morphism on $A_n$. Note that for an object $A$ of $\mathcal C$ and 
an object $X=\varprojlim_n (A_n, f_n)$ of ${\mathcal C}^\omega$, 
a continuous function $F\colon X\to A$ is a morphism precisely when there exist $n$ and a morphism $f\colon A_n\to A$ in $\mathcal C$ such that $F= f\circ {\rm pr}_n^X$.

We say that $\mathcal C$ is a {\bf projective Fra{\"i}ss{\'e} class} if it fulfills the following two conditions: 
\begin{enumerate}
\item[(i)] for any two objects $A,B$ in $\mathcal C$, there exist an object $C$ in $\mathcal C$ and morphisms $f\colon C\to A$ and $g\colon C\to B$ in $\mathcal C$; 

\item[(ii)] for any two morphisms $f, g$ in $\mathcal C$ with the same codomain, there exist morphisms $f', g'$ in $\mathcal C$ such that $f\circ f'= g\circ g'$. 
\end{enumerate} 
Condition (i) is called the {\bf joint projection property} and condition (ii) is called the {\bf projective amalgamation property}. The following fact was established in \cite{IrSo}. 

\smallskip

\noindent {\em Let $\mathcal C$ be a countable projective Fra{\"i}ss{\'e} category. There exists an object ${\mathbb C}$ in ${\mathcal C}^\omega$ such that 
\begin{enumerate}
\item[($\alpha$)] for each object $A$ of $\mathcal C$, there exists a morphism ${\mathbb C}\to A$;

\item[($\beta$)] for each object $A$ of $\mathcal C$ and morphisms $f, g\colon {\mathbb C}\to A$, there exists an isomorphism $\phi\colon {\mathbb C}\to {\mathbb C}$ 
such that $f=g\circ\phi$; 

\item[($\gamma$)] for all objects $A, B$ of $\mathcal C$ and morphisms $f\colon {\mathbb C}\to A$ and $g\colon B\to A$, there exists a morphism $h\colon {\mathbb C}\to B$ 
such that $f=g\circ h$.
\end{enumerate}
Furthermore, $\mathbb C$ is unique up to an isomorphism in ${\mathcal C}^\omega$ with properties (i) and (ii) and properties (i) and (iii).}
\smallskip

We call the unique $\mathbb C$ in the statement above the {\bf projective Fra{\"i}ss{\'e} limit} of $\mathcal C$. Property (i) of $\mathbb C$ is called {\bf projective universality}, while property (ii) is called {\bf projective homogeneity}.

A projective Fra{\"i}ss{\'e} category $\mathcal C$ will be called {\bf transitive} if the interpretation $R^{\mathbb C}$ of $R$ on the projective Fra{\"i}ss{\'e} limit 
$\mathbb C$ of $\mathcal C$ is transitive. Since $R^{\mathbb C}$ is always reflexive, symmetric, and compact, we see that, for a transitive category $\mathcal C$, $R^{\mathbb C}$ 
is a compact equivalence relation. It leads to the definition of the canonical quotient topological space 
\[
{\mathbb C}/R^{\mathbb C}.
\]
We call this space the {\bf canonical quotient space} of the class $\mathcal C$ or simply the {\bf canonical quotient} of $\mathcal C$. 

We outline a construction of the projective Fra{\"i}ss{\'e} limit for a countable projective Fra{\"i}ss{\'e} class $\mathcal C$. We call an inverse system $(C_n, f_n)$, with 
$C_n$ and $f_n\colon C_{n+1}\to C_n$ in ${\mathcal C}$, {\bf generic} if 
\begin{enumerate} 
\item[(i)] for each $A$ in $\mathcal C$, there exists $n$ and $f\colon C_{n+1}\to A$ in $\mathcal C$; 

\item[(ii)] for each $g\colon B\to A$ in $\mathcal C$, each $n$ and $f\colon A_n\to A$ in $\mathcal C$, there exists $n'\geq n$ and $f'\colon A_{n'+1}\to B$ in $\mathcal C$ such that 
\[
f\circ {\rm pr}^{n'+1}_n = g\circ f'.
\]
\end{enumerate}  
Here and below we set 
\[
{\rm pr}^l_k = f_k\circ \cdots \circ f_{l-1}, \hbox{ for }k<l.
\]
One sees that the inverse limit $\varprojlim (C_n, f_n)$ of a generic sequence yields a structure $({\mathbb C}, R^{\mathbb C})$ that fulfills conditions ($\alpha$), ($\beta$), and ($\gamma$) above, it is, that is a projective Fra{\"i}ss{\'e} limit of $\mathcal C$. Furthermore, one checks that each structure fulfilling conditions ($\alpha$), ($\beta$), and ($\gamma$) can be represented as the inverse limit of a generic sequence.

One checks that if $(C_n, f_n)$ and $(C_n', f_n')$ are generic sequences for $\mathcal C$, then there exist $m_0<n_0<m_1<n_1<\cdots$ and 
\[
g_i'\colon C_{n_i}'\to C_{m_i}\;\hbox{ and }\; g_i\colon C_{m_{i+1}}\to C'_{n_i}
\]
in $\mathcal C$ such that 
\[
g_i'\circ g_i = {\rm pr}^{m_{i+1}}_{m_i}\;\hbox{ and }\; g_i\circ g_{i+1}' = ({\rm pr}')^{n_{i+1}}_{n_i}. 
\]
So any two generic sequences are isomorphic in the sense explained above. Such an isomorphisms guarantees the uniqueness, up to isomorphism, of projective Fra{\"i}ss{\'e} limit $({\mathbb C}, R^{\mathbb C})$. 

We outline a construction of a generic sequence $(C_n, f_n)$ for $\mathcal C$. The sequence is constructed recursively. The construction needs to fulfill two types of demands:
\begin{enumerate} 
\item[(a)] for each $A$ in $\mathcal C$, we need to have $n$ and $f\colon C_{n+1}\to A$ in $\mathcal C$; 

\item[(b)] for each $g\colon B\to A$, each $n$ and $f\colon A_n\to A$, we need to have $n'\geq n$ and $f'\colon A_{n'+1}\to B$ in $\mathcal C$ such that 
\[
f\circ f_n\circ \cdots \circ f_{n'} = g\circ f'.
\]
\end{enumerate}  
There are countably many demands to satisfy as $\mathcal C$ is countable. They can be organized in a way that they are met in the recursive construction of $(C_n, f_n)$ in succession with  
\begin{enumerate} 
\item[(A)] demands of type (a) met by applying the joint projection property of $\mathcal C$, which produces $f$ and $f_n$ to satisfy (a), 

\item[(B)] demands of type (b) met by using the projective amalgamation property of $\mathcal C$, which produces $f'$ and $f_{n'}$ to satisfy (b). 
\end{enumerate}

\subsection{Set theoretic background}\label{A:sett}

We keep our set theoretic considerations on a somewhat informal footing. But all that is done in the paper can be accomplished 
assuming a weak fragment of the standard set theory (ZF) augmented by the existence of urelements. For example, 
the Kripke--Platek theory with urelements such as ${\rm KPU}^+$ in \cite[Chapter 1]{Bar}
is appropriate as the background to our considerations. 
Having urelements is convenient in calculations and is natural for our approach to modeling of complexes and their divisions. They are, however, not necessary; see below.

Urelements are objects that themselves do not have elements but they can be elements of sets. We will assume that there exists a set $\rm Ur$ of all urelements. The reader may consult \cite{Bar} for more information on this. Their introduction in this paper is done exclusively for computational convenience. Working in the usual set theory without urelements one can replace $\rm Ur$ by a set ${\rm Ur}'$ of sets such that $x\not\in {\rm tc}(x')$ for all $x,x'\in {\rm Ur}'$. Such ${\rm Ur}'$ of arbitrary size can be easily produced in the standard system ${\rm ZF}$ of set theory. One then runs all the definitions and arguments in our paper with ${\rm Ur}$ replaced by ${\rm Ur}'$. It does seem, however, that introducing $\rm Ur$ makes for a cleaner presentation. See \cite{Bar} for a more carefully argued case for urelements.

An important property of the membership relation $\in$ that is used in the paper is its {\bf well foundedness}, namely, there does not exist a sequence $x_n$, $n\in {\mathbb N}$, such that $x_{n+1}\in x_n$, for each $n$. In particular, $x\not\in x$, for each set $x$.

For an arbitrary set $x$, let 
\[
\cup x = \{ z\mid z\in y\hbox{ for some } y\in x\}.
\]
So $\cup x$ is the union of $x$, where we treat $x$ as a family of sets, namely, the family of those sets $y$ that are 
elements $x$. Set further 
\[
\cup^0 x = x\; \hbox{ and }\, \cup^{n+1} x = \cup\big( \cup^n x\big). 
\]
For a set $x$, define 
\[
{\rm tc}(x) = \bigcup_{n=0}^\infty \cup^n x. 
\]
So elements of ${\rm tc}(x)$ are elements of $x$, elements of elements of $x$, etc. That is, ${\rm tc}(x)$ collects all sets used to build $x$, and it is called the {\bf transitive closure of} $x$. 
We observe that, for sets $x$ and $y$, 
\begin{equation}\label{E:trpr}
\begin{split}
{\rm tc}(x) &\subseteq {\rm tc}(y),\hbox{ if $x\in y$ or $x\subseteq y$},\\
{\rm tc}(x\cup y) &= {\rm tc}(x)\cup {\rm tc}(y),\\
{\rm tc}\big(\{ x\}\big) &= \{ x\} \cup {\rm tc}(x).
\end{split} 
\end{equation}

For more background on set theory, the reader may consult \cite{Kun}.

\section{Faces of divided complexes}\label{A:faces} 

The following lemma is an immediate consequence of the definition of division. 

\begin{lemma}\label{L:calc} 
Let $A$ be a divided complex, let $s\in {\rm Fin}^+$, and let $x\in {\rm Fin}^+$. 
\begin{enumerate} 
\item[(i)] If $s\in x$ and $s$ is not a vertex of $A$, then 
\[
x\in s\, A\Leftrightarrow \big( s\cup (x\setminus \{ s\})\in A\hbox{ and } s\not\subseteq x\big).
\]

\item[(ii)] If $s\not\in x$, then 
\[
x\in s\, A\Leftrightarrow \big( x\in A\hbox{ and } s\not\subseteq x\big).
\]
\end{enumerate}
\end{lemma}

In the next lemma, given a divided complex $A$ and an additive family $T$ of faces of $A$, we give a description of faces of $TA$. Note that 
\[
{\rm Vr}(TA)\subseteq {\rm Vr}(A)\cup T, 
\]
from which we see that each face of $TA$ is of the form 
\begin{equation}\label{E:xX}
x\cup X, 
\end{equation} 
where $x\subseteq {\rm Vr}(A)$ and $X\subseteq T$. Lemma~\ref{L:linear} below determines precisely which sets of the form \eqref{E:xX} are actually faces of $TA$.

\begin{lemma}\label{L:linear}
Let $A$ be a divided complex and let $T$ be an additive set of faces of $A$. Let $x\subseteq {\rm Vr}(A)$ and let $X\subseteq T$. 
Then 
\[
x\cup X \hbox{ is a face of } TA 
\]
if and only if the following conditions hold 
\begin{enumerate} 
\item[(a)] $t\not\subseteq x$ for each $t\in T$;

\item[(b)] $x\cup \bigcup X$ is a face of $A$;

\item[(c)] $X$ is linearly ordered by $\subseteq$;

\item[(d)] for each $s\in X$ and $t\in T$, if $t\subseteq x\cup s$, then $t\subseteq s$.
\end{enumerate} 
\end{lemma}

\begin{proof} We first note that if $X$ is empty, then the current lemma is an immediate consequence of Lemma~\ref{L:calc}(ii). We have to show that $x$ is a face of $TA$ if and only if (a) and (b) hold (as (c) and (d) are fulfilled vacuously). We order $T$ as $t_0\cdots t_n$ in a non-decreasing manner. Fix $i\leq n$, note that $t_i\not\in x$, since $x\subseteq {\rm Vr}(A)$ and $t_i$ is a face of $A$.  Now, Lemma~\ref{L:calc}(ii) implies that $x$ is a face of $t_it_{i+1}\cdots t_n A$ if and only if $x$ is a face of $t_{i+1}\cdots t_nA$ and $t_i\not\subseteq x$.  
The equivalence from the conclusion follows. 

Now, we tackle direction $\Leftarrow$. It is proved by induction on the size of $X$. 
The case $X=\emptyset$ is already taken care of. Assume now that $X$ has $k$ elements with $k\geq 1$.
Let $s'$ be the smallest under inclusion element of $X$, which exists by (c), and let 
\[
T' = \{ s\in T\mid s\not\subseteq s' \}\; \hbox{ and }\; X'= \{ s\in X\mid s\not= s'\}.
\]
Then $T'$ is additive in $A$ and $X'$ is a linearly ordered subset of $T'$ of size $k-1$. 
Consider $T'A$. It is easy to check that conditions (a--d) hold for $T'$, $X'$, and $x\cup s'$. We only point out that 
(b) and (c) are immediate from (b) and (c) for $T, X, x$; to see (a), we apply (d) for $T, X, x$; to see (d) we apply (c) and (d) for $T, X, x$ and use the choice of $s'$. 
So, by the induction hypothesis for $T', X', x\cup s'$, we get 
\begin{equation}\label{E:gar}
x\cup s' \cup X' \hbox{ is a face of } T'A. 
\end{equation} 
By (a) for $T, X, x$ and since $s\not\in s'$ for $s\in T'$, we have 
\begin{equation}\label{E:xsp2} 
s'\not\subseteq x\cup  X'.
\end{equation}
Observe also that $s'$ is not a vertex of $T'A$ as $s'\not\in T'$ and $s'$ is not a vertex of $A$ since it is a face of $A$. Now, from \eqref{E:gar} and \eqref{E:xsp2}, by Lemma~\ref{L:calc}(i), we get 
\[
x \cup X' \cup \{ s'\} \hbox{ is a face of } s'T'A, 
\]
that is, 
\begin{equation}\label{E:ysp}
x\cup  X \hbox{ is a face of } s' T' A.
\end{equation} 

Note that $T\setminus (T'\cup \{ s'\})$ is an additive family of faces of $s' T' A$. Additivity is obvious since $T$ is additive and $s'$ is not a face of $s'T'A$; to see that $T\setminus (T'\cup \{s'\})$ consist of faces of $s'T'A$ apply Lemma~\ref{L:calc}(ii) and observe that no set from $\{ s'\} \cup T'$ is included in a set from $T\setminus (T'\cup \{ s'\})$.  
Furthermore, by (a) for $T, X, x$, we have that, for each $t\in T\setminus (T'\cup \{ s'\})$,  
\[
t\not\subseteq x\cup X.
\]
Thus, it follows from \eqref{E:ysp} and the conclusion of the lemma with $X=\emptyset$ that 
\[
x\cup X \in \big(T\setminus (T'\cup \{ s'\})\big)\, s'\, T' A = TA,
\]
as required. 

Now we show $\Rightarrow$. Assume $x\cup X$ is a face of $TA$. Then $x$ is a face of $TA$ and (a) follows from the case $X=\emptyset$ proved above. 

Since the case $X=\emptyset$ is taken care of in the beginning of the proof, we assume that $X\not=\emptyset$ in what follows. 

Fix $s\in X$. 
We show that if $x\cup \{ s\}\in TA$, then $x\cup s \in A$ and, for each $t\in T$, if $t\subseteq x\cup s$, then $t\subseteq s$. Note, that this will prove (d) and, also, it will prove (b) as long as we show (c) (which we do below). Fix a non-decreasing enumeration $t_0\cdots t_i\cdots t_n$ of $T$ with $t_i=s$ and such that of $t_j\not\subseteq s$, then $j>i$. Since $x\cup \{ s\}$ is a face of $TA$, $s$ is not a vertex of $A$, and $s\not= t_j$ for all $j\not= i$, we have that $x\cup \{ s\}$ is a face of $t_i\cdots t_nA$.  Thus, by Lemma~\ref{L:calc}(i), we get that $x\cup s$ is a face of $t_{i+1}\cdots t_nA$. 
Since $x\cup s \subseteq {\rm Vr}(A)$, so $t_j\not\in x\cup s$ for all $j$, we see, by Lemma~\ref{L:calc}(ii), that $x\cup s\in A$ and $t_j\not\subseteq x\cup s$ for all $j>i$. It follows that if $t\in T$ and $t\subseteq x\cup s$, then $t=t_j$ for $j\leq i$, which, by the choice of our enumeration of $T$, implies that $t\subseteq s$. 

It remains to show (c). Towards a contradiction, assume that $x\cup X\in TA$ and there are $s_1, s_2\in X$ such that $s_1\not\subseteq s_2$ and 
$s_2\not\subseteq s_1$. Then $\{ s_1, s_2\}$ is a face of $TA$. We show that this leads to a contradiction. 
Let $T$ be non-decreasingly enumerated as $t_0\cdots t_n$. Then, for some $i, j\leq n$, we have 
\begin{equation}\label{E:thr}
\begin{split}
&s_1= t_i\hbox{ and  }s_2 = t_j, \hbox{ and }\\
&\big(s_1\cup s_2 = t_k,\hbox{ for some }k>\max (i,j),\hbox{ or } s_1\cup s_2 \hbox{ is not a face of }A\big). 
\end{split}
\end{equation} 
By symmetry, we can assume $i<j$. Then, by Lemma~\ref{L:calc}(i), 
$s_1\cup \{ s_2\}$ is a face of 
$t_{i+1}\cdots t_n A$. This holds precisely when $s_1\cup s_2$ is a face of $t_{j+1}\cdots t_n\, A$, by Lemma~\ref{L:calc}(i) again. But $s_1\cup s_2$ is not a face of this complex by the third conjunct of \eqref{E:thr}, a contradiction.  
\end{proof}

We also register the following lemma with the description of the set of vertices a complex divided by an additive family of faces. 

\begin{lemma}\label{L:vrst} 
Let $A$ be a divided complex, and let $T$ be an additive family of faces of $A$. Then 
\begin{equation}\notag
{\rm Vr}(TA\,) = T\cup \Big( {\rm Vr}(A\,)\setminus \{ y\mid \{ y\}\in T\}\Big).
\end{equation} 
\end{lemma} 

\begin{proof} Clearly, we have ${\rm Vr}(TA\,) \subseteq  T\cup {\rm Vr}(A\,)$. We note that for each $t'\in T$ and $y'\in  {\rm Vr}(A\,)\setminus \{ y\mid \{ y\}\in T\}$, 
the sets $\{ t'\}$ and $\{ y'\}$ satisfy the conditions of Lemma~\ref{L:linear}, so they are faces of $TA$, which makes $t'$ and $y'$ into vertices of $TA$. On the other hand, for $y'\in {\rm Vr}(A)$ with $\{ y'\}\in T$, we see that $\{ y'\}$ contains an element of $T$ as a subset, so it does not fulfill condition (a). It follows that $y'\not\in {\rm Vr}(TA)$. 
\end{proof}

We will need a precise description of faces of complexes after two and three divisions. This is done in the two lemmas below.

\begin{lemma}\label{L:calc2} 
Let $A$ be a divided complex, and let $s,t\in {\rm Fin}^+$ be such that $s$ is not a vertex of $tA$, and $t$ is not a vertex of $A$. 
Let $x$ be in ${\rm Fin}^+$, and put 
$x^0 = x\setminus \{ s,t\}$. 
We have 
\begin{equation}\notag
x\in s\, t A \Longleftrightarrow 
\begin{cases}
\big(s\setminus \{ t\}\big) \cup t\cup x^0 \in A,\; (t\setminus s)\not\subseteq x^0,\; s\not\subseteq x, &\hbox{ if } s\in x,\, t\in s\cup x,\\ 
s\cup x^0  \in A,\; (t\setminus s)\not\subseteq x^0,\; s\not\subseteq x, &\hbox{ if } s\in x,\, t\not\in s\cup x,\\
t\cup  x^0  \in A,\; t\not\subseteq x^0,\; s\not\subseteq x, &\hbox{ if } s\not\in x,\, t\in x,\\
x^0 \in A,\; t\not\subseteq x^0,\; s\not\subseteq x, &\hbox{ if } s\not\in x,\, t\not\in x.
\end{cases}
\end{equation} 
\end{lemma} 

\begin{proof} 
Two careful applications of Lemma~\ref{L:calc} yield the conclusion. 
\end{proof}

The next lemma is a technical consequence of Lemmas~\ref{L:calc} and \ref{L:calc2}. Given two faces $s, t$ of a complex $A$, it determines the faces of the complex $u\,s\,t\,A$, for an appropriately chosen $u$. The lemma will be crucially used in the proofs of Theorems~\ref{T:ord3} and \ref{T:ord}.

\begin{lemma}\label{L:calc4} 
Let $A$ be a divided complex, let $s,t\in A$, and let $r$ be a set such that 
\[
s\setminus t\subseteq r \subseteq s.
\] 
Set $u= r\cup \{ t\}$. Let $y\in {\rm Fin}^+$, and set $y^0 = y\setminus \{ s,t,u\}$.

\begin{enumerate} 
\item[(A)] Assume $s\not= t$. The following statements hold. 
\begin{enumerate}
\item[(i)]  If $u\in y$, 
then $y\in u\, s\, t A$ if and only if 
\begin{equation}\notag
s\cup t\cup y^0 \in A,\; (s\setminus r)\not\subseteq y^0,\;  r \cup \{ t\}\not\subseteq y,\; 
(t\setminus s)\cup \{ s\} \not\subseteq y. 
\end{equation}

\item[(ii)]  If $u\not\in y$, then $y\in u\, s\, t A$ if and only if 
\begin{equation}\notag
\begin{cases}
s\cup t\cup y^0 \in A,\; (t\setminus s)\not\subseteq y^0,\; r\not\subseteq y^0, &\hbox{ if } s\in y,\, t\in y,\\ 
t\cup  y^0  \in A,\; t\not\subseteq y^0,\; r\not\subseteq y^0, &\hbox{ if } s \not\in y, \, t\in y,\\
s\cup y^0  \in A,\; (t\setminus s)\not\subseteq y^0,\; s\not\subseteq y^0, &\hbox{ if } s \in y,\, t\not\in y,\\
y^0 \in A,\; t\not\subseteq y^0,\; s\not\subseteq y^0, &\hbox{ if } s \not\in y,\, t\not\in y.
\end{cases}
\end{equation} 
\end{enumerate}

\item[(B)] Assume $s=t$. Then $s\,t\,A= tA$ and the following statements hold. 
\begin{enumerate}
\item[(i)]  If $u\in y$, 
then $y\in u\, t A$ if and only if 
\begin{equation}\notag
t\cup y^0 \in A,\; (t\setminus r)\not\subseteq y^0,\;  r \cup \{ t\}\not\subseteq y. 
\end{equation} 

\item[(ii)]  If $u\not\in y$ and $t \in y$, then $y\in u\, t A$ if and only if 
\begin{equation}\notag
\begin{cases} 
 t\cup y^0 \in A,\; r\not\subseteq y^0,& \hbox{ if $t\in y$,}\\
 y^0 \in A,\; t\not\subseteq y^0,& \hbox{ if $t\not\in y$.}
 \end{cases}
\end{equation} 
\end{enumerate}
\end{enumerate} 
\end{lemma}

\begin{proof} We freely use Lemma~\ref{L:ok} in this proof, so, for example, from the assumptions $s, t\in A$, we get $s, t\not\in {\rm tc}\big( {\rm Vr}(A)\big)$, $s\not\in t$, and $t\not\in s$. 

Observe that $u$ is not a vertex of $stA$ since $u\not= s$ (as $t\not\in s$), $u\not= t$ (as $t\not\in t$), and $u$ is not a vertex of $A$ (as $t\not\in {\rm tc}\big( {\rm Vr}(A)\big)$). Also $t$ is not a vertex of $A$ (as $t\in A$) and $s$ is not a vertex of $tA$ (as $t\not\in s$ and $s$ is not a vertex of $A$). 
Thus, Lemmas~\ref{L:calc} and \ref{L:calc2} are applicable.

We prove (A). 
Assume first that $u \in y$.
Applying Lemma~\ref{L:calc}(i), we conclude that $y\in u\, s\, t A$ precisely when 
\begin{equation}\label{E:uull}
u\cup \big( y\setminus \{ u \} \big) \in s\, t A\;\hbox{ and }\;u\not\subseteq y.
\end{equation}  
To 
unravel the condition $u\cup \big( y\setminus \{ u\} \big) \in s\, t A$, we apply Lemma~\ref{L:calc2} with 
\[
x= u\cup \big( y\setminus \{ u\} \big). 
\]
We note that  $t \in u$, so 
\[
t \in x, 
\]
and, therefore, 
only the first and third cases on the right hand side of the equivalence in Lemma~\ref{L:calc2} are relevant. We note that 
$s \not\in u$, since $s\not= t$,  and, therefore, 
\[
s \in x\;\Leftrightarrow\; s \in y. 
\]
Now, we apply Lemma~\ref{L:calc2} with $x=u\cup \big( y\setminus \{ u \} \big)$. Keeping in mind that 
\[
x^0= r\cup y^0,\; t\not\in s,\; r\subseteq s,
\]  
we see that $x\in s\, t A$ holds precisely when 
\begin{equation}\notag
\begin{cases}
s\cup t\cup y^0 \in A,\; (t\setminus s)\not\subseteq y^0,\; (s\setminus r)\not\subseteq y^0, &\hbox{ if } s \in y,\\ 
r\cup t\cup y^0  \in A,\;  (t\setminus r)\not\subseteq  y^0,\; (s\setminus r) \not\subseteq y^0, &\hbox{ if }s\not\in y.\\
\end{cases}
\end{equation} 
Observe that by using our assumptions on $r$, we have $r\cup t = s\cup t$ and $s\setminus r\subseteq t\setminus r$, so the above formulas can be restated as 
\begin{equation}\notag
\begin{cases}
s\cup t\cup y^0 \in A,\; (t\setminus s)\not\subseteq y^0,\; (s\setminus r)\not\subseteq y^0, &\hbox{ if } s \in y,\\ 
s\cup t\cup y^0  \in A,\;  (s\setminus r) \not\subseteq y^0, &\hbox{ if } s\not\in y.\\
\end{cases}
\end{equation} 
Thus, using the second part of \eqref{E:uull}, we reach the conclusion that, under the assumption $u\in y$, the condition $y\in u\, s\, t A$ is equivalent to 
\begin{equation}\label{E:eqlar}
s\cup t\cup y^0 \in A,\, r \cup \{ t \}  \not\subseteq y,\, (s\setminus r)\not\subseteq y, \hbox{ and, when }s \in y,\, (t\setminus s)\not\subseteq y.
\end{equation}
The last part of the right hand side of \eqref{E:eqlar} can be rephrased as $(t\setminus s)\cup \{ s\}\not\subseteq y$, which transforms \eqref{E:eqlar} into (i).

Now, assume $u \not\in y$. We apply Lemma~\ref{L:calc}(ii) to see that $y\in u\, s\, t A$ precisely when 
$y \in s\, t A$ and $u\not\subseteq y$; this last condition being equivalent to the disjunction
$r\not\subseteq y$ or $t\not\in y$. Thus, under this condition, using Lemma~\ref{L:calc2} with $x=y$ and keeping in mind that $u,s\not\in r$ and $u,t\not\in s$, 
we get the statements in (ii).

We show (B). We apply Lemma~\ref{L:calc2} (with $s=u$ and $t$ and $x=y$). Since $t\in u$, we see that the conditions from Lemma~\ref{L:calc2} equivalent to $y\in u\,t\,A$ translate to 
\begin{equation}\notag
\begin{cases}
r \cup t\cup y^0 \in A,\; (t\setminus r)\not\subseteq y^0,\; u\not\subseteq y, &\hbox{ if } u\in y,\, t\in u\cup y,\\ 
t\cup  y^0  \in A,\; t\not\subseteq y^0,\; u\not\subseteq y, &\hbox{ if } u\not\in y,\, t\in y,\\
y^0 \in A,\; t\not\subseteq y^0,\; u\not\subseteq y, &\hbox{ if } u\not\in y,\, t\not\in y.
\end{cases}
\end{equation} 
with the second condition in Lemma~\ref{L:calc2} omitted since $t\not\in u\cup y$ fails. Now after noticing that $u= r\cup \{t \}$ and $r\subseteq t$, we restate the assertions above as 
\begin{equation}\notag
\begin{cases}
t\cup y^0 \in A,\; (t\setminus r)\not\subseteq y^0,\; r\cup \{ t\} \not\subseteq y, &\hbox{ if } u\in y,\\ 
t\cup  y^0  \in A,\; t\not\subseteq y^0,\; r\cup \{ t\} \not\subseteq y, &\hbox{ if } u\not\in y,\, t\in y,\\
y^0 \in A,\; t\not\subseteq y^0,\; r\cup \{ t\}\not\subseteq y, &\hbox{ if } u\not\in y,\, t\not\in y.
\end{cases}
\end{equation} 
Seeing that $u\not\in r$, we get the final restatement 
\begin{equation}\notag
\begin{cases}
t\cup y^0 \in A,\; (t\setminus r)\not\subseteq y^0,\; r\cup \{ t\} \not\subseteq y, &\hbox{ if } u\in y,\\ 
t\cup  y^0  \in A,\; t\not\subseteq y^0,\; r \not\subseteq y^0, &\hbox{ if } u\not\in y,\, t\in y,\\
y^0 \in A,\; t\not\subseteq y^0, &\hbox{ if } u\not\in y,\, t\not\in y.
\end{cases}
\end{equation} 
Clearly, this is the content of (B). 
\end{proof}

\section{Proofs of grounded simpliciality of maps}\label{A:isome}

In this section, we prove that various maps are grounded simplicial. In light of Lemma~\ref{L:simsp}, our task is as follows. We are given two divided complexes $A$ and $B$ and a map 
$f\colon {\rm Vr}(B)\to {\rm Vr}(A)$. We need to show that $f$ fulfills the following two conditions:
\begin{enumerate}
\item[(S1)] if $t\in B$, then $f(t)\in A$ and ${\rm sp}\big( f(t)\big)\subseteq {\rm sp}(t)$;

\item[(S2)] if $s\in A$, then there exists $t\in B$ with $f(t)=s$ and ${\rm sp}(t)= {\rm sp}(s)$.
\end{enumerate}

\subsection{Weld maps}\label{Su:weldpro} 

We restate the definition of weld maps given in Section~\ref{S:divmap} in a manner appropriate to the setup of Section~\ref{S:divset}.
Let $A$ be a divided complex and let $s$ be in ${\rm Fin}^+$.
We have 
\[
{\rm Vr}(sA) =
\begin{cases}
{\rm Vr}(A), &\hbox{ if }s\not\in A;\\
\big({\rm Vr}(A)\setminus s\big)\cup \{ s\}, &\hbox{ if }s\in A\hbox{ and } s\hbox{ is a one-element set};\\
{\rm Vr}(A) \cup \{ s\}, &\hbox{ if }s\in A\hbox{ and }s \hbox{ has at least two elements}.
\end{cases} 
\]
If $s\not\in {\rm Vr}(A)$, define 
\[
\pi^A_{p,s} (y) = 
\begin{cases}
y, &\hbox{ if } y\in {\rm Vr}(A);\\
p, &\hbox{ if }y=s.
\end{cases} 
\]

\begin{lemma}\label{L:welgr2} 
Let $A$ be a divided complex, let $s\in {\rm Fin}^+$ not be a vertex of $A$, and let $p\in s$. Then the weld map $\pi^A_{p.s}\colon sA\to A$ is grounded simplicial. 
\end{lemma}

\begin{proof}
Set $\pi^A_{p,s}=\pi$. We check that $\pi$ fulfills (S1). If $t$ is an old face of $sA$, then $t\in A$ and $\pi(t)=t$; obviously, in this case, ${\rm sp}(t)={\rm sp}(\pi(t))$. If 
$\{ s\}\cup t$ is a new face of $sA$, then $s\cup t\in A$, so $\{ p\}\cup t\in A$ as $p\in s$, and $\pi(\{ s\}\cup t)= \{ p\}\cup t$. Clearly, since $p\in s$, we have 
\[
{\rm sp}\big(\pi(\{s\}\cup t)\big)\subseteq {\rm sp}(\{ s\}\cup t).
\]

To see that $\pi$ satisfies (S2), let $t$ be a face of $A$. If $s\not\subseteq t$, then $t$ is a face of $sA$, $\pi(t)=t$, and ${\rm sp}(t)={\rm sp}\big(\pi(t)\big)$. If $s\subseteq t$, then 
\[
\{ s\}\cup (t\setminus \{ p\}) \in sA 
\]
since $s\cup (t\setminus \{ p\}) = t\in A$. We also have
\[
\pi\big( \{ s\}\cup (t\setminus \{ p\}) \big) = \{ p \}\cup \big( t\cup \{ p\}\big)= t,
\]
and, by $s\cup (t\setminus \{ p\}) = t$,
\[
{\rm sp}\big( \{ s\}\cup (t\setminus \{ p\}) \big) = {\rm sp}(t), 
\]
and the lemma is proved. 
\end{proof}

We now restate the definition of maps of the form $\pi_\iota^A$. 
Let $A$ be a divided complex, and let $S$ be an additive family of faces of $A$. Observe that 
\[
{\rm Vr}(SA) = \big( {\rm Vr}(A) \setminus \{ x\mid \{ x\}\in S\} )\cup S.
\]
Let $\iota\colon S\to \bigcup S$ be such that $\iota(s)\in s$ for each $s\in S$. Let 
\[
\pi^A_\iota\colon SA\to A
\] 
be the map defined by 
\[
\pi^A_\iota(y) = 
\begin{cases}
y, &\hbox{ if } y\in {\rm Vr}(A) \setminus \{ x\mid \{ x\}\in S\};\\
\iota(y), &\hbox{ if }y\in S.
\end{cases} 
\]

\begin{lemma}\label{L:ioco2}
Let $A$ be a divided complex, let $S$ be a an additive family of faces of $A$, and let $\iota\colon S\to \bigcup S$ be such that $\iota(s)\in s$. 
Then $\pi^A_\iota$ is a composition of weld maps. 
\end{lemma}

\begin{proof} 
Let $\vec{S}= s_0\cdots s_n$ be a non-decreasing enumeration of $S$. For $i\leq n$,
\[
\pi_i=\pi^{s_{i}\cdots s_n A}_{\iota(s_i), s_i}\colon s_is_{i+1}\cdots s_nA \to s_{i+1}\cdots s_nA
\]
given by $s_i\to \iota(s_i)$ is a weld map and 
\[
\pi_\iota= \pi_n\circ\cdots \circ \pi_0.
\]
So $\pi^A_\iota$ is a composition of weld maps. 
\end{proof}

\subsection{Division of grounded simplicial maps} 

Recall the definition of division of grounded simplicial maps from Section~\ref{Su:stdm}. We restate it here with the notation appropriate to divided complexes. Let $A$ and $B$ be divided complexes. Given a grounded simplicial map $f\colon B\to A$ and $s\in A$, define a map 
\[
sf\colon {\rm Vr}\big( \big(f^{-1}(s)\big) B\big) \to {\rm Vr}(sA)
\]
as follows. By Lemma~\ref{L:vrst}, we have 
\[
{\rm Vr}\big( \big(f^{-1}(s)\big) B\big) = f^{-1}(s) \cup \Big( {\rm Vr}(B)\setminus \big\{ x\in {\rm Vr}(B)\mid \{ x\}\in f^{-1}(s)\big\}\Big)
\]
and 
\[
{\rm Vr}(sA) = 
\begin{cases} 
\{ s\}\cup {\rm Vr}(A), &\hbox{ if }\#s >1;\\
\{ s\}\cup \big({\rm Vr}(A)\setminus s\big), &\hbox{ if }\#s =1.
\end{cases}
\]
Define 
\[
(sf)(v)= 
\begin{cases}
s, & \hbox{ if } v\in f^{-1}(s);\\
f(v), & \hbox{ if } v\in {\rm Vr}(B)\setminus \{ x\in {\rm Vr}(B)\mid \{ x\}\in f^{-1}(s).
\end{cases}
\]

\begin{lemma}\label{L:divpr} 
Assume $A, B$ are divided complexes. Let $f\colon B\to A$ be a grounded simplicial map, and let $s$ be a face of $A$. 
Then the map 
\[
sf\colon \big( f^{-1}(s)\big) B\to sA
\]
is grounded simplicial. 
\end{lemma}

\begin{proof} We freely use properties \eqref{E:sppr} of the operation $\rm sp$ in the calculations below.

We start with proving that $sf$ fulfills (S1). By Lemma~\ref{L:linear} with $T=f^{-1}(s)$, a face of $f^{-1}(s)B$ can have one of two forms: 
\begin{enumerate} 
\item[---] it is $r$ with $r\in B$ and $t\not\subseteq r$, for each $t\in f^{-1}(s)$, or 

\item[---] it is $r\cup X$, where $X$ is a non-empty linearly ordered by inclusion subset of $f^{-1}(s)$, with $r\cup \bigcup X\in B$ and $t\not\subseteq r$, for each $t\in f^{-1}(s)$. 
\end{enumerate} 

In the first case, $(sf)(r)= f(r)\in A$. So, to check that $sf(r)\in sA$, it suffices to see that $s\not\subseteq f(r)$. But otherwise, there exists $t\subseteq r$ with $f(t)=s$, that is, $t\in f^{-1}(s)$, contradicting the assumptions on $r$. Note also that since $f$ fulfills (S1), 
\[
{\rm sp}\big((sf)(r)\big)= {\rm sp }\big(f(r)\big)\subseteq {\rm sp}(r).
\]

In the second case, 
\[
(sf)(r\cup X) = f(r)\cup \{ s\}. 
\]
To see that the right side is a face of $sA$, by Lemma~\ref{L:calc}(i), we need to show that $f(r)\cup s \in A$ and $s\not\subseteq f(r)$. But, 
\[
f(r)\cup s = f(r\cup \bigcup X),
\]
which is a face of $A$ since $r\cup \bigcup X$ is a face of $A$ and $f$ has (S1). The checking that the condition $s\not\subseteq f(r)$ holds is the same as in the first case. 
Furthermore, since $f$ fulfills (S1), we get 
\[
{\rm sp}\big((sf)(r\cup X)\big) = {\rm sp}\big(f(r)\cup \{ s\}\big) \subseteq {\rm sp}\big( r\cup \{ s\}\big) \subseteq {\rm sp}(r\cup X)
\]
with the last inclusion following from $f(t)=s$ for $t\in X$ and $f$ having (S1).

We now check that $sf$ satisfies (S2). From the definition of division, a face of $sA$ can have one of the two forms: 
\begin{enumerate}
\item[---] $r$ with $r\in A$ and $s\not\subseteq r$;

\item[---] $r\cup \{ s\}$ with $r\cup s\in A$ and $s\not\subseteq r$.
\end{enumerate} 

In the first case, since $f$ has (S2), there exists $t\in B$ with $f(t)=r$ and ${\rm sp}(t)={\rm sp}(r)$. Note that no element of $f^{-1}(s)$ is included in $t$ since otherwise $s\subseteq r$. Thus, $t\in f^{-1}(s)B$ and $(sf)(t)= r$. 

In the second case, since $f$ fulfills (S2), there exists $t\in B$ such that $f(t)= r\cup s$ and 
\begin{equation}\label{E:sptrs} 
{\rm sp}(t)= {\rm sp}(r\cup s). 
\end{equation} 
Let $t_1\subseteq t$ be maximal with respect to inclusion with $f(t_1)=s$ and let $t_2\subseteq t$ be such that $f(t_2)=r$. We consider $t_2\cup \{ t_1\}$. Using \eqref{E:sptrs}, properties of the operation $\rm sp$, and the fact that $f$ satisfies (S1), we get 
\[
{\rm sp}(t)\supseteq {\rm sp}\big( t_2\cup \{ t_1\}\big) = {\rm sp}(t_2)\cup {\rm sp}(t_1) \supseteq {\rm sp}(r)\cup {\rm sp}(s)= {\rm sp}(r\cup s)= {\rm sp}(t), 
\]
from which, by using again the properties of $\rm sp$, we obtain 
\begin{equation}\label{E:essenc}
{\rm sp}\big( t_2\cup \{ t_1\}\big)={\rm sp}(r)\cup {\rm sp}(s) = {\rm sp}\big(r\cup \{ s\}\big).
\end{equation} 
Now, we use Lemma~\ref{L:linear} to see that 
\begin{equation}\label{E:essena} 
t_2\cup \{ t_1\} \in f^{-1}(s)\,B. 
\end{equation} 
Indeed, $u\not\subseteq t_2$, for each $u\in f^{-1}(s)$, since otherwise 
\[
s= f(u)\subseteq f(t_2)=r,
\]
which contradicts our assumptions. Also $t_1\cup t_2\subseteq t$, so $t_1\cup t_2$ is a face of $B$. Obviously, $\{ t_2\}$ is linearly ordered by inclusion, so it remains to see that, for each $u\in f^{-1}(s)$, if $u\subseteq t_1\cup t_2$, then $u\subseteq t_1$, but this follows from maximality of $t_1$. So, \eqref{E:essena} holds. We have 
\[
(sf)\big( t_2\cup \{ t_1\}\big) = f(t_2) \cup \{ s\} = r\cup \{ s\}.
\]
The above equalities and \eqref{E:essenc} finish the argument  for $sf$ having (S2). 
\end{proof}

\subsection{Combinatorial isomorphisms} 

Let $A, \, B$ be divided complexes. 
We observe that $f\colon {\rm Vr}(B)\to {\rm Vr}(A)$ is a grounded isomorphism, that is, it is an invertible grounded simplicial map whose inverse is grounded simplicial, precisely when the following two conditions hold 
\begin{enumerate} 
\item[---] $f$ is bijective and, for $t\subseteq {\rm Vr}(B)$, $t\in B$ if and only if $f(t)\in A$; 

\item[---] for $t\in B$, ${\rm sp}\big( f(t)\big)= {\rm sp}(t)$.
\end{enumerate} 
We will be checking these conditions in Theorems~\ref{T:ord3} and \ref{T:ord} below.

It is clear that in most situations the complexes $s\, t A$ and 
$t\, s A$, where $s,t$ are faces of $A$, are not isomorphic. It turns out, however,  
that one may reverse the order of $s$ and $t$ if one agrees to compensate the order change with 
applications of appropriate additional divisions. This is the content of Theorems~\ref{T:ord3} and \ref{T:ord}.

\begin{theorem}\label{T:ord3}
Let $A$ be a divided complex. Let $t\in {\rm Fin}^+$ and let $r, s\subseteq t$ be such that $r\cup s\not=\emptyset$ and $r\cap s=\emptyset$. 
Assume that $t\not\in {\rm tc}\big( {\rm Vr}(A)\big)$. Set 
\[
B_1= \big( r\cup \{ t \} \big) \, (r\cup s)\, t A\;\hbox{ and }\;B_2 = \big( s\cup \{ t \} \big) \, (r\cup s)\, t A.
\]
\begin{enumerate} 
\item[(i)] We have 
\begin{enumerate}
\item[---] $t$ is a vertex of $B_1$ $\Longleftrightarrow$ $s\cup \{ t\}$ is a vertex of $B_2$ 
$\Longleftrightarrow$ $t$ is a face of $A$ and $r\not=\emptyset$ and 

\item[---] $r\cup \{ t\}$ is a vertex of $B_1$ $\Longleftrightarrow$ $t$ is a vertex of $B_2$ $\Longleftrightarrow$ $t$ is 
a face of $A$ and $s\not=\emptyset$.
\end{enumerate} 

\item[(ii)] The complexes $B_1$ and $B_2$ are ground isomorphic via the assignment 
\[
t \to  s\cup \{ t \} , \;  r\cup \{ t \} \to  t. 
\]
\end{enumerate}
\end{theorem}

In the second theorem, one changes the orders of $s$ and $t$ at the expense of an additional division.

\begin{theorem}\label{T:ord}
Let $A$ be a divided complex, and let $s, \,t\in {\rm Fin}^+$ be such that $s\not\in t$, $t\not\in s$, and $s, t\not\in {\rm tc}\big( {\rm Vr}(A)\big)$. Set 
\begin{equation}\notag
B_1= \big( (s\setminus t)\cup \{ t \} \big) \, s\, t A \;\hbox{ and }\; B_2= \big( (t\setminus s)\cup \{ s \} \big) \, t\, s A. 
\end{equation} 

\begin{enumerate} 
\item[(i)] $(s\setminus t)\cup \{ t \}$ is a vertex of $B_1$ $\Longleftrightarrow$ 
$(t\setminus s)\cup \{ t \}$ is a vertex of $B_2$ $\Longleftrightarrow$   $s\cup t$ is a face of $A$ and $s\cap t\not=\emptyset$.

\item[(ii)] The complexes $B_1$ and $B_2$ are ground isomorphic via the assignment 
\[
(s\setminus t)\cup \{ t \} \to   (t\setminus s)\cup \{ s \}. 
\]
\end{enumerate} 
\end{theorem}

Notice that, by Lemma~\ref{L:ok}, the assumptions $t\not\in {\rm tc}\big({\rm Vr}(A)\big)$ and $s,t\not\in {\rm tc}\big({\rm Vr}(A)\big)$ in Theorems~\ref{T:ord3} and \ref{T:ord} above are implied by the simpler to state condition $t\in A$. But the theorems will be applied with the weaker assumptions.

\begin{corollary}\label{C:comg} 
Let $A$ be a divided complex, and let $s, t\in {\rm Fin}^+$ be such that $s\not\in t$ and $t\not\in s$. 
\begin{enumerate}
\item[(i)]  If $s \cup t$ is not a face of $A$, then $s\, t A = t\, s A$. 
\item[(ii)] If $s\cap t=\emptyset$, then $s\, t A = t\, s A$.
\end{enumerate} 
\end{corollary}

\begin{proof}
If $s$ is not a face of $A$, then, by the definition of dividing, since $t\not\in s$, $s$ is not a face of $tA$. It follows that $tsA=tA$ and $stA=tA$ and the conclusions of (i) and (ii) follow. The same argument gives the conclusions if $t$ is not a face of $A$. So we assume that $s$ and $t$ are faces of $A$. 

(i) Under the assumption of (i), by Theorem~\ref{T:ord}(i), we have that $(s\setminus t)\cup \{ t \}$ and 
$(t\setminus s)\cup \{ s \}$ are not faces of $stA$ and $tsA$, respectively. So, we get 
\begin{equation}\notag
\Big( (s\setminus t)\cup \{ t \} \Big) \, s\, t A = s\,t A \;\hbox{ and }\; \Big( (t\setminus s)\cup \{ s \} \Big) \, t\, s A = t\, sA.
\end{equation} 
Now, the conclusion of Theorem~\ref{T:ord}(ii) gives that the identity assignment on the vertices is a grounded isomorphism between $stA$ and $tsA$. Thus, we get the equality $st A = tsA$. 

(ii) By Theorem~\ref{T:ord}(i), $(s\setminus t)\cup \{ t\}$ and $(t\setminus s)\cup \{ s\}$ are not faces of $stA$ and $tsA$, respectively. Thus, we get 
\begin{equation}\notag
\Big( (s\setminus t)\cup \{ t \} \Big) \, s\, t A = s\,t A \;\hbox{ and }\; \Big( (t\setminus s)\cup \{ s \} \Big) \, t\, s A = t\, sA.
\end{equation}
Now, we finish the proof of (ii) by applying Theorem~\ref{T:ord}(ii). 
\end{proof}

Now we move to the proof of Theorem~\ref{T:ord3}.

\begin{proof}[Proof of Theorem~\ref{T:ord3}]
We note that since $r, s\subseteq t$, we have 
\[
{\rm sp}(t) ={\rm sp}\big(  s\cup \{ t \}\big) = {\rm sp} \big(r\cup \{ t \}\big). 
\]
Thus, it remains to check that the assignment from (ii) gives a bijective function between ${\rm Vr}(B_1)$ and ${\rm Vr}(B_2)$ such that the function and its inverse map faces to faces. 

We start with considering the case $t\not\in A$. Then, if $r\cup s=t$, we have 
\[
B_1=(r\cup \{ t\}) A\;\hbox{ and }B_2= (s\cup \{ t\}) A. 
\]
Since $t$ is not a vertex of $A$, we see that $r\cup \{ t\}$ and $s\cup \{ t\}$ are not faces of $A$, so actually 
\[
B_1=A=B_2. 
\]
Thus, by our assumption $t\not\in {\rm tc}\big( {\rm Vr}(A)\big)$, we see that $t,\, r\cup \{ t\},\, s\cup \{ t\}$ are not vertices of $B_1$ and $B_2$, so points (i) and (ii) of the lemma follow. 
If $r\cup s\subsetneq t$, then 
\[
B_1=(r\cup \{ t\}) (r\cup s) A\;\hbox{ and }B_2= (s\cup \{ t\}) (r\cup s) A. 
\]
Since $t$ is not a vertex of $A$ and since $t\not= r\cup s$, we see that $r\cup \{ t\}$ and $s\cup \{ t\}$ are not faces of $(r\cup s)A$; thus, 
\[
B_1=(r\cup s) A\;\hbox{ and }B_2= (r\cup s) A.
\]
Since $t,\, r\cup \{ t\},\, s\cup \{ t\}$ are not vertices of $A$, they are not vertices of $B_1$ and $B_2$ since otherwise $r\cup s$ would be equal to one of them, which 
is ruled out by our case assumption that $t\not= r\cup s$ and by $t\not\in r\cup s$ (as $r\cup s\subseteq t$). Now, points (i) and (ii) of the lemma follow immediately.

So, assume that $t\in A$. Note first that then $t$ is a vertex of $tA$ and it remains a vertex of $B_1$ precisely when $r\cup s\not= \{ t\}$ and $r\cup \{ t\}\not=\{ t\}$. This happens precisely when $r\not=\emptyset$ since $t\not\in r\cup s$ (as $r\cup s\subseteq t$). An analogous argument gives that $t$ is a vertex of $B_2$ precisely $s\not=\emptyset$. 
Now observe that $r\cup \{ t\}$ is a face of $tA$ since $r\cup t=t$ is a face of $A$. It follows that $r\cup \{ t\}$ is a face of $(r\cup s) tA$ precisely when $r\cup s \not\subseteq r\cup \{ t\}$, which is equivalent to saying $s\not=\emptyset$ (as $t\not\in r\cup s$). Thus, $r\cup \{ t\}$ is a vertex of $B_1$ precisely when $s\not=\emptyset$. A similar argument shows that $s\cup \{ t\}$ is a vertex of $B_2$ precisely when $r\not=\emptyset$. Point (i) follows. 

Now we argue for point (ii) under the assumption $t\in A$. Set 
\[
u= r\cup \{ t\},\; v= s\cup \{ t\}, \hbox{ and }w= r\cup s.
\]
We split our argument into two cases: $w\not=t$ and $w= t$. 

Assume first $w\not= t$. Define 
\[
f\colon {\rm Vr}(A)\cup \{ w, t, u\}\to {\rm Vr}(A)\cup \{ w, t, v\}
\]
by letting 
\[
f\res \big( {\rm Vr}(A)\cup \{ w\} \big) = {\rm id},\; f(t)=v,\; f(u)=t. 
\]
We need to see that, for $y\subseteq {\rm Vr}(A)\cup \{ w, t, u\}$,  
\begin{equation}\label{E:want2}
y\in u\, w\, tA\Leftrightarrow f(y)\in v\, w\, tA.
\end{equation}

We consider the complex $u\, w\, tA$. Let $y\subseteq {\rm Vr}(A)\cup \{ w,\, t,\,  u\}$, and 
set $y^0 = y\cap {\rm Vr}(A)$. By Lemma~\ref{L:calc4}(A) (with $r, s=w, t$) using the assumed relationships $w\not= t$ (which implies that $w$ is not a vertex of $tA$) and $w \subseteq t$, we obtain the following points with (i) of Lemma~\ref{L:calc4}(A) split into ($u$-i) and ($u$-ii) and (ii) of Lemma~\ref{L:calc4}(A) split into ($u$-iii) and ($u$-iv).

\begin{enumerate}
\item[($u$-i)]  If $u\in y$ and $t\in y$, then $y\in u\, w\, tA$ if and only if 
\begin{equation}\notag
t\cup y^0 \in A,\;  s\not\subseteq y^0,\; r \not\subseteq y^0,\;
(t\setminus w)\cup \{ w\} \not\subseteq y. 
\end{equation} 

\item[($u$-ii)]  If $u\in y$ and $t\not\in y$, then $y\in u\, w\, tA$ if and only if 
\begin{equation}\notag
t\cup y^0 \in A,\; s\not\subseteq y^0,\; 
(t\setminus w)\cup \{ w\} \not\subseteq y. 
\end{equation} 

\item[($u$-iii)]  If $u\not\in y$ and $t\in y$, then $y\in u\, w\, tA$ if and only if 
\begin{equation}\notag
t\cup y^0 \in A,\; r\not\subseteq y^0,\; (t\setminus w)\cup \{ w\} \not\subseteq y. 
\end{equation} 

\item[($u$-iv)]  If $u\not\in y$ and $t\not\in y$, then $y\in u\, w\, tA$ if and only if 
\begin{equation}\notag
\begin{cases}
w\cup y^0  \in A,\; (t\setminus w)\not\subseteq y^0,\; w\not\subseteq y^0, &\hbox{ if } w\in y,\\
y^0 \in A,\;  w\not\subseteq y^0, &\hbox{ if } w\not\in y.
\end{cases}
\end{equation} 
\end{enumerate}

The same analysis for the complex $v\, w\, tA$, shows that, for a subset $z$ of 
${\rm Vr}(A)\cup \{ w,\, t,\, v\}$, if we let $z^0 = z\cap {\rm Vr}(A)$, we have the following statements. 

\begin{enumerate}
\item[($v$-i)]  If $v\in z$ and $t\in z$, then $z\in v\, w\, tA$ if and only if 
\begin{equation}\notag
t\cup z^0 \in A,\; s\not\subseteq z^0,\; r \not\subseteq z^0,\; 
(t\setminus w)\cup \{ w\} \not\subseteq z. 
\end{equation} 

\item[($v$-ii)]  If $v\in z$ and $t\not\in z$, then $z\in v\, w\, tA$ if and only if 
\begin{equation}\notag
t\cup z^0 \in A,\; r\not\subseteq z^0,\; 
(t\setminus w)\cup \{ w\} \not\subseteq z. 
\end{equation} 

\item[($v$-iii)]  If $v\not\in z$ and $t\in z$, then $z\in v\, w\, tA$ if and only if 
\begin{equation}\notag
t\cup z^0 \in A,\; s\not\subseteq z^0,\; (t\setminus w)\cup \{ w\} \not\subseteq z. 
\end{equation} 

\item[($v$-iv)]  If $v\not\in z$ and $t\not\in z$, then $z\in v\, w\, tA$ if and only if 
\begin{equation}\notag
\begin{cases}
w\cup z^0  \in A,\; (t\setminus w)\not\subseteq z^0,\; w\not\subseteq z^0, &\hbox{ if } w\in z,\\
z^0 \in A,\;  w\not\subseteq z^0, &\hbox{ if } w\not\in z.
\end{cases}
\end{equation} 
\end{enumerate}

Substituting $z=f(y)$ in ($v$-i)--($v$-iv) and using the definition of $f$, we get the following statements.

\begin{enumerate}
\item[($v$-v)]  If $t\in y$ and $u\in y$, then $f(y)\in v\, w\, tA$ if and only if 
\begin{equation}\notag
t\cup y^0 \in A,\;  s\not\subseteq y^0,\; r \not\subseteq y^0,\;
(t\setminus w)\cup \{ w\} \not\subseteq y. 
\end{equation} 

\item[($v$-vi)]  If $t\in y$ and $u\not\in y$, then $f(y)\in v\, w\, tA$ if and only if 
\begin{equation}\notag
t\cup y^0 \in A,\; r\not\subseteq y^0,\; 
(t\setminus w)\cup \{ w\} \not\subseteq y. 
\end{equation} 

\item[($v$-vii)]   If $t\not\in y$ and $u\in y$, then $f(y)\in v\, w\, tA$ if and only if 
\begin{equation}\notag
t\cup y^0 \in A,\; s\not\subseteq y^0,\; (t\setminus w)\cup \{ w\} \not\subseteq y. 
\end{equation} 

\item[($v$-viii)]   If $t\not\in y$ and $u\not\in y$, then $f(y)\in v\, w\, tA$ if and only if 
\begin{equation}\notag
\begin{cases}
w\cup y^0  \in A,\; (t\setminus w)\not\subseteq y^0,\; w\not\subseteq y^0, &\hbox{ if } w\in y,\\
y^0 \in A,\;  w\not\subseteq y^0, &\hbox{ if } w\not\in y.
\end{cases}
\end{equation} 
\end{enumerate}

We note that the conditions in the pairs (($u$-i) and ($v$-v)),  (($u$-ii) and ($v$-vii)), (($u$-iii) and ($v$-vi)),  (($u$-iv) and ($v$-viii))
are identical, which proves \eqref{E:want2}. 

Assume now $w=t$, and recall that $t$ is a face of $A$. We keep in mind that $t, u, v\not\in {\rm Vr}(A)$ (since $t\not\in {\rm tc}\big( {\rm Vr}(A)\big)$. 
We have 
\[
B_1 = u\,t\,A\;\hbox{ and }\; B_2=v\,t\,A. 
\]
We need to show that for the function 
\[
f\colon {\rm Vr}(A)\cup \{ t, u\}\to {\rm Vr}(A)\cup \{ t, v\}
\]
defined by letting 
\[
f\res {\rm Vr}(A) = {\rm id},\; f(t)=v,\; f(u)=t, 
\]
we have that, for $y\subseteq {\rm Vr}(A)\cup \{ t, u\}$,  
\begin{equation}\label{E:want3} 
y\in u\, tA\Leftrightarrow f(y)\in v\, tA.
\end{equation} 
We use Lemma~\ref{L:calc4}(B) (substituting $w$ for $s=t$). For $y\subseteq {\rm Vr}(A)\cup \{ u, t\}$ and $z\subseteq {\rm Vr}(A)\cup \{ v, t\}$ after setting $y^0 = y\setminus \{ u, t\}$ and $z^0= z\setminus \{ v,t\}$, and noting that $t\setminus r =s$, 
we get 
\begin{equation}\notag
y\in u\, t A \Longleftrightarrow 
\begin{cases}
t\cup  y^0  \in A,\; s\not\subseteq y^0,\, r\not\subseteq y^0&\hbox{ if } u\in y,\, t\in y,\\
t\cup  y^0  \in A,\; s\not\subseteq y^0, &\hbox{ if } u\in y,\, t\not\in y,\\
t\cup  y^0  \in A,\; r\not\subseteq y^0, &\hbox{ if } u\not\in y,\, t\in y,\\
y^0 \in A,\; t\not\subseteq y^0, &\hbox{ if } u\not\in y,\, t\not\in y. 
\end{cases}
\end{equation} 
and 
\begin{equation}\label{E:spare} 
z\in v\, t A \Longleftrightarrow 
\begin{cases}
t\cup  z^0  \in A,\; s\not\subseteq z^0,\, r\not\subseteq z^0&\hbox{ if } v\in y,\, t\in y,\\
t\cup  z^0  \in A,\; r\not\subseteq z^0, &\hbox{ if } v\in y,\, t\not\in y,\\
t\cup  z^0  \in A,\; s\not\subseteq z^0, &\hbox{ if } v\not\in z,\, t\in z,\\
z^0 \in A,\; t\not\subseteq z^0, &\hbox{ if } v\not\in z,\, t\not\in z. 
\end{cases}
\end{equation} 
We finish the proof as in the case $w\not=t$ by substituting $f(y)$ for $z$ in \eqref{E:spare} and using the definition of $f$. 
\end{proof}

\begin{proof}[Proof of Theorem~\ref{T:ord}] 
We note 
\[
{\rm sp}\big((s\setminus t)\cup \{ t \}\big)  = {\rm sp}\big( (t\setminus s)\cup \{ s \}\big). 
\]
It follows that we only need to see that the assignment in (ii) gives a bijective function between ${\rm Vr}(B_1)$ and ${\rm Vr}(B_2)$ such that the function and its inverse map faces to faces. 

First we consider the case $s=t$. Then $B_1=\{ t\}\, t\,A=B_2$ 
and the assignment in (ii) is $\{ t\}\to \{ t\}$, which determines the identity map. So, (ii) follows. When $s=t$, point (i) reads 
\[
\{ t\}\hbox{ is a vertex of }\{ t\}\,t\,A\; \Leftrightarrow\; t\hbox{ is a face of }A. 
\]
The implication $\Leftarrow$ is clear from the definition of division. As for the implication $\Rightarrow$ note that if $t$ is not a face of $A$, then $\{ t\}\, t\,A=\{ t\}\,A$ and 
$\{ t\}$ is a vertex of $\{ t\}\,A$ only if $\{ t\}$ is a face of $A$ or $\{ t\}$ is a vertex of $A$. Both these possibilities are excluded by the assumption $t\not\in {\rm tc}\big( {\rm Vr}(A)\big)$. Thus, (ii) follows.

From this point on, we assume that $s\not= t$. 

We start with proving (i). Let $u= (s\setminus t)\cup \{ t\}$. Since $u\not = s$ (as $t\not\in s$),  $u\not= t$ (as $t\not\in s$), and $u$ is not a vertex of $A$ (as $t\not\in {\rm tc}\big( {\rm Vr}(A)\big)$), we have that $u$ is a vertex of $B_1$ precisely when $u$ is a face of $s\,t\,A$. We apply Lemma~\ref{L:calc2} with $x=u$. Note that $s$ is not a vertex of $tA$ since $s\not= t$ and $s$ is assumed not to be a vertex of $A$, so the lemma can be applied. Note further that $t\in u$ and $s\not\in u$ since $s\not= t$ and $s\not\in s$. It follows that $u\in s\,t\,A$ if and only if $(s\setminus t)\cup t\in A$, $t\not\subseteq s\setminus t$, and $s\not\subseteq s\setminus t$, which is equivalent to $s\cup t\in A$ and $s\cap t\not=\emptyset$. By an analogous argument we see that $(t\setminus s)\cup \{ s\}$ is a vertex of $B_2$ if and only if $s\cup t\in A$ and $s\cap t\not=\emptyset$, and (i) is proved.

Now we show (ii). Note that by point (i), if $s\cup t$ is not a face of $A$ or $s\cap t=\emptyset$, then $B_1=stA=tsA=B_2$ and the assignment gives the identity map, so point (ii) holds. 
We therefore assume that $s\cup t\in A$ and $s\cap t\not=\emptyset$. The proof is based on an application of Lemma~\ref{L:calc4} with $r=s\setminus t$. 
We consider the two complexes 
\begin{equation}\notag
u\, s\,  t A\;\hbox{ and }\;v \, t\, s A, 
\end{equation} 
where 
\[
u=(s\setminus t)\cup \{ t \}\;\hbox{ and }\; v=(t\setminus s)\cup \{ s \}. 
\]
Note that 
\begin{equation}\label{E:ustn} 
u\not= s, t\hbox{ and }v\not= s,t
\end{equation}
since $s\not\in t$ and $t\not\in s$. Also $u,v\not \in {\rm Vr}(A)$ since $s, t\not\in {\rm tc}\big( {\rm Vr}(A)\big)$. So, we can define a function $f\colon {\rm Vr}(A)\cup \{ s, t, u\} \to {\rm Vr}(A)\cup \{ s, t, v\}$ by letting $f\res \big( {\rm Vr}(A)\cup \{ s, t\}\big)$ be the identity and $f(u)=v$. 
The function $f$ is a bijection and, for $y\subseteq {\rm Vr}(A)\cup \{ s, t, u\}$, we have 
\[
{\rm sp}\big(f(y)\big)= {\rm sp}(y). 
\]
So, we only need to show that 
\begin{equation}\label{E:want} 
y \in u \, s\, tA \Leftrightarrow f(y) \in v\, t\, sA. 
\end{equation} 

Observe first that directly from the definition, we have, for $y\subseteq {\rm Vr}(A)\cup \{ s, t, u\}$, 
\begin{equation}\label{E:lin}
\begin{split}
f(y) \cap \big(  {\rm Vr}(A)\cup \{ s,  t \}\big) & = y \cap  \big( {\rm Vr}(A)\cup \{ s, t \} \big)\\
v \in f(y) &\Leftrightarrow u \in y. 
\end{split}
\end{equation}
Consider the complex $u \, s\, tA$. Set $y^0= y\setminus \{ s, t, u\}$. 
We apply Lemma~\ref{L:calc4} with $r=s\setminus t$. By Lemma~\ref{L:calc4}(Ai), 
we get the following statement.
\begin{enumerate}
\item[(a)] If $u\in y$, then 
$y\in u\, s\, tA$ if and only if 
\begin{equation}\label{E:case1} 
s\cup t\cup y^0 \in A,\; (s\setminus t)\cup \{ t\}  \not\subseteq y,\; (s\cap t)\not\subseteq y^0,\; 
(t\setminus s)\cup \{ s\} \not\subseteq y. 
\end{equation} 
\end{enumerate} 
Whereas by Lemma~\ref{L:calc4}(Aii), we get the following statement.

\begin{enumerate}
\item[(b)] If $u\not\in y$, then 
$y\in u\, s\, tA$ if and only if 
\begin{equation}\label{E:case2} 
\begin{cases}
s\cup t\cup y^0 \in A,\; (t\setminus s)\not\subseteq y^0,\; (s\setminus t)\not\subseteq y^0, &\hbox{ if } s\in y,\, t\in y,\\ 
s\cup y^0  \in A,\; (t\setminus s)\not\subseteq y^0,\; s\not\subseteq y^0, &\hbox{ if } s\in y,\, t\not\in y,\\
t\cup  y^0  \in A,\; t\not\subseteq y^0,\; (s\setminus t)\not\subseteq y^0, &\hbox{ if } s\not\in y,\, t\in y,\\
y^0 \in A,\; t\not\subseteq y^0,\; s\not\subseteq y^0, &\hbox{ if } s\not\in y,\, t\not\in y.  
\end{cases}
\end{equation} 
\end{enumerate} 

We work under the assumption of (a). By the second line of \eqref{E:lin}, we have $v\in f(y)$. By $v\in f(y)$, 
the same application of Lemma~\ref{L:calc4}(Ai) as the one we used to get \eqref{E:case1} shows that 
the condition $f(y)\in v \, t\, sA$ is equivalent to \eqref{E:case1} with $s$ and $t$ switched and with $y$ replaced by $f(y)$. Note that 
\eqref{E:case1} is invariant under switching of $s$ and $t$. Thus, by the first line of \eqref{E:lin}, we get that $f(y)\in v \, t\, sA$ is equivalent to \eqref{E:case1}. Thus, \eqref{E:want} follows. 

We work now under the assumption of (b). By the second line of \eqref{E:lin}, we have $v\not\in f(y)$. Again, by $v\not\in f(y)$,  Lemma~\ref{L:calc4}(Aii) implies that 
the condition $f(y)\in v \, t\, sA$ is equivalent to \eqref{E:case2} with $s$ and $t$ switched and with $y$ replaced by $f(y)$. Since 
\eqref{E:case2} is invariant under switching of $s$ and $t$, by the first line of \eqref{E:lin}, we have that 
$f(y)\in v \, t\, sA$ is equivalent to \eqref{E:case2}. Again, \eqref{E:want} follows. 
\end{proof}

\end{document}